\documentclass[11pt]{amsart}
\usepackage{amssymb}
\usepackage{amsfonts}
\usepackage{amsmath,latexsym,amssymb,amsfonts,amsbsy, amsthm}
\usepackage[usenames]{color}
\usepackage{xspace,colortbl}
\usepackage{mathrsfs}
\usepackage[colorlinks,linkcolor=blue,citecolor=blue]{hyperref}
\usepackage[titletoc]{appendix}
\allowdisplaybreaks

\newcommand{\beq}{\begin{equation}}
\newcommand{\eeq}{\end{equation}}
\newcommand{\ben}{\begin{eqnarray}}
\newcommand{\een}{\end{eqnarray}}
\newcommand{\beno}{\begin{eqnarray*}}
\newcommand{\eeno}{\end{eqnarray*}}
\newcommand{\R}{\mathbb{R}}
\newtheorem{thm}{Theorem}[section]

\newtheorem{lem}[thm]{Lemma}
\newtheorem{prop}[thm]{Proposition}
\newtheorem{coro}[thm]{Corollary}
\newtheorem{rmk}[thm]{Remark}

\begin{document}

\title[Second order estimate]{Second order estimate on transition layers$^*$}
\author[K. Wang and J. Wei]{Kelei Wang$^\dag$ and Juncheng Wei$^\S$}
\thanks{$^\dag$School of Mathematics and Statistics \& Computational
Science Hubei Key Laboratory, Wuhan University, Wuhan 430072, China.
{Email: wangkelei@whu.edu.cn}. }
\thanks{$^\S$Department of Mathematics, University of British
Columbia, Vancouver, B.C., Canada, V6T 1Z2.
{Email: jcwei@math.ubc.ca}.}

\thanks{$*$  The research of K. Wang was supported
by  NSFC no. 11871381 and no. 11631011. J. Wei is partially supported by NSERC of Canada. }
\date{\today}

\begin{abstract}
In this paper we establish a uniform $C^{2,\theta}$ estimate for level sets of stable solutions to the singularly perturbed Allen-Cahn equation in dimensions $ n\leq 10$ (which is optimal). The proof combines two ingredients: one is the infinite dimensional reduction method which enables us to reduce the $C^{2,\theta}$ estimate for these level sets to a corresponding one on  solutions of Toda system; the other one uses a small regularity theorem on stable solutions of Toda system to establish various decay estimates on these solutions, which gives a lower bound on distances between different sheets of solutions to Toda system or level sets of solutions to Allen-Cahn equation.
\end{abstract}


\maketitle
\renewcommand{\theequation}{\thesection.\arabic{equation}}
\setcounter{equation}{0}
%%%%%%%%%%%%%%%%%%%%%%%%%%%%%%%%%%%%%%%%%%%%%%

\tableofcontents

\section{Introduction}\label{sec introduction}
\setcounter{equation}{0}

\subsection{Main result} In this paper, continuing the study in \cite{Wang-Wei2}, we establish a second order estimate on level sets of stable solutions to the singularly perturbed Allen-Cahn equation
\begin{equation}\label{singularly perturbed eqn}
\varepsilon\Delta u_\varepsilon=\frac{1}{\varepsilon}W^\prime(u_\varepsilon), \quad |u_\varepsilon|<1 \quad \mbox{in } B_1(0)\subset\R^n.
\end{equation}
Here $W(u)$ is a general double well potential, that is, $W\in C^3([-1,1])$ satisfying
\begin{itemize}
\item $W>0$ in $(-1,1)$ and $W(\pm1)=0$;
\item  $W^\prime(\pm1)=0$ and $W^{\prime\prime}(-1)=W^{\prime\prime}(1)=1$; (Note a slight notation difference here with other literatures.)
\item there exists only one critical point of $W$ in $(-1,1)$, which we assume to be $0$.
\end{itemize}
A typical model is given by $W(u)=(1-u^2)^2/8$.

Under these assumptions on $W$, it is known that there exists a unique solution  to the following one dimensional problem
\begin{equation}\label{1d problem}
g^{\prime\prime}(t)=W^\prime(g(t)),  \quad \ g(0)=0 \ \ \quad \mbox{and }\ \lim_{t\to\pm\infty}g(t)=\pm 1.
\end{equation}

A solution of \eqref{singularly perturbed eqn} is stable if for any $\eta\in C_0^\infty(B_1(0))$,
\begin{equation}\label{stability condition}
 \int_{B_1(0)}\left[\varepsilon^2|\nabla\eta|^2+W^{\prime\prime}(u_\varepsilon)\eta^2\right]\geq0.
\end{equation}

By Sternberg-Zumbrun \cite{S-Z}, the stability condition is equivalent to
\begin{equation}\label{stability condition}
 \int_{B_1(0)}|\nabla\eta|^2\varepsilon|\nabla u_\varepsilon|^2\geq \int_{B_1(0)}\eta^2|B(u_\varepsilon)|^2\varepsilon|\nabla u_\varepsilon|^2, \quad \forall \eta\in C_0^\infty(B_1(0)).
\end{equation}
Here
\[
|B(u_\varepsilon)|^2=
\begin{cases}
  \frac{|\nabla^2u_\varepsilon|^2-|\nabla|\nabla u_\varepsilon||^2}{|\nabla u_\varepsilon|^2}, & \mbox{if }|\nabla u_\varepsilon|\neq 0 \\
  0, & \mbox{otherwise}.
\end{cases}
\]

It is known that if $|\nabla u_\varepsilon(x)|\neq 0$, it holds that
\begin{equation}\label{curvature term}
 |B(u_\varepsilon)(x)|^2=|A_\varepsilon(x)|^2+|\nabla_T\log|\nabla u_\varepsilon(x)||^2,
\end{equation}
where $A_\varepsilon(x)$ is the second fundamental form of the level set $\{u_\varepsilon=u_\varepsilon(x)\}$ and $\nabla_T$ denotes the tangential derivative along the level set $\{u_\varepsilon=u_\varepsilon(x)\}$.

The main result of this paper is
\begin{thm}\label{main result}
For any $\theta\in(0,1)$, $0<b_1\leq b_2<1$ and $\Lambda>0$, there exist two constants $C=C(\theta,b_1,b_2,\Lambda)$ and $\varepsilon_\ast=\varepsilon(\theta,b_1,b_2,\Lambda)$ so that the following holds.
Suppose $u_\varepsilon$ is a   stable solution of \eqref{singularly perturbed eqn} in $B_1(0)\subset\R^n$ satisfying
\begin{equation}\label{C 1 1 bound}
|\nabla u_\varepsilon|\neq 0 \quad \mbox{and} \quad |B(u_\varepsilon)|\leq \Lambda, \quad \mbox {in } \{|u_\varepsilon|\leq 1-b_2\}\cap B_1(0).
\end{equation}
If $n\leq 10$ and $\varepsilon\leq\varepsilon_\ast$, then for any $t\in[-1+b_1,1-b_1]$, $\{u_\varepsilon=t\}$ are smooth hypersurfaces and the $C^\theta$ norm of their second fundamental forms are  bounded by $C$. Moreover,
\begin{equation}\label{vansihing of mean curvature}
|H(u_\varepsilon)|\leq C\varepsilon \left(\log|\log\varepsilon|\right)^2,
\end{equation}
where $H(u_\varepsilon)$ denotes the mean curvature of $\{u_\varepsilon=t\}$.
\end{thm}

Several corollaries follow from this theorem.
\begin{coro}\label{main coro 1}
For any $\theta\in(0,1)$, $b\in(0,1)$ and $Q>0$, there exist two  constants $\varepsilon_1$  and $C_1$ so that the following holds.
Suppose that $u_\varepsilon$ is a sequence of stable solution of \eqref{singularly perturbed eqn} in $\mathcal{C}_1:=B_1^{n-1}(0)\times(-1,1)\subset\R^n$ with $\varepsilon \to0$, satisfying
\begin{itemize}
\item [{\bf (H1)}] there exists  a sequence of $t_\varepsilon\in(-1+b,1-b)$ such that  $\{u_\varepsilon=t_\varepsilon\}$
consists of $Q$ connected components
\[\Gamma_{\alpha,\varepsilon}=\left\{x_n=f_{\alpha,\varepsilon}(x^\prime), \quad x^\prime:=(x_1,\cdots, x_{n-1})\in B_1^{n-1}\right\}, \ \ \ \alpha=1,\cdots, Q,\]
where $-1/2<f_{1,\varepsilon}<f_{2,\varepsilon}<\cdots< f_{Q,\varepsilon}<1/2$;
\item  [{\bf (H2)}] for each $\alpha$, $\nabla f_{\alpha,\varepsilon}$ are uniformly continuous in $B_1^{n-1}(0)$.
\end{itemize}
If $n\leq 10$, then the same conclusion of Theorem \ref{main result} holds for all $u_\varepsilon$ if $\varepsilon\leq \varepsilon_1$, with $C$ replaced by   $C_1$.
\end{coro}

\begin{coro}\label{main coro 2}
For any $\theta\in(0,1)$ and $b\in(0,1)$, there exist three  constants $\varepsilon_2$, $\delta_2$ and $C_2$ so that the following holds.
Suppose that $u_\varepsilon$ is a sequence of stable solutions of \eqref{singularly perturbed eqn} in $B_1(0)\subset\R^n$, satisfying for any $x\in \{u_\varepsilon=0\}\cap B_{1-\varepsilon}(0)$,
\begin{equation}\label{close to 1D}
\sup_{y\in B_\varepsilon(x)}\big|u_\varepsilon(y)-g\left(\frac{y\cdot e-t}{\varepsilon}\right)\big|\leq \delta_2,
\end{equation}
where $e$ is a unit vector and $t$ is a constant, both depending on $x$.
If \emph{Stable Bernstein Conjecture} is true in dimension $n$, then the same conclusion of Corollary \ref{main coro 1} holds with $\varepsilon_1,C_1$ replaced by $\varepsilon_2,C_2$.
\end{coro}
Note that \emph{Stable Bernstein Conjecture} is expected to be true for $n\leq 7$ and it has been verified for $n=3$, see do Carmo and Peng \cite{Carmo-Peng}, Fischer-Colbrie  and Schoen \cite{F-Schoen} and Pogorelov \cite{Pogorelov}. For regularity theory of stable minimal surfaces in higher dimensions we refer to Wickramasekera \cite{Wick}.

\medskip

Some remarks are in order.
\begin{rmk}
\begin{itemize}
  \item The $n=2$ case is essentially contained in our paper \cite{Wang-Wei2}. Recently  Chodosh  and Mantoulidis established this second order regularity result for the  $n=3$ case, which was used in their analysis of Allen-Cahn approximation to minimal surfaces in three dimensional manifold, see \cite{chodosh2018minimal}.
       The relations between the number of ends and Morse index is discussed in Mantoulidis \cite{Mantoulidis}.

  \item The assumption \eqref{close to 1D} says $u_\varepsilon$ is close to the one dimensional solution $g$ at $O(\varepsilon)$ scales. This is guaranteed  by assumptions {\bf (H1)} and {\bf (H2)}, see \cite{Wang-Wei2} as well as Section \ref{sec completion of proof}.

  \item The dimension bound $n\leq 10$ is sharp. If $n\geq 11$, there exists  a smooth, radially symmetric, stable solution to  the Liouville equation (i.e. two component Toda system)
  \[\Delta f=e^{-f}, \quad \mbox{in } \mathbb{R}^{n-1}.\]
Agudelo-del Pino-Wei \cite{ADW2} constructed a  family of solutions $u_\varepsilon$ of \eqref{singularly perturbed eqn} in $\R^n$, with its nodal set $\{u_\varepsilon=0\}$ given by the graph $\{x_n=\pm f_\varepsilon(x^\prime)\}$, $x^\prime\in\R^{n-1}$, where
\[f_\varepsilon(x)\approx   \varepsilon f \left(\varepsilon^{-\frac{1}{2}}x\right)+ \varepsilon|\log\varepsilon|.\]
Clearly we have
\[|\nabla^2f_\varepsilon(0)|\approx|\nabla^2 f(0)|\]
while
\[\lim_{\varepsilon\to0}|\nabla^2f_\varepsilon(x)|=0, \quad \forall x\neq 0.\]
Hence $\nabla^2f_\varepsilon$ is not uniformly continuous.

  \item The stability condition is also necessary for this second order regularity. Without the stability condition it is not true even for $n=2$. Counterexamples are provided by the multiple end solutions of \eqref{singularly perturbed eqn} in $\R^2$, constructed by  Del Pino-Kowalczyk-Pacard-Wei \cite{DKW3}.  By utilizing solutions of Toda system
  \[f_\alpha^{\prime\prime}=e^{-(f_\alpha-f_{\alpha-1})}-e^{-f_{\alpha+1}-f_\alpha}, \quad \mbox{on } \R, \quad 1\leq\alpha\leq Q,\]
  they constructed  a  family of solutions $u_\varepsilon$ of \eqref{singularly perturbed eqn} in $\R^2$, with its nodal set $\{u_\varepsilon=0\}$ given by the graph of
\[f_{\alpha,\varepsilon}(x)\approx   \varepsilon f_\alpha \left(\varepsilon^{-\frac{1}{2}}x\right)+ \alpha\varepsilon|\log\varepsilon|.\]
As in the previous case, $\nabla^2f_\varepsilon$ is not uniformly continuous.

\item Although this second order regularity does not hold any more for $n\geq 11$. A partial regularity result may still hold. For example, under assumptions of Corollary \ref{main coro 1}, there should exist a closed set of Hausdorff dimension at most $n-10$ such that in any compact set outside this singular set, uniform second order regularity of level sets for stable solutions of \eqref{equation} still hold.

\item It seems that the second order regularity problem is quite different  in nature from the first order regularity problem, i.e. uniform $C^{1,\theta}$ estimates on level sets. See Caffarelli-Cordova \cite{CC} and Tonegawa-Wickramasekera \cite{Tonegawa2}. For example, it can be checked that the above counterexamples (\cite{ADW2} \cite{DKW3}) to second order regularity still enjoy a  uniform $C^{1,\theta}$ estimate.

\item We do not touch any aspect on higher order regularity (e.g. $C^{k,\theta}$ regularity for $k\geq 3$) of level sets. It will be interesting to obtain such a result even for the multiplicity one case.
\end{itemize}

\end{rmk}

\subsection{Outline of proof}\label{subsec outline of proof}
The proof of Theorem \ref{main result} consists of the following three steps.

\medskip

\noindent
{\bf Step 1. Infinite dimensional Lyapunov-Schmidt reduction.} Sections \ref{sec preliminary analysis}-\ref{sec improved estimate} are devoted to this reduction procedure. It is almost the same with the one in \cite{Wang-Wei2}, but here  various simplifications and improvements  will be given.

The main difference is that in \cite{Wang-Wei2}, it is either assumed that there are only finitely many connected components of transition layers (as in Corollary \ref{main coro 1}) or the distance between different connected components of transition layers has a lower bound in the form $c\varepsilon|\log\varepsilon|$ (see \cite[Section 17]{Wang-Wei2}), but now both assumptions are removed and we only need the assumption that the distance between different connected components of transition layers to be $\gg \varepsilon$ (see Lemma \ref{O(1) scale} below) as a starting point. Moreover, now we can show that all  estimates in this step hold uniformly with respect to the number of connected components of transition layers. Hence there is no assumption on the number of connected components of transition layers in Theorem \ref{main result} and Corollary \ref{main coro 2}.

The reduction method proceeds as follows. First from the assumptions in Theorem \ref{main result} (or Corollary \ref{main coro 1} or \ref{main coro 2}), it follows that the solution $u_\varepsilon$ is close to the one dimensional profile at $O(\varepsilon)$ scales, see Section \ref{sec preliminary analysis}. Therefore the solution has the form
\begin{equation}\label{1.1}
  u_\varepsilon=\sum_{\alpha}g_{\alpha,\varepsilon}+\phi_\varepsilon,
\end{equation}
where $g_{\alpha,\varepsilon}$ is the one dimensional solution in composition with the distance function to $\Gamma_{\alpha,\varepsilon}$, a connected component of $\{u_\varepsilon=0\}$, and $\phi_\varepsilon$ is a small error between our solution $u_\varepsilon$ and the approximate solution $\sum_{\alpha}g_{\alpha,\varepsilon}$.

Writing $u_\varepsilon$ in this way, the single equation for $u_\varepsilon$, \eqref{singularly perturbed eqn}, is almost decoupled into two equations: one is the equation for the level set $\{u_\varepsilon=0\}$ and the other one is an equation for $\phi_\varepsilon$. Such a decoupling is possible by choosing an optimal approximation in \eqref{1.1}, which then implies that $\phi_\varepsilon$ lies in the subspace orthogonal to the kernel space at $\sum_{\alpha}g_{\alpha,\varepsilon}$, see Proposition \ref{prop optimal approximation} for a precise statement. To this end, it is necessary to take a small perturbation in the normal direction of each $\Gamma_{\alpha,\varepsilon}$ so that $g_{\alpha,\varepsilon}$ is the optimal approximation to $u_\varepsilon$ in the normal direction. Here it is convenient to introduce Fermi coordinates with these $\Gamma_{\alpha,\varepsilon}$ and rewrite everything in these coordinates, see Section \ref{sec Fermi coordinates}-\ref{sec approximate solution}.

Since $\Gamma_{\alpha,\varepsilon}$ are far from each other and they are almost parallel,
the interaction pattern between different $g_{\alpha,\varepsilon}$, which represents the interaction between different components of $\{u_\varepsilon=0\}$, can be determined by using asymptotic expansions of the one dimensional profile at infinity. This gives the equation for $\Gamma_{\alpha,\varepsilon}$,
\begin{equation}\label{1.2}
H_{\alpha,\varepsilon}=\frac{2A_{(-1)^{\alpha-1}}^2}{\varepsilon}e^{-\frac{|d_{\alpha-1,\varepsilon}|}{\varepsilon}}-
\frac{2A_{(-1)^\alpha}^2}{\varepsilon}e^{-\frac{|d_{\alpha+1,\varepsilon}|}{\varepsilon}}+\mbox{higher order terms},
\end{equation}
where $H_{\alpha,\varepsilon}$ is the mean curvature of $\Gamma_{\alpha,\varepsilon}$, $|d_{\alpha-1,\varepsilon}|$ and $|d_{\alpha+1,\varepsilon}|$ are distances to $\Gamma_{\alpha-1,\varepsilon}$ and $\Gamma_{\alpha+1,\varepsilon}$ respectively, see Section \ref{sec Toda system} for a precise statement.

Higher order terms in \eqref{1.2} involve some terms containing $\phi_\varepsilon$.
In order to get a good reduced problem, a precise estimate on $\phi_\varepsilon$ is needed. This is established in Section \ref{sec first order} and Section \ref{sec improved estimate}. Since $\phi_\varepsilon$ is known to be a small perturbation, it satisfies an almost linearized equation. (This is the reduction procedure, i.e. we partially linearize \eqref{singularly perturbed eqn} in the $\phi_\varepsilon$ component.) To estimate $\phi_\varepsilon$, we need to consider two separate cases: the inner problem near $\{u_\varepsilon=0\}$, and the outer one which is concerned with the part far away from $\{u_\varepsilon=0\}$. It is important here that these two parts are still almost decoupled, which is guaranteed by the fast decay of the one dimensional profile at infinity.

\medskip

\noindent
 {\bf Step 2. Reduction of the stability condition.} Now the $C^{2,\theta}$ estimate is reduced to a corresponding one on \eqref{1.2}. It turns out that this depends in an essential way on lower bounds on $|d_{\alpha-1,\varepsilon}|$ and $|d_{\alpha+1,\varepsilon}|$, as observed in \cite{Wang-Wei2}. To get these lower bounds, we use the stability condition \eqref{stability condition}. In Section \ref{sec reduction of stability}, we show that if $u_\varepsilon$ is a stable solution, then solutions to the reduced  problem \eqref{1.2} satisfies an almost stability condition. This is achieved by choosing test functions in \eqref{stability condition} to be
 \[\sum_{\alpha}\eta_\alpha g_{\alpha,\varepsilon}^\prime,\]
 where $\eta_\alpha\in C_0^\infty(\Gamma_{\alpha,\varepsilon})$. In other words, we consider variations along directions tangential to $\{u_\varepsilon=0\}$.  This choice of test functions in the stability condition is similar to the one used in \cite{ADW2} and \cite[Appendix D]{chodosh2018minimal}. Then by a careful analysis of contributions from tangential parts, normal parts, cross terms and the interaction between different components, we get a stability condition on solutions to \eqref{1.2}, see Proposition \ref{prop reduction of stability}.

\medskip

\noindent
  {\bf Step 3. Decay estimates.} Finally, a small regularity theorem on stable solutions of \eqref{1.2} will be employed to give  decay estimates on $e^{- |d_{\alpha-1,\varepsilon}|/\varepsilon}$ in the interior, which then leads to a $C^{2,\theta}$ estimate on \eqref{1.2}.

   This small regularity result has been established by the first author in \cite{Wangstable1, Wangstable2} in the setting of stable solutions for the Liouville equation and it can be generalized to Toda system \eqref{1.2}. Here the dimension restriction $n\leq 10$ appears, due to the fact that this small regularity theorem requires an $L^1$ smallness assumption on $e^{- |d_{\alpha-1,\varepsilon}|/\varepsilon}$ as the starting point. This $L^1$ smallness condition holds unconditionally only in $n\leq 10$, which can be proved by an $L^p$ estimate of Farina \cite{Farina}.

A reduction procedure is still needed in order to apply this small regularity theorem to \eqref{1.2}. In this paper, two of such approaches are employed. The first one is  extrinsic and uses the graph representation (with respect to a fixed hyperplane) of $\Gamma_{\alpha,\varepsilon}$. This works well when they are very close, which implies that different $\Gamma_{\alpha,\varepsilon}$ are almost parallel to each other. This then allows us to represent distances between them by differences of functions, and replace the minimal surface operator in \eqref{1.2} by the standard Laplacian operator

The second reduction method is intrinsic and uses the Jacobi field construction introduced in Chodosh  and Mantoulidis \cite{chodosh2018minimal}. Here we fix an $\Gamma_{\alpha,\varepsilon}$ and view other components as graphs of functions defined on this component. Combined with some elliptic estimates on these functions, this approach  gives a stronger distance lower bound. This then implies that the exponential nonlinearity in \eqref{1.2} is dominated by mean curvature terms. Using this we can construct positive Jacobi fields as in \cite{chodosh2018minimal}.

In this paper, a constant is called universal if it depends only on the dimension $n$, the double well potential $W$ and the constants $b_1,b_2,\Lambda$ in Theorem \ref{main result}. If $A\leq CB$ for a universal constant $C$, then we denote it by $A\lesssim B$ or $A=O(B)$. If the constant $C$ depends on a parameter $K$, it is written as $A=O_K(B)$.

By letting  $u(x):=u_\varepsilon(\varepsilon x)$, we obtain the unscaled Allen-Cahn equation
\begin{equation}\label{equation}
  \Delta u=W^\prime(u).
\end{equation}

\section{Preliminary analysis}\label{sec preliminary analysis}
\setcounter{equation}{0}

In the following we will only be concerned with one level set of $u_\varepsilon$, $\{u_\varepsilon=0\}$.
It will be clear that our proof goes through without any change when $0$ is replaced by any other $t\in[-1+b_1,1-b_1]$, and all of the following estimates are uniform in $t\in[-1+b_1,1-b_1]$.

By standard elliptic regularity theory, $u_\varepsilon\in C^4_{loc}(B_1(0))$. Concerning the regularity of $\{u_\varepsilon=0\}$, we first prove that different components of it are at least $O(\varepsilon)$ apart. In the following a connected component of $\{u_\varepsilon=0\}$ is denoted by $\Gamma_{\alpha,\varepsilon}$, where $\alpha$ is the index. The following lemma also  shows that the cardinality of the index set is always finite for fixed $\varepsilon$, although it could go to infinity as $\varepsilon\to0$.

\begin{lem}\label{O(1) scale}
For any $\alpha$ and $x_{\varepsilon}\in\Gamma_{\alpha,\varepsilon}\cap B_{7/8}(0)$, as $\varepsilon\to0$, $\tilde{u}_\varepsilon(x):=u_\varepsilon(x_{\varepsilon}+\varepsilon x)$ converges to a one dimensional solution in $C^2_{loc}(\R^n)$.
In particular,
\begin{equation}\label{2.1}
\varepsilon^{-1}\mbox{dist}\left(x_\varepsilon, \{u_\varepsilon=0\}\setminus \Gamma_{\alpha,\varepsilon}\right)\to+\infty \quad\mbox{ uniformly}.
\end{equation}
\end{lem}
\begin{proof}
In $B_{\varepsilon^{-1}/8}(0)$, $\tilde{u}_\varepsilon(x)$ satisfies the Allen-Cahn equation \eqref{equation}.
By standard elliptic regularity theory, $\tilde{u}_\varepsilon(x)$ is uniformly bounded in $C^4_{loc}(\R^n)$. By Arzela-Ascoli theorem, as $\varepsilon\to0$, it converges to a limit function $\tilde{u}_\infty$ in $C^2_{loc}(\R^n)$. Clearly $\tilde{u}_\infty$ is a stable solution of \eqref{equation} in $\R^n$.

Since $\tilde{u}_\varepsilon(0)=0$, $\tilde{u}_\infty(0)=0$. By our assumption on the double well potential $W$, the constant solution $0$ is not stable. Hence $\tilde{u}_\infty$ is a non-constant solution. As a consequence, by unique continuation principle, the critical set $\{\nabla\tilde{u}_\infty=0\}$ has zero Lebesgue measure.

By \eqref{C 1 1 bound},
\[\big|B\left(\tilde{u}_\varepsilon\right)\big|\leq\Lambda\varepsilon, \quad \mbox{in } \{|\tilde{u}_\varepsilon|\leq 1-b\}\cap B_{\varepsilon^{-1}/7}(0).\]
By the convergence of $\tilde{u}_\varepsilon$, we can pass this inequality to the limit in any compact set outside $\{\nabla\tilde{u}_\infty=0\}$, which leads to $\big|B\left(\tilde{u}_\infty\right)\big|\equiv 0$ in $\{|\tilde{u}_\infty|\leq 1-b\}$. Hence by \eqref{curvature term} and Sard theorem, almost all level sets $\{\tilde{u}_\infty=t\}$, $t\in[-1+b,1-b]$, are hyperplanes. Then it is directly verified that $\tilde{u}_\infty$ is one dimensional.

Let
$\widetilde{\Gamma}_{\alpha,\varepsilon}:=\varepsilon^{-1}\left(\Gamma_{\alpha,\varepsilon}-x_\varepsilon\right)$.
By the convergence of $\tilde{u}_{\alpha,\varepsilon}$, in any compact set of $\R^n$,  $\widetilde{\Gamma}_{\alpha,\varepsilon}$ converge to $\{\tilde{u}_\infty=0\}$ in the Hausdorff distance. Since $\{\tilde{u}_\infty=0\}$ is a single hyperplane, we get
\[ \mbox{dist}\left(0, \{\tilde{u}_\varepsilon=0\}\setminus \widetilde{\Gamma}_{\alpha,\varepsilon}\right)\to+\infty,\]
where the convergence rate depends only on $\varepsilon$. This gives \eqref{2.1}.
\end{proof}
The above proof implies that the Implicit Function Theorem can be applied to $u_\varepsilon$ at $O(\varepsilon)$ scales, which gives the $C^4$ regularity of $\{u_\varepsilon=0\}$. Of course it is not known whether there exists a uniform bound independent of $\varepsilon$.

The following lemma can be proved by combining the curvature bound \eqref{C 1 1 bound} with the fact that different connected components of $\{u_\varepsilon=u_\varepsilon(0)\}$ are disjoint. (This fact   has been used a lot in minimal surface theory, see for instance \cite{CM2}.)
\begin{lem}\label{lem graph construction}
There exist two universal constants $\sigma\in(0,1/16)$ and $C(\sigma,\Lambda)$ so that the following holds. For any $x_\ast\in \{|u_\varepsilon|\leq 1-b\}\cap B_{7/8}(0)$, in a suitable coordinate system,
$\{u_\varepsilon=u_\varepsilon(x_\ast)\}\cap B_\sigma(x_\ast)$ is a family of graphs $\cup_\alpha\{x_n=f_{\alpha,\varepsilon}(x^\prime)\}$,
 where $f_{\alpha,\varepsilon}\in C^4(B_{2\sigma}^{n-1}(x_\ast^\prime))$ satisfying $\|f_{\alpha,\varepsilon}\|_{C^{1,1}(B_{2\sigma}^{n-1}(x_\ast^\prime))}\leq C(\sigma,\Lambda)$.
\end{lem}

\section{Fermi coordinates}\label{sec Fermi coordinates}
\setcounter{equation}{0}

\subsection{Definition}
For simplicity of presentation, we now work in the stretched version and do not write the dependence on $\varepsilon$ explicitly.

By denoting $R=\varepsilon^{-1}$,
$u(x)=u_\varepsilon(\varepsilon x)$
satisfies the Allen-Cahn equation \eqref{equation} in $B_R(0)$. Its nodal set $\{u=0\}$ consists of finitely many connected components, $\Gamma_\alpha$.

By our assumption, for each $\alpha$, the second fundamental form $A_\alpha$ of $\Gamma_\alpha$ satisfies
\begin{equation}\label{bound on second fundamental form}
|A_\alpha(y)|\leq \Lambda\varepsilon, \quad \forall y\in \Gamma_\alpha\cap B_R(0).
\end{equation}
We will assume $\Lambda$ is sufficiently small, perhaps after restricting to a  small ball and then rescaling its radius to $1$.

Let $y$ be a local coordinates of $\Gamma_\alpha$.
The Fermi coordinate is defined as $(y,z)\mapsto x$, where
$x=y+zN_\alpha(y)$.
Here
$N_\alpha(y)$ is a unit normal vector to $\Gamma_\alpha$, $z$ is  the signed distance to $\Gamma_\alpha$.
By \eqref{bound on second fundamental form},  Fermi coordinates are well defined and smooth in $B_R(0)$.

By Lemma \ref{lem graph construction} (recall that we have assumed $\Lambda\ll 1$), after a rotation,
\begin{equation}\label{graph representation}
  \Gamma_\alpha\cap B_R(0)=\{x_n=f_\alpha(x^\prime)\}, \quad \forall \alpha.
\end{equation}
Therefore a canonical way to choose local coordinates of $\Gamma_\alpha$ is by letting $y=x^\prime$ for each $\alpha$. Then the induced metric on $\Gamma_\alpha$ is
\[g_{\alpha,ij}(y)=\delta_{ij}+\frac{\partial f_\alpha}{\partial y_i}(y)\frac{\partial f_\alpha}{\partial y_j}(y).\]
By Lemma \ref{lem graph construction} and \eqref{bound on second fundamental form}, we get a universal constant $C$ such that
\begin{equation}\label{derivatives bound}
  |\nabla f_\alpha|\leq C \quad \mbox{and} \quad |\nabla^2 f_\alpha|\leq C\varepsilon, \quad \mbox{in } B_R^{n-1}(0).
\end{equation}

Sometimes the singed distance to $\Gamma_\alpha$ is also denoted by $d_\alpha$.
Since $\Gamma_\alpha\cap\Gamma_\beta=\emptyset$ for any $\alpha\neq\beta$, we can choose the sign so that $\{d_\alpha>0\}\cap\{d_\beta>0\}\neq \emptyset$ for any $\alpha\neq\beta$.

For any $z\in(-R,R)$, let
$\Gamma_{\alpha,z}:=\left\{dist(x,\Gamma_\alpha)=z\right\}$. Hence $\Gamma_{\alpha,0}$ is just $\Gamma_\alpha$.
Define the vector field
\[X_i:=\frac{\partial}{\partial y^i}+z\frac{\partial N_\alpha}{\partial y^i}=\sum_{j=1}^{n-1}\left(\delta_{ij}-zA_{\alpha,ij}\right)\frac{\partial}{\partial y^j}, \quad 1\leq i\leq n-1.\]
The tangent space of $\Gamma_{\alpha,z}$ is spanned by $X_i$. The Euclidean metric restricted
to $\Gamma_{\alpha,z}$ is denoted by $g_{\alpha,ij}(y,z)dy^i\otimes dy^j$%(for simplicity the supcript $\alpha$ will be omitted)
, where
\begin{eqnarray}\label{metirc tensor}
g_{\alpha,ij}(y,z)&=&X_i(y,z)\cdot X_j(y,z) \nonumber \\
&=&g_{\alpha,ij}(y,0)-2z\sum_{k=1}^{n-1}A_{\alpha,ik}(y,0)g_{jk}(y,0)+z^2\sum_{k,l=1}^n g_{\alpha,kl}(y,0)A_{\alpha,ik}(y,0)A_{\alpha,jl}(y,0).
\end{eqnarray}

The second fundamental form of $\Gamma_{\alpha,z}$ has the form
\begin{equation}\label{A(z)}
A_\alpha(y,z)=\left[I-zA_\alpha(y,0)\right]^{-1}A_\alpha(y,0).
\end{equation}
%Here by definition $A_{\alpha,0}(y,0)=A_\alpha(y)$, the second fundamental form of $\Gamma_\alpha$ at $y$.

\subsection{Some notations}
In the remaining part of this paper the following notations will be employed.
\begin{itemize}

\item Given a point on $\Gamma_\alpha$ with local coordinates $(y,0)$ in the Fermi coordinates, denote
\[D_\alpha(y):=\min_{\beta\neq\alpha}|d_{\beta}(y,0)|.\]

\item For any $x\in B_R(0)$ and $r\in(0,R-|x|)$, denote
\[A(r;x):=\max_\alpha\max_{y\in\overline{\Gamma_\alpha\cap B_r(x)}}e^{- D_\alpha(y)}.\]

\item The covariant derivative on $\Gamma_{\alpha,z}$ with respect to the induced metric is denoted by $\nabla_{\alpha,z}$.
\item The area form on $\Gamma_{\alpha,z}$ with respect to the induced metric is denoted by $dA_{\alpha,z}=\lambda_\alpha(y,z)dy$, where $\lambda_\alpha(y,z)=\sqrt{\mbox{det}\left[g_{\alpha,ij}(y,z)\right]}$.
\item We use $B^\alpha_r(y)$ to denote the open ball on $\Gamma_\alpha$ with center $y$ and radius $r$, which is measured with respect to intrinsic distance.
%\item For $\lambda>0$, let $\mathcal{N}_\alpha^\lambda$ be the $\lambda$-neighborhood of $\Gamma_\alpha$, i.e. the open set $\{|d_\alpha|<\lambda\}$.

%\item For any $y\in\Gamma_\alpha$ and $r>0$, let \[\mathcal{N}_\alpha^\lambda(y;r):=\left\{(\widetilde{y},z)\in\mathcal{N}_\alpha^\lambda, \quad \widetilde{y}\in B_r^\alpha(y)\right\}.\]

\item For $\lambda\in\R$, let
\[\mathcal{M}_\alpha^\lambda:=\left\{|d_\alpha|<|d_{\alpha-1}|+\lambda \quad \mbox{and } \quad  |d_\alpha|<|d_{\alpha+1}|+\lambda\right\}.\]

\item
In the Fermi coordinates with respect to $\Gamma_\alpha$, there exist two continuous functions $\rho_\alpha^\pm(y)$ such that
\[\mathcal{M}_\alpha^0=\left\{(y,z): \rho_\alpha^-(y)<z<\rho_\alpha^+(y)\right\}.\]

\end{itemize}

\subsection{Deviation in $z$} In this subsection we collect several estimates on the deviation of various terms in $z$, when $z\neq 0$. Recall that $\varepsilon$ is the upper bound on curvatures of level sets of $u$, see \eqref{bound on second fundamental form}.

By \eqref{bound on second fundamental form}, $|A_\alpha(y,0)|\lesssim \varepsilon$. Thus by \eqref{A(z)}, for $|z|< R$, $|A_\alpha(y,z)|\lesssim \varepsilon$. As in \cite{Wang-Wei2}, we also have
\begin{lem}
For any $y\in\Gamma_\alpha\cap B_{R-1}(0)$,
\begin{equation}\label{bound on 3rd derivatives}
|\nabla_{\alpha,0} A_\alpha(y,0)|+|\nabla_{\alpha,0}^2 A_\alpha(y,0)|\lesssim\varepsilon.
\end{equation}
\end{lem}

By \eqref{A(z)}, we have
\begin{equation}\label{error in z 1}
|A_\alpha(y,z)-A_\alpha(y,0)|\lesssim |z||A_\alpha(y,0)|^2\lesssim\varepsilon^2|z|.
\end{equation}
Similarly, by \eqref{metirc tensor}, the deviation of metric tensors is
\begin{equation}\label{error in z 2}
\begin{cases}
   & |g_{\alpha,ij}(y,z)-g_{\alpha,ij}(y,0)|\lesssim \varepsilon|z|,  \\
  & |g_{\alpha}^{ij}(y,z)-g_{\alpha}^{ij}(y,0)|\lesssim \varepsilon|z|.
\end{cases}
\end{equation}
As a consequence, the deviation of mean curvature is
\begin{equation}\label{error in z 4}
|H_\alpha(y,z)-H_\alpha(y,0)|\lesssim  \varepsilon^2|z|.
\end{equation}

By \eqref{derivatives bound} and \eqref{bound on 3rd derivatives}, for any $|z|<R$,
\begin{equation}\label{derivatives of metric tensor}
\sum_{i,j=1}^{n-1}\left(|\nabla_{\alpha,z}g_{\alpha,ij}(y,z)|+|\nabla_{\alpha,z}g_{\alpha}^{ij}(y,z)|+|\nabla_{\alpha,z}^2g_{\alpha,ij}(y,z)|+|\nabla_{\alpha,z}^2g_{\alpha}^{ij}(y,z)|\right)\lesssim\varepsilon.
\end{equation}

The Laplacian operator in Fermi coordinates has the form
\[\Delta_{\R^n}=\Delta_{\alpha,z}-H_\alpha(y,z)\partial_z+\partial_{zz},\]
where $\Delta_{\alpha,z}$ is the Beltrami-Laplace operator on $\Gamma_{\alpha,z}$, that is,
\begin{eqnarray}\label{Laplacian}
\Delta_{\alpha,z}&=&\sum_{i,j=1}^{n-1}\frac{1}{\sqrt{\mbox{det}(g_{\alpha,ij}(y,z))}}\frac{\partial}{\partial y_j}\left(\sqrt{\mbox{det}(g_{\alpha,ij}(y,z))}g_{\alpha}^{ij}(y,z)\frac{\partial}{\partial y_i}\right) \nonumber\\
&=&\sum_{i,j=1}^{n-1}g_{\alpha}^{ij}(y,z)\frac{\partial^2}{\partial y_i\partial y_j}+\sum_{i=1}^{n-1}b_\alpha^i(y,z)\frac{\partial}{\partial y_i}
\end{eqnarray}
with
\[b_\alpha^i(y,z)=\frac{1}{2}\sum_{j=1}^{n-1}g_{\alpha}^{ij}(y,z)\frac{\partial}{\partial y_j}\log\mbox{det}(g_{\alpha,ij}(y,z)).\]

By \eqref{error in z 2} and \eqref{derivatives of metric tensor}, we get
\begin{lem}
For any function $\varphi\in C^2\left(\Gamma_\alpha\right)$ and $|z|<R$,
\begin{equation}\label{error in z 5}
|\Delta_{\alpha,z}\varphi -\Delta_{\alpha,0}\varphi |\lesssim\varepsilon|z|\left(|\nabla_{\alpha,0}^2\varphi |+|\nabla_{\alpha,0}\varphi |\right).
\end{equation}
\end{lem}

Finally we recall a commutator estimate from \cite{Wang-Wei2}.
\begin{lem}\label{commutator estimate}
For any $\varphi\in C^3(\Gamma_\alpha)$ and $|z|<R$,
\[\Big|\frac{\partial}{\partial y_i}\Delta_{\alpha,z}\varphi-\Delta_{\alpha,z}\frac{\partial\varphi}{\partial y_i}\Big|\lesssim \varepsilon\left(|\nabla_{\alpha,0}^2\varphi|+|\nabla_{\alpha,0}\varphi|\right).\]
\end{lem}

\subsection{Comparison of distance functions}

Given a point $X$, let $\Pi_\alpha(X)$ be the nearest point on $\Gamma_\alpha$ to $X$.
The following lemma is taken from \cite{Wang-Wei2}.
\begin{lem}\label{comparison of distances}
For any $K>0$, there exists a constant $C(K)$ so that the following holds.
For any $X\in B_{7R/8}(0)$ and $\alpha\neq\beta$, if
$|d_\alpha(X)|\leq K|\log\varepsilon|$ and $|d_\beta(X)|\leq K|\log\varepsilon|$ at the same time, then we have
\[
\mbox{dist}_{\Gamma_\beta}\left(\Pi_\beta\circ\Pi_\alpha(X),\Pi_\beta(X)\right)\leq C(K)\varepsilon^{1/2}|\log\varepsilon|^{3/2},
\]
\[
|d_\beta\left(\Pi_\alpha(X)\right)+d_\alpha\left(\Pi_\beta(X)\right)|\leq C(K)\varepsilon^{1/2}|\log\varepsilon|^{3/2},
\]
\[
|d_\alpha(X)-d_\beta(X)+d_\beta\left(\Pi_\alpha(X)\right)|\leq C(K)\varepsilon^{1/2}|\log\varepsilon|^{3/2},
\]
\[
|d_\alpha(X)-d_\beta(X)-d_\alpha\left(\Pi_\beta(X)\right)|\leq C(K)\varepsilon^{1/2}|\log\varepsilon|^{3/2},
\]
\[
1-\nabla d_\alpha(X)\cdot \nabla d_\beta(X)\leq C(K)\varepsilon^{1/2}|\log\varepsilon|^{3/2},
\]
\end{lem}
The following lemma is an easy consequence of Lemma \ref{lem graph construction}.
\begin{lem}\label{biLip of projection operator}
  For any $\alpha\neq\beta$, both $\Pi_\beta\lfloor_{\Gamma_\alpha}$ and its inverse are Lipschitz continuous with their Lipschitz constants bounded by a universal constant $C$.
\end{lem}
Finally, the following fact will be used a lot in this paper.
\begin{lem}\label{lem distance ladder}
For any $y\in\Gamma_\alpha$,
\[\sum_{\beta\neq\alpha}e^{-|d_\beta(y,0)|}\lesssim e^{-D_\alpha(y)}.\]
\end{lem}
\begin{proof}
  By Lemma \ref{O(1) scale} and Lemma \ref{comparison of distances}, there exists a constant $C\gg 1$ such that
\[
|d_\beta(y,0)|\geq D_\alpha(y)+C\left(|\beta-\alpha|-1\right).
\]
Summing in $\beta$ we conclude the proof.
\end{proof}

\section{An approximate solution}\label{sec approximate solution}
\setcounter{equation}{0}

\subsection{Optimal approximation}\label{sec optimal approximation}

Fix a function $\zeta\in C_0^\infty(-2,2)$ with $\zeta\equiv 1$ in $(-1,1)$, $|\zeta^\prime|+|\zeta^{\prime\prime}|\leq 16$. Let
\[\bar{g}(t)=\zeta\left(4|\log\varepsilon|t\right)g(t)+\left[1-\zeta(4|\log\varepsilon|t)\right]\mbox{sgn}(t), \quad t\in(-\infty,+\infty).\]
In particular, $\bar{g}\equiv 1$ in $(8|\log\varepsilon|,+\infty)$ and $\bar{g}\equiv -1$ in $(-\infty, -8|\log\varepsilon|)$.

$\bar{g}$ is an approximate solution to the one dimensional Allen-Cahn equation, that is,
\begin{equation}\label{1d eqn}
\bar{g}^{\prime\prime}=W^\prime\left(\bar{g}\right)+\bar{\xi},
\end{equation}
where $\mbox{spt}(\bar{\xi})\in\{4|\log\varepsilon|<|t|<8|\log\varepsilon|\}$, and $|\bar{\xi}|+|\bar{\xi}^\prime|+|\bar{\xi}^{\prime\prime}|\lesssim \varepsilon^3$.

We also have (see Appendix \ref{sec 1d solution} for the definition of $\sigma_0$)
\begin{equation}\label{1d energy}
\int_{-\infty}^{+\infty}\bar{g}^\prime(t)^2dt=\sigma_0+O\left(\varepsilon^3\right).
\end{equation}

\medskip

Suppose $u$ has the same sign as $(-1)^\alpha d_\alpha$ near $\Gamma_\alpha$.
Given a  function $h_\alpha\in C^2(\Gamma_\alpha)$, let
\[g_\alpha(y,z;h_\alpha):=\bar{g}\left((-1)^\alpha\left(z-h_\alpha(y)\right)\right),\]
where $(y,z)$ is the Fermi coordinates with respect to $\Gamma_\alpha$.

Given a sequence of functions $(h_\alpha)=:h$, define the function $g(y,z;h)$ in the following way: for each $\alpha$,
\[g(y,z;h):= g_\alpha+\sum_{\beta<\alpha}\left[g_\beta-(-1)^\beta\right]+\sum_{\beta>\alpha}\left[g_\beta+(-1)^\beta\right] \quad \mbox{in } \mathcal{M}_\alpha^0.\]
By the definition of $\bar{g}$ and Lemma \ref{O(1) scale}, there are only finitely many terms in the above sum.

For simplicity of notation, denote
\[g_\alpha^\prime(y,z;h_\alpha)=\bar{g}^\prime\left((-1)^{\alpha}(z-h_\alpha(y))\right), \quad g_\alpha^{\prime\prime}(y,z;h_\alpha)=\bar{g}^{\prime\prime}\left((-1)^{\alpha}(z-h_\alpha(y))\right), \quad \cdots .\]

\begin{prop}\label{prop optimal approximation}
There exists $h=(h_\alpha)$ with $|h_\alpha|\ll 1$ for each $\alpha$, such that for any $\alpha$ and $y\in\Gamma_\alpha\cap B_{7R/8}(0)$,
\begin{equation}\label{orthogonal condition}
\int_{-\infty}^{+\infty}\left[u(y,z)-g(y,z;h)\right]g_\alpha^\prime(y,z;h_\alpha)dz=0,
\end{equation}
where $(y,z)$ denotes the Fermi coordinates with respect to $\Gamma_\alpha$.
\end{prop}
\begin{proof}
Denote
\[F(h):=\left(\int_{-\infty}^{+\infty}\left[u(y,z)-g(y,z;h)\right]g_\alpha^\prime\left(y,z;h_\alpha\right)dz\right),\]
which is viewed as a map from the Banach space $\mathcal{X}:=\bigoplus_\alpha C^0(\Gamma_\alpha\cap B_{7R/8}(0))$ to itself.

Clearly $F$ is a $C^1$ map. Furthermore, $\left(DF(h)\xi\right)_\alpha$, the $\alpha$-component of $ DF(h)\xi$, equals
\begin{eqnarray*}
&&(-1)^{\alpha}\xi_\alpha(y)\int_{-\infty}^{+\infty}\left[ g_\alpha^\prime\left(y,z;h_\alpha\right)^2-\left(u(y,z)-g(y,z;h)\right)g_\alpha^{\prime\prime}\left(y,z;h_\alpha\right) \right]dz\\
&+&\sum_{\beta\neq\alpha}(-1)^{\beta}\int_{-\infty}^{+\infty}\xi_\beta(\Pi_\beta(y,z))g_\alpha^\prime\left(y,z;h_\alpha\right)g_\beta^\prime\left(y,z;h_\beta\right)\nabla d_\beta\left(y,z\right)\cdot\nabla d_\alpha\left(y,z\right) dz.
\end{eqnarray*}
By Lemma \ref{O(1) scale}, there exists a $\tau>0$ such that if for all $\alpha$, $\|h_\alpha\|_{C^0(\Gamma_\alpha\cap B_{7R/8}(0))}<\tau$, then
\[\int_{-\infty}^{+\infty}\left[ g_\alpha^\prime\left(y,z;h_\alpha\right)^2-\left(u(y,z)-g(y,z;h)\right)g_\alpha^{\prime\prime}\left(y,z;h_\alpha\right) \right]dz\geq \frac{\sigma_0}{2},\]
\[\Big|\int_{-\infty}^{+\infty}g_\alpha^\prime\left(y,z;h_\alpha\right)g_\beta^\prime\left(y,z;h_\beta\right)\nabla d_\beta\left(y,z\right)\cdot\nabla d_\alpha\left(y,z\right) dz\Big|\lesssim e^{-|d_\beta(y,0)|}.\]

By Lemma \ref{lem distance ladder}, $DF(h)$ is diagonal dominated and thus invertible with $\|DF(h)^{-1}\|_{\mathcal{X}\mapsto\mathcal{X}}\leq C$.
By Lemma \ref{O(1) scale}, for all $\varepsilon$ small enough, $\|F(0)\|_{\mathcal{X}}\ll 1$. The existence of $h$ then follows from the inverse function theorem.
\end{proof}
\begin{rmk}
The proof shows that for each $\alpha$, $\|h_\alpha\|_{C^0(\Gamma_\alpha\cap B_{7R/8}(0))}=o(1)$. By differentiating \eqref{orthogonal condition}, we can show that $\|h_\alpha\|_{C^3(\Gamma_\alpha\cap B_{7R/8}(0))}=o(1)$ for each $\alpha$.
\end{rmk}

Denote $g_\alpha(y,z):=g_\alpha(y,z;h_\alpha)$ and $g_\ast(y,z):=g(y,z;h)$, where $h$ is given in the previous proposition.
As before we denote
\[g_\alpha^\prime(y,z)=g_\alpha^\prime(y,z;h_\alpha), \quad g_\alpha^{\prime\prime}(y,z)=g_\alpha^{\prime\prime}(y,z;h_\alpha), \quad \cdots .\]
Let $\phi:=u-g_\ast$ be the error between the  solution $u$ and the approximate solution $g_\ast$.

In the Fermi coordinates with respect to $\Gamma_\alpha$,
$\phi$ satisfies the following equation
\begin{eqnarray}\label{error equation}
&&\Delta_{\alpha,z}\phi-H_\alpha(y,z)\partial_z\phi+\partial_{zz}\phi \nonumber\\
&=&W^\prime(g_\ast+\phi)-\sum_{\beta}
 W^\prime(g_\beta) +(-1)^{\alpha}g_\alpha^\prime\left[H_\alpha(y,z)+\Delta_{\alpha,z} h_\alpha(y)\right]-g_\alpha^{\prime\prime}|\nabla_{\alpha,z}h_\alpha|^2\\
&+&\sum_{\beta\neq\alpha} \left[(-1)^{\beta}g_\beta^\prime\mathcal{R}_{\beta,1}-
g_\beta^{\prime\prime}\mathcal{R}_{\beta,2}\right] -\sum_{\beta}\xi_\beta,  \nonumber
\end{eqnarray}
where for each $\beta$, in the Fermi coordinates with respect to $\Gamma_\beta$,
\[\xi_\beta(y,z)=\bar{\xi}\left((-1)^{\beta}(z-h_\beta(y))\right),\]
\[
\mathcal{R}_{\beta,1}(y,z):=H_\beta(y,z)+\Delta_{\beta,z} h_\beta(y),
\]
\[\mathcal{R}_{\beta,2}(y,z):=|\nabla_{\beta,z} h_\beta(y)|^2.\]

\subsection{Interaction terms}

In this subsection we collect several estimates on the interaction term $\mathcal{I}:=W^\prime(g_\ast)-\sum_{\beta} W^\prime(g_\beta)$ between different components.

\begin{lem}\label{interaction term}
In $\mathcal{M}^4_\alpha$,
\begin{eqnarray}\label{interaction}
\mathcal{I}
&=&\left[W^{\prime\prime}(g_\alpha)-1\right]\left[g_{\alpha-1}-(-1)^{\alpha-1}\right]+\left[W^{\prime\prime}(g_\alpha)-1\right]\left[g_{\alpha+1}+(-1)^{\alpha+1}\right] \\
&+&O\left(e^{-2d_{\alpha-1}}+e^{2d_{\alpha+1}}\right)+O\left(e^{-d_{\alpha-2}-|d_\alpha|}+e^{d_{\alpha+2}-|d_\alpha|}\right). \nonumber
\end{eqnarray}
\end{lem}

The following upper bound on the interaction term will be used a lot in the below.

\begin{lem}\label{upper bound on interaction}
For any $(y,z)\in \mathcal{M}^4_\alpha$,
\[\big|\mathcal{I}(y,z)\big|\lesssim e^{-D_{\alpha}(y)}+\varepsilon^2.\]
\end{lem}

The Lipschitz norm of interaction terms can also be estimated in a similar way.
\begin{lem}\label{Holder bound on interaction}
For any $(y,z)\in\mathcal{M}_\alpha^3$,
\[\big\|\mathcal{I}\big\|_{ Lip(B_1(y,z))}\lesssim \max_{B_1^\alpha(y)}e^{-D_\alpha}+\varepsilon^2.\]
\end{lem}

\subsection{Controls on $h$ using $\phi$}

The choice of optimal approximation in Subsection \ref{sec optimal approximation} has the advantage that
 $h$ is controlled by $\phi$. This will allow us to iterate various elliptic estimates in Section \ref{sec first order} below.
\begin{lem}\label{control on h_0}
For each $\alpha$ and $y\in\Gamma_\alpha$, we have
\begin{equation}\label{control on h 1}
 \|h_\alpha\|_{C^{2,\theta}(B_1^\alpha(y))}\lesssim\|\phi\|_{C^{2,\theta}(B_1(y,0))}+\max_{B_1^\alpha(y)}e^{-D_\alpha},
\end{equation}
\begin{eqnarray}\label{control on h_2}
\|\nabla_{\alpha,0} h_\alpha\|_{C^{1,\theta}(B_1^\alpha(y))}&\lesssim&\|\nabla_{\alpha,0}\phi\|_{C^{1,\theta}(B_1(y,0))}
+\varepsilon^{1/6}\max_{B_1^\alpha(y)}e^{-D_\alpha} \nonumber\\
&+&\left(\max_{\beta: |d_\beta(y,0)|\leq 8|\log\varepsilon|}\|\nabla_{\beta,0}h_\beta\|_{C^{1,\theta}(B_2^\beta(\Pi_\beta(y,0)))}\right)\left(\max_{B_1^\alpha(y)}e^{-D_\alpha}\right).
\end{eqnarray}
\end{lem}
\begin{proof}
Fix an $\alpha$.  In the Fermi coordinates with respect to $\Gamma_\alpha$, because $u(y,0)=0$,
\begin{eqnarray}\label{representation of h_alpha}
\phi(y,0)=-\bar{g}\left((-1)^{\alpha+1} h_\alpha(y)\right)&-&\sum_{\beta<\alpha}\left[\bar{g}\left((-1)^{\beta}\left(d_\beta(y,0)-h_\beta(\Pi_\beta(y,0))\right)\right)-(-1)^{\beta}\right] \nonumber\\
&  -&\sum_{\beta>\alpha}\left[\bar{g}\left((-1)^{\beta}\left(d_\beta(y,0)-h_\beta(\Pi_\beta(y,0))\right)\right)+(-1)^{\beta}\right].
\end{eqnarray}
Note that for $\beta\neq\alpha$,  $|h_\beta(\Pi_\beta(y,0))|\ll 1$. Then using Lemma \ref{lem distance ladder}, we get
\begin{equation}\label{4.1}
|h_\alpha(y)| \lesssim |\phi(y,0)|+\sum_{\beta\neq\alpha}e^{-|d_\beta(y,0)|} \lesssim  |\phi(y,0)|+e^{-D_\alpha(y)}.
\end{equation}

Differentiating \eqref{representation of h_alpha}, we get
\[
\nabla_{\alpha,0}\phi(y,0)=(-1)^{\alpha} \bar{g}^\prime\left((-1)^{\alpha+1} h_\alpha(y)\right)\nabla_{\alpha,0}h_\alpha(y)-\sum_{\beta\neq\alpha}(-1)^\beta g_\beta^\prime(y,0)\nabla_{\alpha,0}\left( d_\beta-h_\beta\circ\Pi_\beta\right)(y,0),
\]
and
\begin{eqnarray*}
\nabla_{\alpha,0}^2\phi(y,0)
&=&(-1)^{\alpha} \bar{g}^\prime\left((-1)^{\alpha+1} h_\alpha(y)\right)\nabla_{\alpha,0}^2h_\alpha(y)-
\bar{g}^{\prime\prime}\left((-1)^\alpha h_\alpha(y)\right)\nabla_{\alpha,0}h_\alpha(y)
\otimes\nabla_{\alpha,0}h_\alpha(y)\\
&-&\sum_{\beta\neq\alpha}(-1)^{\beta} g_\beta^\prime(y,0)\nabla_{\alpha,0}^2\left( d_\beta-h_\beta\circ\Pi_\beta\right)(y,0)\\
&-&\sum_{\beta\neq\alpha}
g_\beta^{\prime\prime}(y,0)\nabla_{\alpha,0}\left( d_\beta-h_\beta\circ\Pi_\beta\right)(y,0)
\otimes\nabla_{\alpha,0}\left( d_\beta-h_\beta\circ\Pi_\beta\right)(y,0).
\end{eqnarray*}

By Lemma \ref{comparison of distances}, if $g_\beta^\prime(y,0)\neq0$,
\[|\nabla_{\alpha,0}d_\beta|=\sqrt{1-\nabla d_\beta\cdot\nabla d_\alpha}\lesssim\varepsilon^{1/6}.\]
Thus
\[
|\nabla_{\alpha,0} h_\alpha(y)| \lesssim |\nabla_{\alpha,0} \phi(y,0)|+\left(\varepsilon^{1/6}+|\nabla_{\beta,0} h_\beta(\Pi_\beta(y,0))|\right)\left(e^{d_{\alpha+1}(y,0)}+e^{-d_{\alpha-1}(y,0)}\right).
\]
A similar calculation leads to an upper bound on $|\nabla_{\alpha,0}^2h_\alpha(y)|$.

Finally,  the H\"{o}lder estimate in \eqref{control on h_2} follows by combining the above representation formula, Lemma \ref{Holder bound on interaction} and the
bound
\[|\nabla_{\alpha,0}^2d_\beta|\leq|\nabla^2d_\beta|\lesssim\varepsilon.\]
Here it is useful to note that $\nabla^2d_\beta$ is the second fundamental form of $\Gamma_{\beta,z}$, as well as the fact that $\Pi_\beta(B_1^\alpha(y))\subset B_2^\beta(\Pi_\beta(y,0))$ if $|d_\beta(y,0)|\leq 8|\log\varepsilon|$.
\end{proof}

\section{A Toda system}\label{sec Toda system}
\setcounter{equation}{0}

In the Fermi coordinates with respect to $\Gamma_\alpha$, multiplying \eqref{error equation} by $g_\alpha^\prime$ and integrating in
$z$ leads to
\begin{eqnarray}\label{H eqn}
&&\int_{-\infty}^{+\infty}\left[g_\alpha^\prime\Delta_{\alpha,z}\phi-H_\alpha(y,z)g_\alpha^\prime\partial_z\phi+g_\alpha^\prime\partial_{zz}\phi\right] \\
&=&\int_{-\infty}^{+\infty} \left[W^\prime(g_\ast+\phi)-\sum_\beta W^\prime(g_\beta)\right] g_\alpha^\prime+(-1)^{\alpha}\int_{-\infty}^{+\infty}\left[H_\alpha(y,z)+\Delta_{\alpha,z}h_\alpha(y)\right]g_\alpha^\prime(z)^2\nonumber\\
&-&\int_{-\infty}^{+\infty} g_\alpha^{\prime\prime}g_\alpha^\prime|\nabla_{\alpha,z}
h_\alpha|^2 +\sum_{\beta\neq\alpha}\int_{-\infty}^{+\infty}\left[(-1)^{\beta}g_\alpha^\prime g_\beta^\prime \mathcal{R}_{\beta,1} - g_\alpha^\prime g_\beta^{\prime\prime}\mathcal{R}_{\beta,2}\right]  -\sum_{\beta}\int_{-\infty}^{+\infty}\xi_\beta g_\alpha^\prime .      \nonumber
\end{eqnarray}
From this equation we deduce that
\begin{equation}\label{Toda system}
H_\alpha(y,0)+\Delta_{\alpha,0}h_\alpha(y)
=\frac{2 A_{(-1)^{\alpha-1}}^2}{\sigma_0} e^{-d_{\alpha-1}(y,0)}-\frac{2 A_{(-1)^{\alpha}}^2}{\sigma_0}e^{d_{\alpha+1}(y,0)}+E_\alpha^0(y),
\end{equation}
where $E_\alpha^0$ is a higher order term. (See Appendix \ref{sec 1d solution} for the definition of $A_1$ and $A_{-1}$.) More precisely, we have
\begin{lem}\label{lem 5.1}
For any $x\in B_{5R/6}(0)$ and $r\in(0,R/7)$,
\begin{eqnarray}\label{5.1}
\max_\alpha\big\|E_\alpha^0\big\|_{C^\theta({\Gamma_\alpha\cap B_{r}(x)})}
&\lesssim& \varepsilon^2+ \varepsilon^{\frac{1}{3}}A\left(r+10|\log\varepsilon|;x\right) +A\left(r+10|\log\varepsilon|;x\right)^{\frac{3}{2}}  \nonumber\\
&+&\max_\alpha\big\|H_\alpha+\Delta_{\alpha,0}h_\alpha\big\|_{C^\theta(\Gamma_\alpha\cap B_{r+10|\log\varepsilon|}(x))}^2+ \|\phi\|_{C^{2,\theta}(B_{r+10|\log\varepsilon|}(x))}^2.
\end{eqnarray}
\end{lem}
The proof is given in Appendix \ref{sec proof of Lemma 5.1}.

Since all terms in the right hand side of \eqref{5.1} are small quantities, a direct consequence of this lemma is
\begin{coro}\label{coro 5.2}
There exists a universal constant $C>0$ such that for any $x\in B_{5R/6}(0)$ and $r\in(0,R/7)$,
\begin{eqnarray}\label{5.2}
&&\max_\alpha\big\|H_\alpha+\Delta_{\alpha,0}h_\alpha\big\|_{C^\theta({\Gamma_\alpha\cap B_{r}(x)})} \nonumber\\
&\leq&\frac{1}{4}\left(\max_\alpha\big\|H_\alpha+\Delta_{\alpha,0}h_\alpha\big\|_{C^\theta(\Gamma_\alpha\cap B_{r+10|\log\varepsilon|}(x))}+ \|\phi\|_{C^{2,\theta}(B_{r+10|\log\varepsilon|}(x))}\right) \\
&+&C \varepsilon^2+C A\left(r+10|\log\varepsilon|;x\right) .\nonumber
\end{eqnarray}
\end{coro}

\section{ Estimates on $\phi$}\label{sec first order}
\setcounter{equation}{0}

In this section we prove the following $C^{2,\theta}$ estimate on $\phi$.
\begin{prop}\label{prop Schauder estimate}
For any $x\in B_{5R/6}(0)$ and $r\in(0,R/7)$,
\begin{equation}\label{Schauder estimates}
\max_\alpha\big\|H_\alpha+\Delta_{\alpha,0}h_\alpha\big\|_{C^\theta(\Gamma_\alpha\cap B_{r}(x))}+\|\phi\|_{C^{2,\theta}(B_r(x))}
\lesssim \varepsilon^2+A\left(r+50|\log\varepsilon|^2;x\right).
\end{equation}
\end{prop}
The first order H\"{o}lder estimates of $\phi$ will be established in Subsection \ref{subsec 6.1} and Subsection \ref{subsec 6.2}. The second order  H\"{o}lder estimate will be proved in Subsection \ref{sec second order estimate}.

To prove the first order H\"{o}lder estimate on $\phi$, fix a large constant $L>0$, for each $\alpha$ define
\[\Omega_\alpha^1:=\{-L<d_\alpha<L\}\cap\mathcal{M}_\alpha^0,\quad   \quad \Omega_\alpha^2:=\{d_\alpha>L/2\}\cap\mathcal{M}_\alpha^0,\]
and
\[\Omega_\alpha^3:=\{-L/2\leq d_\alpha\leq L/2\}\cap\mathcal{M}_\alpha^0.\]
%,\quad   \quad \Omega_\alpha^4:=\{d_\alpha\geq 2L\}\cap\mathcal{M}_\alpha^0
We will estimate the $C^{1,\theta}$ norm of $\phi$ in $\Omega_\alpha^1\cap B_r(x)$ and $\Omega_\alpha^2\cap B_r(x)$ separately.
Roughly speaking, in $\Omega_\alpha^1$, $\phi$ satisfies
\[-\Delta\phi+W^{\prime\prime}(g_\alpha)\phi=\mbox{interaction terms}+\mbox{parallel component}+\mbox{errors}.\]
Together with the orthogonal condition \eqref{orthogonal condition} we get a control on $\phi$, which is possible by the decay estimate of the operator $-\Delta+W^{\prime\prime}(g)$ in the class of functions satisfying the orthogonal condition \eqref{orthogonal condition}, see for example \cite{DKWY}.
In  $\Omega_\alpha^2$, $\phi$ satisfies
\[-\Delta\phi+\phi=\mbox{interaction terms}+\mbox{errors}.\]
Hence a control on $\phi$ is possible by using the decay estimate of the coercive operator $-\Delta+1$.

\subsection{$C^{1,\theta}$ estimate in $\Omega_\alpha^2$}\label{subsec 6.1}
We start with the easy case.
In $\Omega_\alpha^2$, the equation for $\phi$ can be written in the following way.
\begin{lem}\label{lem 8.1}
For any $\alpha$, in $\Omega_\alpha^2$,
\[\Delta_{\alpha,z}\phi(y,z)-H_\alpha(y,z)\partial_z\phi(y,z)+\partial_{zz}\phi(y,z)=\left[1+o(1)\right]\phi(y,z)+E_\alpha^2(y,z),\]
where
\begin{eqnarray*}
\|E_\alpha^2\|_{L^\infty(\Omega^2_\alpha\cap B_r(x))}&\leq&C \varepsilon^2+C A\left(r+10|\log\varepsilon|;x\right) +C\|\phi\|_{C^{2,\theta}(B_{r+10|\log\varepsilon|}(x))}^2 \\
&+&Ce^{-L}\max_\alpha\big\|H_\alpha+\Delta_{\alpha,0}h_\alpha\big\|_{L^\infty(\Gamma_\alpha\cap B_{r+10|\log\varepsilon|}(x))}.
\end{eqnarray*}
\end{lem}
\begin{proof}
The $L^\infty$ estimate on $E_\alpha^2$ is a consequence of the following estimates on those terms in \eqref{error equation}.
\begin{itemize}
  \item First we have $W^\prime(g_\ast+\phi)-W^\prime(g_\ast)=\left[W^{\prime\prime}(g_\ast)+O\left(\phi\right)\right]\phi=\left[1+o(1)\right]\phi$.
  \item By Lemma \ref{upper bound on interaction}, $W^\prime(g_\ast)-\sum_\beta W^\prime(g_\beta)=O\left(e^{-D_\alpha}\right)+O\left(\varepsilon^2\right)$.
  \item By \eqref{error in z 1}-\eqref{error in z 5}, we get
  \begin{eqnarray*}
    &&g_\alpha^\prime\left[H_\alpha(y,z)+\Delta_{\alpha,z} h_\alpha\right]\\
    &=& g_\alpha^\prime\left[H_\alpha(y,0)+\Delta_{\alpha,0} h_\alpha(y)\right]
   +g_\alpha^\prime\left[H_\alpha(y,z)-H_\alpha(y,0)\right]+g_\alpha^\prime\left[\Delta_{\alpha,z}h_\alpha-\Delta_{\alpha,0}h_\alpha\right] \\
    &=& O\left(e^{-L}\big| H_\alpha(y,0)+\Delta_{\alpha,0} h_\alpha(y)\big|\right)+O\left(\varepsilon^2\right)+O\left(|\nabla_{\alpha,0}^2h_\alpha|^2+|\nabla_{\alpha,0}h_\alpha|^2\right).
  \end{eqnarray*}
Concerning estimates on the last two terms involving $h_\alpha$, we use Lemma \ref{control on h_0}.
  \item Similarly,  estimates on $g_\alpha^{\prime\prime}|\nabla_{\alpha,z}h_\alpha|^2$ follow from Lemma \ref{control on h_0}.
  \item Those two terms involving $g_\beta^\prime \mathcal{R}_{\beta,1}$ and $g_\beta^{\prime\prime} \mathcal{R}_{\beta,2}$ can be estimated as in the above two cases, but now in Fermi coordinates with respect to $\Gamma_\beta$. Note that we need only to consider those $\beta$ satisfying $\Gamma_\beta\cap B_{r+8|\log\varepsilon|}(x)\neq\emptyset$, because otherwise $g_\beta^\prime=0$ in $B_r(x)$. To put all  estimates of $\beta\neq\alpha$ together, we use Lemma \ref{lem distance ladder}.
  \item Finally, by definition of $\xi_\beta$, $\sum_{\beta}\xi_\beta=O\left(\varepsilon^3|\log\varepsilon|\right)=O\left(\varepsilon^2\right)$.\qedhere
\end{itemize}
\end{proof}

By standard interior elliptic estimates on the coercive operator $-\Delta+1$, we deduce that, for any  $\alpha$,
\begin{eqnarray}\label{first order estimate 1}
 \|\phi\|_{C^{1,\theta}(\Omega_\alpha^2\cap B_r(x))}
&\leq& Ce^{-cL}\left(\|\phi\|_{C^{1,\theta}(\Omega_\alpha^2\cap B_{r+10|\log\varepsilon|}(x))} +\|\phi\|_{C^{1,\theta}(\Omega_\alpha^3\cap B_{r+10|\log\varepsilon|}(x))}\right) \nonumber\\
&+&Ce^{-L}\max_\alpha\big\|H_\alpha+\Delta_{\alpha,0}h_\alpha\big\|_{L^\infty(\Gamma_\alpha\cap B_{r+10|\log\varepsilon|}(x))}\\
&+&C \varepsilon^2+C A\left(r+10|\log\varepsilon|;x\right).\nonumber
\end{eqnarray}

\subsection{$C^{1,\theta}$ estimate in $\Omega_\alpha^1$}\label{subsec 6.2}

In $\Omega_\alpha^1$, the equation for $\phi$ can be written in the following way.
\begin{lem}\label{lem 8.2}
In $\Omega_\alpha^1$,
\[\Delta_{\alpha,0}\phi+\partial_{zz}\phi=W^{\prime\prime}(g_\alpha)\phi+(-1)^{\alpha}g_\alpha^\prime
\left[H_\alpha(y,0)+\Delta_{\alpha,0} h_\alpha\right]+E_\alpha^1,
\]
where for some constant $C(L)$,
\[
\|E^1_\alpha\|_{L^\infty(\Omega^1_\alpha\cap B_r(x))}\leq C(L) \varepsilon^2+C(L) A\left(r+10|\log\varepsilon|;x\right) +C(L)\|\phi\|_{C^{2,\theta}(B_{r+10|\log\varepsilon|}(x))}^2.
\]
\end{lem}
\begin{proof}
The proof is similar to the one for Lemma \ref{lem 8.1}, in particular,
\begin{itemize}
\item we use Cauchy inequality and \eqref{error in z 5} to bound $\Delta_{\alpha,z}\phi-\Delta_{\alpha,0}\phi$ (here it is usefule to note that $|z|<2L$ in $\Omega_\alpha^1$);
\item we use Cauchy inequality and the fact that $|H_\alpha(y,z)|\lesssim\varepsilon$ to bound $H_\alpha(y,z)\partial_z\phi$;
\item we use Lemma \ref{upper bound on interaction} to bound interaction terms;
\item we use Lemma \ref{control on h_0} to bound those terms involving $h_\alpha$;
\item by the exponential decay of $\bar{g}^\prime$ at infinity and Lemma \ref{lem distance ladder},
$\sum_{\beta\neq\alpha}g_\beta^\prime \mathcal{R}_{\beta,1}$ and $\sum_{\beta\neq\alpha}g_\beta^\prime \mathcal{R}_{\beta,2}$ are bounded by $ e^{-D_\alpha(y)} $ in $\Omega_\alpha^1$. (Although there are constants $e^{CL}$ appearing when we bound $g_\beta^\prime$ by $O\left(e^{-D_\alpha(y)}\right)$,  they can be incorporated because $|\mathcal{R}_{\beta,1}|+|\mathcal{R}_{\beta,2}|\ll 1$ while $L$, although large, is a fixed constant.) \qedhere
\end{itemize}
\end{proof}

Take a  function  $\xi\in C_0^\infty(-2L,2L)$ satisfying $\xi\equiv 1$ in $(-L,L)$, $|\xi^\prime|\lesssim L^{-1}$ and $|\xi^{\prime\prime}|\lesssim L^{-2}$. Let $\phi_\alpha(y,z):=\phi(y,z)\xi(z)-c_\alpha(y)g_\alpha^\prime(y,z)$, where
\begin{equation}\label{def of c}
c_\alpha(y)=\frac{\int_{-\infty}^{+\infty}\phi(y,z)\left(\xi(z)-1\right)g_\alpha^\prime(y,z)dz}{\int_{-\infty}^{+\infty}g_\alpha^\prime(y,z)^2dz}.
\end{equation}
Hence   by \eqref{orthogonal condition} we still have the orthogonal condition
\begin{equation}\label{orthogonal condition 6.1}
\int_{-\infty}^{+\infty} \phi_\alpha(y,z) g_\alpha^\prime(y,z)dz=0, \quad\forall \ y\in \Gamma_\alpha.
\end{equation}

We have the following estimates on $c_\alpha$.
\begin{lem}\label{estimates on c}
For any $y\in\Gamma_\alpha$,
\[
|c_\alpha(y)|\lesssim e^{-   L}\max_{L<|z|<8|\log\varepsilon|} |\phi(y,z)|,
\]
\[
|\nabla_{\alpha,0} c_\alpha(y)|\lesssim e^{-   L}\max_{L<|z|<8|\log\varepsilon|} \left(|\phi(y,z)|+|\nabla_{\alpha,z}\phi(y,z)|\right),
\]
\[
|\nabla_{\alpha,0}^2 c_\alpha(y)|\lesssim e^{-   L}\max_{L<|z|<8|\log\varepsilon|} \left(|\phi(y,z)|+|\nabla_{\alpha,z}\phi(y,z)|+|\nabla_{\alpha,z}^2\phi(y,z)|\right).
\]
\end{lem}
\begin{proof}
By \eqref{def of c} and the definition of $\bar{g}$ and $\xi$,
\begin{eqnarray*}
|c_\alpha(y)|
&\lesssim& \left(\max_{L<|z|<8|\log\varepsilon|} |\phi(y,z)|\right)\int_L^{+\infty}e^{-   z}dz\\
&\lesssim& e^{-   L}\max_{L<|z|<8|\log\varepsilon|} |\phi(y,z)|.
\end{eqnarray*}

Differentiating \eqref{def of c} gives
\begin{eqnarray*}
&&\nabla_{\alpha,0} c_\alpha(y)\left(\int_{-\infty}^{+\infty}g_\alpha^\prime(y,z)^2dz\right)+c_\alpha(y)\left(\nabla_{\alpha,0}\int_{-\infty}^{+\infty}g_\alpha^\prime(y,z)^2dz\right)\\
&=&\int_{-\infty}^{+\infty}\nabla_{\alpha,0}\phi(y,z)\left(\xi(z)-1\right)g_\alpha^\prime(y,z)dz\\
&-&(-1)^\alpha\nabla_{\alpha,0} h_\alpha(y)\int_{-\infty}^{+\infty}\phi(y,z)\left(\xi(z)-1\right)g_\alpha^{\prime\prime}(y,z)dz.
\end{eqnarray*}
The second estimate follows as above. The third one can be proved in the same way.
\end{proof}
The factor $e^{-L}$ reveals the fact that behavior of $\phi$ in $\Omega_\alpha^2$ has little effect on the behavior of $\phi$ in $\Omega_\alpha^1$, that is, these two parts are almost decoupled.

In the Fermi coordinates with respect to $\Gamma_\alpha$, the equation satisfied by $ \phi_\alpha$ reads as
\begin{equation}\label{6.2}
\Delta_{\alpha,0} \phi_\alpha+\partial_{zz}\phi_\alpha
=W^{\prime\prime}(g_\alpha)\phi_\alpha
+p_\alpha(y)g_\alpha^\prime+F_\alpha,
\end{equation}
where
\[
p_\alpha(y)=(-1)^{\alpha }\left[H_\alpha(y,0)+\Delta_{\alpha,0}h_\alpha(y)\right]-\Delta_{\alpha,0} c_\alpha(y)\]
and
\begin{eqnarray*}
F_\alpha(y,z)&=& 2(-1)^\alpha \nabla_{\alpha,0}c_\alpha(y)\cdot\nabla_{\alpha,0}h_\alpha(y) g_\alpha^{\prime\prime}(y,z)\\
&+&c_\alpha(y)\left[(-1)^{\alpha+1}\Delta_{\alpha,0}h_\alpha(y) g_\alpha^{\prime\prime}(y,z)+g_\alpha^{\prime\prime\prime}(y,z)|\nabla_{\alpha,0}h_\alpha(y)|^2\right]\\
&+&2\partial_z\phi(y,z)\xi^\prime(z)+\phi(y,z)\xi^{\prime\prime}(z)+
E_\alpha^1(y,z)\xi(z)\\
&+&(-1)^{\alpha }\left[H_\alpha(y,0)+\Delta_{\alpha,0}h_\alpha(y)\right]g_\alpha^\prime(y,z)\left[\xi(z)-1\right]-c_\alpha(y)\xi_\alpha^\prime(y,z).
\end{eqnarray*}

\medskip

Combining this expression with Lemma \ref{lem 8.2}, Lemma \ref{estimates on c} and the definition of $\xi$, we obtain
\begin{lem}\label{estimates on e}
For any $x\in B_{5R/6}(0)$ and $r\in(0,R/7)$,
\begin{eqnarray*}
  &&\|F_\alpha\|_{L^\infty(\Omega^1_\alpha\cap B_r(x))} \\
  &\leq& C(L) \varepsilon^2+C(L) A\left(r+10|\log\varepsilon|;x\right) +C(L)\|\phi\|_{C^{2,\theta}(B_{r+10|\log\varepsilon|}(x))}^2 \\
   &+&Ce^{-L} \left[ \max_\alpha\big\|H_\alpha+\Delta_{\alpha,0}h_\alpha\big\|_{L^\infty(\Gamma_\alpha\cap B_{r+10|\log\varepsilon|}(x))}
   +\|\phi\|_{C^{2,\theta}(B_{r+10|\log\varepsilon|}(x))}\right].
\end{eqnarray*}
\end{lem}

By \eqref{6.2} and the orthogonal condition \eqref{orthogonal condition 6.1}, applying standard estimates on the linearized operator $-\Delta+W^{\prime\prime}(g)$ (see for example \cite[Proposition 4.1]{DKWY}) leads to
\begin{eqnarray*}
 \|\phi_\alpha\|_{C^{1,\theta}(B_r(x))}
&\leq& Ce^{-cL} \|\phi_\alpha\|_{C^{1,\theta}(B_{r+10|\log\varepsilon|}(x))} \\
  &+&Ce^{-L} \left[\max_\alpha\big\|H_\alpha+\Delta_{\alpha,0}h_\alpha\big\|_{L^\infty(\Gamma_\alpha\cap B_{r+12|\log\varepsilon|}(x))}
   +\|\phi\|_{C^{2,\theta}(B_{r+10|\log\varepsilon|}(x))}\right] \\
&+& C(L) \varepsilon^2+C(L) A\left(r+10|\log\varepsilon|;x\right) +C(L)\|\phi\|_{C^{2,\theta}(B_{r+10|\log\varepsilon|}(x))}^2.
\end{eqnarray*}
Coming back to $\phi$, by the fact that $\|\phi\|_{C^{2,\theta}(B_{r+12|\log\varepsilon|}(x))}\ll 1$ and the estimates on $c_\alpha$ in Lemma \ref{estimates on c}, we get
\begin{eqnarray}\label{first order estimate 2}
 \|\phi\|_{C^{1,\theta}(\Omega_\alpha^1\cap B_r(x))}
&\leq& C \varepsilon^2+C A\left(r+10|\log\varepsilon|;x\right) \\
&+&Ce^{-cL} \|\phi\|_{C^{2,\theta}(B_{r+10|\log\varepsilon|}(x))} \nonumber\\ &+&Ce^{-cL}\max_\alpha\big\|H_\alpha+\Delta_{\alpha,0}h_\alpha\big\|_{L^\infty(\Gamma_\alpha\cap B_{r+10|\log\varepsilon|}(x))}.\nonumber
\end{eqnarray}

Combining \eqref{first order estimate 1} and \eqref{first order estimate 2}, by choosing $L$ large enough and denoting $\sigma(L):=Ce^{-cL}$, we obtain
\begin{eqnarray}\label{first order estimate}
\|\phi\|_{C^{1,\theta}(B_r(x))}
&\leq&C(L)\varepsilon^2+C(L)A\left(r+10|\log\varepsilon|;x\right)  \nonumber  \\
&+&\sigma(L)\left(\max_\alpha\big\|H_\alpha+\Delta_{\alpha,0}h_\alpha\big\|_{L^\infty(\Gamma_\alpha\cap B_{r+10|\log\varepsilon|}(x))}+ \|\phi\|_{C^{2,\theta}(B_{r+10|\log\varepsilon|}(x))}\right).
\end{eqnarray}

\subsection{ Second order H\"{o}lder estimates on $\phi$}\label{sec second order estimate}

First we have the following H\"{o}lder bounds on the right hand side of \eqref{error equation}.
\begin{lem}\label{Holder for RHS}
For any $x\in B_{5R/6}(0)$ and $r\in(0,R/7)$,
\begin{eqnarray*}
\|\Delta\phi-W^{\prime\prime}(g_\ast)\phi\|_{C^{\theta}(B_r(x))}
&\lesssim&\varepsilon^2+A\left(r+10|\log\varepsilon|;x\right)
+\|\phi\|_{C^{2,\theta}(B_{r+10|\log\varepsilon|}(x))}^2\\
&+&A\left(r+10|\log\varepsilon|;x\right) ^{\frac{1}{2}}
\left(\max_\alpha\|H_\alpha+\Delta_{\alpha,0}h_\alpha\|_{C^{\theta}(B_{r+10|\log\varepsilon|}(x))}\right).
\end{eqnarray*}
\end{lem}
The proof is given in Appendix \ref{sec proof of Lem 9.1}.

By  \eqref{first order estimate} and Schauder estimates, we get
\begin{eqnarray}\label{second order estimate}
\|\phi\|_{C^{2,\theta}(B_r(x))}
&\leq& \sigma(L)\left(\max_\alpha\big\|H_\alpha+\Delta_{\alpha,0}h_\alpha\big\|_{C^\theta(\Gamma_\alpha\cap B_{r+10|\log\varepsilon|}(x))}+ \|\phi\|_{C^{2,\theta}(B_{r+10|\log\varepsilon|}(x))}\right) \nonumber  \\
&+&C(L)\varepsilon^2+C(L)A\left(r+10|\log\varepsilon|;x\right).
\end{eqnarray}

Combining this estimate with Corollary \ref{coro 5.2}, we get
\begin{eqnarray*}
&&\max_\alpha\big\|H_\alpha+\Delta_{\alpha,0}h_\alpha\big\|_{C^\theta(\Gamma_\alpha\cap B_{r}(x))}+\|\phi\|_{C^{2,\theta}(B_r(x))}\\
&\leq&  \frac{1}{2}\left(\max_\alpha\big\|H_\alpha+\Delta_{\alpha,0}h_\alpha\big\|_{C^\theta(\Gamma_\alpha\cap B_{r+10|\log\varepsilon|}(x))}+ \|\phi\|_{C^{2,\theta}(B_{r+10|\log\varepsilon|}(x))}\right) \\
&+&C(L)\varepsilon^2+C(L)A\left(r+10|\log\varepsilon|;x\right).
\end{eqnarray*}

An iteration of this inequality from $r+50|\log\varepsilon|^2$ to $r$ leads to \eqref{Schauder estimates}.
The proof of Proposition \ref{prop Schauder estimate} is thus complete.

\section{Improved estimates on horizontal derivatives}\label{sec improved estimate}
\setcounter{equation}{0}

In this section we prove an improvement on the $C^{1,\theta}$ estimates of horizontal derivatives of $\phi$, $\phi_i:=\partial\phi/\partial y_i$. $1\leq i\leq n-1$.
\begin{prop}\label{prop estimates on horizontal derivatives}
For any $x\in B_{5R/6}(0)$ and $r\in(0,R/7)$,
\[
\|\phi_i\|_{C^{1,\theta}(B_r(x))}
\lesssim\varepsilon^2+A\left(r+60|\log\varepsilon|^2;x\right)^{3/2}+\varepsilon^{1/6}A\left(r+60|\log\varepsilon|^2;x\right).
\]
\end{prop}
Combining this with Lemma \ref{control on h_0}, we obtain
\begin{coro}\label{coro imporved estimates on h}
For any $x\in B_{5R/6}(0)$ and $r\in(0,R/7)$,
\[
\max_\alpha\|\nabla h_\alpha\|_{C^{1,\theta}(\Gamma_\alpha\cap B_r(x))}
\lesssim\varepsilon^2+A\left(r+60|\log\varepsilon|^2;x\right)^{3/2}+\varepsilon^{1/6}A\left(r+60|\log\varepsilon|^2;x\right).
\]
\end{coro}

To prove Proposition \ref{prop estimates on horizontal derivatives}, as in Section \ref{sec first order} we still estimate $\phi_i$ in $\Omega_\alpha^1$ and $\Omega_\alpha^2$ separately. To this end, we first rewrite \eqref{error equation} as
\begin{equation}\label{error equation 2}
\Delta_{\alpha,z}\phi+\partial_{zz}\phi
=W^{\prime\prime}(g_\ast)\phi +(-1)^{\alpha}g_\alpha^\prime\left[H_\alpha(y,0)+\Delta_{\alpha,0} h_\alpha(y)\right]+\mathcal{I}+E_\alpha,
\end{equation}
where
\begin{eqnarray*}
E_\alpha&=&H_\alpha(y,z)\partial_z\phi+\left[W^\prime(g_\ast+\phi)-W^\prime(g_\ast)-W^{\prime\prime}(g_\ast)\phi\right]\\
&+&(-1)^{\alpha}g_\alpha^\prime\left[H_\alpha(y,z)-H_\alpha(y,0)+\Delta_{\alpha,z} h_\alpha(y)-\Delta_{\alpha,0}h_\alpha(y)\right]\\
&-&g_\alpha^{\prime\prime}|\nabla_{\alpha,z}h_\alpha|^2+\sum_{\beta\neq\alpha} \left[(-1)^{\beta}g_\beta^\prime\mathcal{R}_{\beta,1}-
g_\beta^{\prime\prime}\mathcal{R}_{\beta,2}\right] -\sum_{\beta}\xi_\beta .
\end{eqnarray*}
The following $L^\infty$ bound on $E_\alpha$ follows from \eqref{Schauder estimates} and the calculation in Appendix \ref{sec proof of Lem 9.1}.
\begin{lem}\label{lem bound on E}
For any $x\in B_{5R/6}(0)$ and $r\in(0,R/7)$,
\begin{equation*}
\|E_\alpha\|_{L^\infty( \mathcal{M}_\alpha^0\cap B_r(x))}
\lesssim \varepsilon^2+A\left(r+50|\log\varepsilon|^2 ;x\right)^{3/2}.
\end{equation*}
\end{lem}

Differentiating \eqref{error equation 2} in $y_i$, we obtain an equation for $\phi_i:=\phi_{y_i}$, which in Fermi coordinates with respect to $\Gamma_\alpha$ reads as
\begin{equation}\label{horizontal error equation}
\Delta_{\alpha,z}\phi_i+\partial_{zz}\phi_i=W^{\prime\prime}(g_\alpha)\phi_i-(-1)^\alpha g_\alpha^\prime \left[H_{\alpha,i}(y,0)+\Delta_{\alpha,0} h_{\alpha,i}(y)\right]+\partial_{y_i}\mathcal{I}+\partial_iE_\alpha+E_i,
\end{equation}
where $H_{\alpha,i}(y,0):=\partial_{y_i} H_\alpha(y,0)$, $h_{\alpha,i}(y):= \partial_{y_i} h_\alpha$ and the remainder term
\begin{eqnarray*}
E_i&=&\underbrace{\left(\Delta_{\alpha,z}\phi_i-\partial_{y_i}\Delta_{\alpha,z}\phi\right)}_{I}+\underbrace{\left[W^{\prime\prime}(g_\ast)-W^{\prime\prime}(g_\alpha)\right]\phi_i}_{II}\\
&+&\underbrace{W^{\prime\prime\prime}(g_\ast)\phi \left[\sum_{\beta\neq\alpha} (-1)^{\beta}g_\beta^\prime\left(\frac{\partial d_\beta}{\partial y_i}-\sum_{j=1}^{n-1} \left(h_{\beta,j}\circ\Pi_\beta\right)\frac{\partial\Pi_\beta^j}{\partial y_i}\right)\right]}_{III}\\
&+&\underbrace{(-1)^\alpha g_\alpha^\prime\left[\partial_{y_i}\Delta_{\alpha,0}h_\alpha(y)-\Delta_{\alpha,0}h_{\alpha,i}(y)\right]}_{IV}\\
&-&\underbrace{g_\alpha^{\prime\prime}h_{\alpha,i}\left[H_{\alpha}(y,0)+\Delta_{\alpha,0} h_{\alpha}(y)\right]}_{V}.
\end{eqnarray*}

We have the following  $L^\infty$ bound on $E_i$.
\begin{lem}\label{bound on remainder term}
For any $x\in B_{5R/6}(0)$ and $r\in(0,R/7)$,
\[
\|E_i\|_{L^\infty(\mathcal{M}_\alpha^0\cap B_r(x))}
\lesssim \varepsilon^2+A\left(r+60|\log\varepsilon|^2 ;x\right)^{3/2}+\varepsilon^{1/3}A\left(r+60|\log\varepsilon|^2 ;x\right).
\]
\end{lem}
\begin{proof}
We estimate the five terms one by one.
  \begin{enumerate}
\item By Lemma \ref{commutator estimate}, for $(y,z)\in \mathcal{M}_\alpha^0\cap B_r(x)$,
\[
|I|
  \lesssim \varepsilon\left(|\nabla_{\alpha,0}^2\phi(y,z)|+|\nabla_{\alpha,0}\varphi(y,z)|\right)
   \lesssim \varepsilon^2+\|\phi\|_{C^{2,\theta}(B_r(x))}^2.
\]

\item For $(y,z)\in\mathcal{M}^0_\alpha\cap B_r(x)$, by Taylor expansion and Lemma \ref{lem distance ladder} we get
\begin{eqnarray*}
\Big|W^{\prime\prime}(g_\ast +\phi )-W^{\prime\prime}(g_\alpha )\Big|&\lesssim&|\phi |+\sum_{\beta\neq\alpha}g_\beta^\prime \\
&\lesssim&\|\phi\|_{L^\infty(B_r(x))}+\max_{\Gamma_\alpha\cap B_r(x)}e^{-\frac{ D_\alpha}{2}}+\varepsilon^2.
\end{eqnarray*}
Hence
\begin{eqnarray*}
\|II\|_{L^\infty(\mathcal{M}_\alpha^0\cap B_r(x))}  &\lesssim& \|\phi\|_{C^{2,\theta}(B_r(x))}\left(\|\phi\|_{C^{2,\theta}(B_r(x))}+\max_{\Gamma_\alpha\cap B_r(x)}e^{-\frac{ D_\alpha}{2}}+\varepsilon^2\right) \\
   &\lesssim&\|\phi\|_{C^{2,\theta}(B_r(x))}^2+A(r;x)^{\frac{1}{2}}\|\phi\|_{C^{2,\theta}(B_r(x))}+\varepsilon^2 .
\end{eqnarray*}

\item For $\beta\neq\alpha$, if $g_\beta^\prime\neq0$, by Lemma \ref{comparison of distances},
\begin{equation}\label{almost parallel}
  \Big|\frac{\partial d_\beta}{\partial y_i}\Big|\lesssim\varepsilon^{1/6}.
\end{equation}
By the Cauchy inequality, Lemma \ref{lem distance ladder} and Lemma \ref{control on h_0}, we obtain
\begin{eqnarray*}
\|III\|_{L^\infty(\mathcal{M}_\alpha^0\cap B_r(x))}  &\lesssim&\|\phi\|_{C^{1,\theta}(B_{r+8|\log\varepsilon|}(x))}^2+A\left(r+8|\log\varepsilon|;x\right)^2 \\
&+&\varepsilon^{1/6}\|\phi\|_{C^{1,\theta}(B_{r+8|\log\varepsilon|}(x))}.
\end{eqnarray*}

\item By Lemma \ref{commutator estimate} and Lemma \ref{control on h_0},
\[
\|IV\|_{L^\infty(\mathcal{M}_\alpha^0\cap B_r(x))}  \lesssim \varepsilon^2+\|\phi\|_{C^{2,\theta}(B_r(x))}^2+A\left(r;x\right)^2.
\]

\item By the Cauchy inequality and Lemma \ref{control on h_0},
\[
 \|V\|_{L^\infty(\mathcal{M}_\alpha^0\cap B_r(x))}  \lesssim
\|\phi\|_{C^{2,\theta}(B_r(x))}^2+\max_{\Gamma_\alpha\cap B_r(x)}e^{-2D_\alpha}+\|H_\alpha+\Delta_{\alpha,0}h_\alpha\|_{L^\infty(\Gamma_\alpha\cap B_r(x))}^2.
\]
\end{enumerate}

Putting these estimates together and applying \eqref{Schauder estimates} we conclude the proof.
\end{proof}

Finally,  the order of $\partial_{y_i}\mathcal{I}$ is increased by one due to the appearance of one more term involving horizontal derivatives of $\phi$.
\begin{lem}\label{lem upper bound on derivative of interaction}
  For any $x\in B_{6R/7}(0)$ and $r\in(0,R/8)$,
\[
\|\partial_{y_i}\mathcal{I}\|_{L^\infty(B_r(x))}
\lesssim \varepsilon^2+A\left(r+60|\log\varepsilon|^2 ;x\right)^2+\varepsilon^{1/6}A\left(r+60|\log\varepsilon|^2 ;x\right).
\]
\end{lem}
\begin{proof}
We have
\[
 \partial_{y_i}\mathcal{I} =  \sum_\beta (-1)^{\beta+1}\left[W^{\prime\prime}(g_\ast)-W^{\prime\prime}(g_\beta)\right]g_\beta^\prime\left(\frac{\partial d_\beta}{\partial y_i}-\sum_{j=1}^{n-1} h_{\beta,j}\left(\Pi_\beta(y,z)\right)\frac{\partial\Pi_\beta^j}{\partial y_i}(y,z)\right).
 \]

Let us first give an estimate on  $\left[W^{\prime\prime}(g_\ast)-W^{\prime\prime}(g_\beta)\right]g_\beta^\prime$ in $\mathcal{M}_\alpha^0$.
There are two cases.
\begin{itemize}
  \item If $\beta=\alpha$, we have
  \[
    \Big|W^{\prime\prime}(g_\ast)-W^{\prime\prime}(g_\alpha)\Big|g_\alpha^\prime \lesssim g_\alpha^\prime\sum_{\beta\neq\alpha}\left(1-g_\beta^2\right)
    \lesssim e^{-D_\alpha},\]
  where the last inequality follows the same reasoning in the proof of Lemma \ref{upper bound on interaction}.
  \item If $\beta\neq\alpha$, still as in the proof of Lemma \ref{upper bound on interaction},
  \begin{eqnarray}\label{8.6}
  \left[W^{\prime\prime}(g_\ast)-W^{\prime\prime}(g_\beta)\right]g_\beta^\prime &=& \left[W^{\prime\prime}(g_\alpha)-W^{\prime\prime}(1)+O\left(\sum_{\beta\neq\alpha}\left(1-g_\beta^2\right)\right)\right]g_\beta^\prime \\
    &=&O\left(\varepsilon^2\right)+O\left( e^{-|d_\beta(y,0)|}\right). \nonumber
  \end{eqnarray}
\end{itemize}
Hence by Lemma \ref{lem distance ladder}, we have
\begin{equation}\label{8.3}
\sum_\beta (-1)^{\beta+1}\left[W^{\prime\prime}(g_\ast)-W^{\prime\prime}(g_\beta)\right]g_\beta^\prime=O\left(\varepsilon^2\right)+O\left( e^{-D_\alpha}\right).
\end{equation}

Using Lemma \ref{control on h_0} to bound $h_{\beta,j}$, we see if $g_\beta^\prime\neq0$, \eqref{almost parallel} holds and
\begin{equation}\label{8.5}
 \Big|\sum_{j=1}^{n-1} h_{\beta,j}\left(\Pi_\beta(y,z)\right)\frac{\partial\Pi_\beta^j}{\partial y_i}(y,z)\Big|
\lesssim A\left(r;x\right) +\|\phi\|_{C^{1,\theta}(B_{r}(x))}.
\end{equation}
 Combining \eqref{8.3}, \eqref{almost parallel} and \eqref{8.5}, using the Cauchy inequality and applying \eqref{Schauder estimates} we conclude the proof.
\end{proof}

Differentiating \eqref{orthogonal condition} we obtain for any $\alpha$ and $y\in\Gamma_\alpha\cap B_r(x)$,
\begin{equation}\label{orthogonal condition 2}
\int_{-\infty}^{+\infty} \phi_i g_\alpha^\prime dz=h_{\alpha,i}(y)\int_{-\infty}^{+\infty} \phi g_\alpha^{\prime\prime}dz=O\left(\|\phi\|_{C^1(B_{r+8|\log\varepsilon|}(x))}^2+\max_{\Gamma_\alpha\cap B_r(x)}e^{-2D_\alpha}\right).
\end{equation}

In view of Lemma \ref{lem bound on E}, Lemma \ref{bound on remainder term} and Lemma \ref{lem upper bound on derivative of interaction}, combining \eqref{horizontal error equation} and the almost orthogonal condition \eqref{orthogonal condition 2}, proceeding as in Section \ref{sec first order} we get Proposition \ref{prop estimates on horizontal derivatives}. Note that although here we only have an $L^\infty$ estimate on $E_\alpha$ instead of $\partial_{y_i}E_\alpha$, we can still use the $W^{2,p}$ estimates (for a sufficiently large $p$) of the linear elliptic operator $-\Delta+1$ (in $\Omega_\alpha^2$, see  \cite[Proposition 4.1]{DKWY}) and $-\Delta+W^{\prime\prime}(g_\alpha)$ (in $\Omega_\alpha^1$) to get the $C^{1,\theta}$ bound on $\phi_i$.

Now \eqref{Toda system} can be rewritten in the following way.
\begin{coro}\label{Toda system new}
For any $x\in B_{6R/7}(0)$ and $r<R/8$, in $B_r(x)$ it holds that
\begin{eqnarray*}
H_\alpha(y,0)+\Delta_{\alpha,0}h_\alpha(y)
&=&\frac{2 A_{(-1)^{\alpha-1}}^2}{\sigma_0} e^{-d_{\alpha-1}(y,0)}-\frac{2 A_{(-1)^{\alpha}}^2}{\sigma_0}e^{d_{\alpha+1}(y,0)}+O\left(\varepsilon^2\right)\nonumber\\
&+&O\left(A\left(r+60|\log\varepsilon|^2;x\right)^{3/2}\right)+O\left(
\varepsilon^{1/6}A\left(r+60|\log\varepsilon|^2;x\right)\right).
\end{eqnarray*}
\end{coro}

\section{Reduction of the stability condition}\label{sec reduction of stability}
\setcounter{equation}{0}

In this section we show that if $u$ is a stable solutionn of the Allen-Cahn equation, then solutions to the Toda system \eqref{Toda system} constructed in Section \ref{sec Toda system} satisfies an almost stable condition.

Given a point $x\in B_{5R/6}(0)$ and $r\in(0,R/7)$, and finitely many functions $\eta_\alpha\in C_0^\infty\left(\Gamma_\alpha\cap B_r(x)\right)$, using Fermi coordinates with respect to $\Gamma_\alpha$ we define
\[\varphi_\alpha(y,z):=\eta_\alpha(y)g_\alpha^\prime(y,z).\]
In the following we will view $\eta_\alpha$ as a function defined in $B_{r+8|\log\varepsilon|}(x)$ by identifying it with $\eta_\alpha\circ\Pi_\alpha$.

Let $\varphi:=\sum_\alpha\varphi_\alpha$. By definition $\varphi\in C_0^\infty(B_{r+8|\log\varepsilon|}(x))$.
The stability condition for $u$ says that
\[\int_{B_{r+8|\log\varepsilon|}(x)}\left[|\nabla\varphi|^2+W^{\prime\prime}(u)\varphi^2\right]\geq0.\]
The purpose of this section is to rewrite this inequality as a stability condition for the Toda system \eqref{Toda system}.
\begin{prop}\label{prop reduction of stability}
If $\eta_\alpha$ are given as above, then we have
\begin{eqnarray*}
&&\sum_\alpha \int_{\Gamma_\alpha }|\nabla_{\alpha,0}\eta_\alpha|^2dA_{\alpha,0}+\mathcal{Q}(\eta)\\
&&\quad \geq \sum_\alpha\frac{2A_{(-1)^{\alpha-1}}^2}{\sigma_0} \int_{\Gamma_\alpha }e^{-d_{\alpha-1}(y)}\left[\eta_\alpha(y)-\eta_{\alpha-1}\left(\Pi_{\alpha-1}(y,0)\right)\right]^2 dA_{\alpha,0}
,
\end{eqnarray*}
where
\begin{eqnarray*}
|Q(\eta)|&\lesssim&\left[\varepsilon^{\frac{1}{4}}+A\left(r+60|\log\varepsilon|^2;x\right)^{\frac{1}{2}}\right]
\left(\sum_\alpha\int_{\Gamma_\alpha}|\nabla_{\alpha,0}\eta_\alpha |^2 dA_{\alpha,0} \right)\\
  &+&  \left[\varepsilon^2+A\left(r+60|\log\varepsilon|^2;x\right)^{\frac{3}{2}}+\varepsilon^{\frac{1}{7}}A\left(r+60|\log\varepsilon|^2;x\right)\right]
\left( \sum_\alpha\int_{\Gamma_\alpha} \eta_\alpha ^2dA_{\alpha,0}\right).
\end{eqnarray*}
\end{prop}

Since
\begin{eqnarray*}
  \int_{B_{r+8|\log\varepsilon|}(x)}\left[|\nabla\varphi|^2+W^{\prime\prime}(u)\varphi^2 \right]&=& \sum_\alpha \int_{B_{r+8|\log\varepsilon|}(x)}\left[|\nabla\varphi_\alpha|^2+W^{\prime\prime}(u)\varphi_\alpha^2\right] \\
  &+&\sum_{\alpha\neq\beta}\int_{B_{r+8|\log\varepsilon|}(x)}\left[\nabla\varphi_\alpha\cdot\nabla\varphi_\beta+W^{\prime\prime}(u)\varphi_\alpha\varphi_\beta\right],
\end{eqnarray*}
we first consider the first integrals and estimate the tangential part $\nabla_{\alpha,z}\varphi_\alpha(y,z)$ in Subsection \ref{subsec 9.1}, then the normal part $\partial_z\varphi_\alpha$ in Subsection \ref{subsec 9.2}, where an interaction term appears and it is studied in Subsection \ref{subsec 9.3}, and finally  in Subsection \ref{subsec 9.4} estimates on cross terms are given. Proposition \ref{prop reduction of stability} follows by putting these estimates together.

%As before, we need to only consider $\mathcal{M}_\alpha^0\cap B_r(x)$ for some $\alpha$. Without loss of generality we will assume $(-1)^\alpha=1$.

\subsection{The tangential part}\label{subsec 9.1}
In this subsection we prove
\begin{lem}\label{lem reduction of horizontal term}
The horizontal part has the expansion
\[
\int_{B_{r+8|\log\varepsilon|}(x)}\big|\nabla_{\alpha,z}\varphi_\alpha(y,z)\big|^2=\left[\sigma_0+O\left(\varepsilon+A(r;x)\right)\right]\int_{\Gamma_{\alpha}}|\nabla_{\alpha,0}\eta_\alpha |^2 dA_{\alpha,0}+\mathcal{Q}_\alpha(\eta_\alpha),
\]
where
\[
  |\mathcal{Q}_\alpha(\eta_\alpha)|\lesssim  \left[\varepsilon^2+A\left(r+60|\log\varepsilon|^2;x\right)^{\frac{3}{2}}+\varepsilon^{\frac{1}{6}}A\left(r+60|\log\varepsilon|^2;x\right)\right]
 \int_{\Gamma_{\alpha}} \eta_\alpha ^2 dA_{\alpha,0}.
\]
\end{lem}
\begin{proof}
A direct differentiation shows that in Fermi coordinates with respect to $\Gamma_\alpha$,
\[  \nabla_{\alpha,z}\varphi_\alpha(y,z) =g_\alpha^\prime(y,z)\nabla_{\alpha,z}\eta_\alpha(y)+(-1)^{\alpha+1}\eta_\alpha(y)g_\alpha^{\prime\prime}(y,z) \nabla_{\alpha,z} h_\alpha(y) .\]
 Hence
\begin{eqnarray*}
  &&\int_{B_{r+8|\log\varepsilon|}(x)}\big|\nabla_{\alpha,z}\varphi_\alpha(y,z)\big|^2 \\
  &=& \underbrace{\int_{-\infty}^{+\infty}\int_{\Gamma_\alpha}|\nabla_{\alpha,z}\eta_\alpha|^2\big|g_\alpha^\prime\big|^2\lambda_\alpha dy dz}_{I} +\underbrace{ \int_{-\infty}^{+\infty}\int_{\Gamma_\alpha} \eta_\alpha^2|\nabla_{\alpha,z}h_\alpha|^2\big|g_\alpha^{\prime\prime}\big|^2\lambda_\alpha dy dz}_{II} \\
   &+&\underbrace{  2(-1)^{\alpha+1} \int_{-\infty}^{+\infty}\int_{\Gamma_{\alpha,z}\cap B_r(x)} \eta_\alpha \left(\nabla_{\alpha,z}\eta_\alpha\cdot\nabla_{\alpha,z}h_\alpha\right)g_\alpha^\prime g_\alpha^{\prime\prime} \lambda_\alpha dy dz}_{III}.
\end{eqnarray*}
These three integrals are estimated in the following way.
\begin{enumerate}
  \item By \eqref{error in z 2}, we have
  \[|\nabla_{\alpha,z}\eta_\alpha|^2=\left[1+O\left(\varepsilon|z|\right)\right]|\nabla_{\alpha,0}\eta_\alpha|^2\]
  and
  \begin{equation}\label{9.1}
    \lambda_\alpha(y,z)=\lambda_\alpha(y,0)+O\left(\varepsilon|z|\right).
  \end{equation}
  Hence by the exponential decay of $g_\alpha^\prime$ at infinity, we get
  \begin{eqnarray*}
    I &=&\int_{\Gamma_\alpha}|\nabla_{\alpha,0}\eta_\alpha|^2\left(\int_{-\infty}^{+\infty}\big|g_\alpha^\prime\big|^2 dz\right)dA_{\alpha,0} \\
   &+&O\left(\varepsilon \int_{-\infty}^{+\infty}\int_{\Gamma_\alpha}|\nabla_{\alpha,0}\eta_\alpha|^2|z|\big|g_\alpha^\prime\big|^2 \lambda_\alpha dy dz\right)\\
     &=&\left[\sigma_0+O(\varepsilon)\right]\int_{\Gamma_\alpha}|\nabla_{\alpha,0}\eta_\alpha(y)|^2dA_{\alpha,0}.
  \end{eqnarray*}
  \item By \eqref{Schauder estimates} and Lemma \ref{control on h_0},
  \[II\lesssim \left[\varepsilon^2+A(r+60|\log\varepsilon|^2;x)^2\right]
  \int_{\Gamma_\alpha} \eta_\alpha^2dA_{\alpha,0}.\]

  \item Integrating by parts on $\Gamma_{\alpha,z}$ leads to
  \begin{eqnarray*}
    III&=& (-1)^\alpha\int_{-\infty}^{+\infty}\int_{\Gamma_{\alpha,z}} \eta_\alpha^2 \Delta_{\alpha,z}h_\alpha g_\alpha^\prime g_\alpha^{\prime\prime} \lambda_\alpha dy dz \\
     &+&  \int_{-\infty}^{+\infty}\int_{\Gamma_{\alpha,z}} \eta_\alpha^2 |\nabla_{\alpha,z}h_\alpha|^2\left(\big|g_\alpha^{\prime\prime}\big|^2+g_\alpha^\prime g_\alpha^{\prime\prime\prime}\right)\lambda_\alpha dy dz.
  \end{eqnarray*}
By Corollary \ref{coro imporved estimates on h} we get
\[  |III|\lesssim\left[\varepsilon^2+A\left(r+60|\log\varepsilon|^2;x\right)^{\frac{3}{2}}+\varepsilon^{\frac{1}{6}}A\left(r+60|\log\varepsilon|^2;x\right)\right]\\
  \int_{\Gamma_\alpha} \eta_\alpha^2dA_{\alpha,0}.
  \]
\end{enumerate}

Putting all of these together we finish the proof.
\end{proof}

\subsection{The normal part}\label{subsec 9.2}
As before we have
\[\partial_z\varphi_\alpha(y,z)=\eta_\alpha(y) g_\alpha^{\prime\prime}(y,z).\]
Integrating by parts in $z$ we get
\begin{eqnarray*}
&&\int_{\Gamma_\alpha}\int_{-\infty}^{+\infty}\eta_\alpha(y)^2 \big|g_\alpha^{\prime\prime}(y,z)\big|^2\lambda_\alpha(y,z)dz dy\\
&=&-\underbrace{\int_{\Gamma_\alpha}\eta_\alpha(y)^2\left[\int_{-\infty}^{+\infty}W^{\prime\prime}(g_\alpha(y,z))\big|g_\alpha^\prime(y,z)\big|^2\lambda_\alpha(y,z)
dz\right]dy}_{I}\\
&+&\underbrace{\int_{\Gamma_\alpha}\eta_\alpha(y)^2\left[\frac{1}{2}\int_{-\infty}^{+\infty}\big|g_\alpha^\prime(y,z)\big|^2\partial_{zz}\lambda_\alpha(y,z) dz-\int_{-\infty}^{+\infty} g_\alpha^\prime(y,z) \xi_\alpha^\prime(y,z)\lambda_\alpha(y,z) dz\right]dy}_{II}.
\end{eqnarray*}

By \eqref{metirc tensor} and the definition of $\lambda_\alpha$ we have
\[|\partial_{zz}\lambda_\alpha(y,z)|\lesssim |A_\alpha(y)|^2\lesssim\varepsilon^2.\]
Using this together with estimates on $\xi_\alpha$ we get
\[|II|\lesssim \varepsilon^2 \int_{\Gamma_\alpha}\eta_\alpha^2dA_{\alpha,0}.\]

It remains to rewrite the integral
\[
\int_{\Gamma_\alpha}\eta_\alpha(y)^2\int_{-\infty}^{+\infty}
\left[W^{\prime\prime}\left(u(y,z)\right)-W^{\prime\prime}\left(g_\alpha(y,z)\right)\right]\big|g_\alpha^\prime(y,z)\big|^2\lambda_\alpha(y,z)
dzdy,
\]
which will be the goal of the next subsection.

\subsection{The interaction part}\label{subsec 9.3}
Multiplying \eqref{error equation} by $\eta_\alpha^2g_\alpha^{\prime\prime}\lambda_\alpha$ and then integrating in $y$ and $z$ gives
\[I-II+III=IV+V-VI+VII-VIII,\]
where
\[I:=\int_{\Gamma_\alpha}\eta_\alpha^2\left[\int_{-\infty}^{+\infty}\Delta_{\alpha,z}\phi g_\alpha^{\prime\prime}\lambda_\alpha dz\right]dy,\]
\[II:=\int_{\Gamma_\alpha}\eta_\alpha^2\left[\int_{-\infty}^{+\infty}H_{\alpha}(y,z)\phi_z g_\alpha^{\prime\prime}\lambda_\alpha dz\right]dy,\]
\[III:=\int_{\Gamma_\alpha}\eta_\alpha^2\left[\int_{-\infty}^{+\infty}\phi_{zz} g_\alpha^{\prime\prime}\lambda_\alpha dz\right]dy,\]
\[IV:=\int_{\Gamma_\alpha}\eta_\alpha^2\left[\int_{-\infty}^{+\infty}\left(W^\prime(u)-\sum_\beta W^\prime(g_\beta)\right) g_\alpha^{\prime\prime}\lambda_\alpha dz\right]dy,\]
\[V:=(-1)^\alpha\int_{\Gamma_\alpha}\eta_\alpha^2\left[\int_{-\infty}^{+\infty}\left(H_\alpha(y,z)+\Delta_{\alpha,z}h_\alpha(y)\right) g_\alpha^\prime g_\alpha^{\prime\prime}\lambda_\alpha dz\right]dy,\]
\[VI:=\int_{\Gamma_\alpha}\eta_\alpha^2\left[\int_{-\infty}^{+\infty}\big|g_\alpha^{\prime\prime}\big||\nabla_{\alpha,z}h_\alpha|^2\lambda_\alpha dz\right]dy,\]
\[VII:=\sum_{\beta\neq\alpha}\int_{\Gamma_\alpha}\eta_\alpha^2\left[\int_{-\infty}^{+\infty}\left[(-1)^{\beta}g_\beta^\prime\mathcal{R}_{\beta,1}-
g_\beta^{\prime\prime}\mathcal{R}_{\beta,2}\right] g_\alpha^{\prime\prime}\lambda_\alpha dz\right]dy,\]
\[VIII:=\sum_{\beta\neq\alpha}\int_{\Gamma_\alpha}\eta_\alpha^2\left[\int_{-\infty}^{+\infty}\sum_{\beta}\xi_\beta g_\alpha^{\prime\prime}\lambda_\alpha dz\right]dy,\]

We need to estimate each of them.
\begin{enumerate}
\item By Proposition \ref{prop estimates on horizontal derivatives},
\[|I|\lesssim \left[ \varepsilon^2+A\left(r+60|\log\varepsilon|^2;x\right)^{\frac{3}{2}}+\varepsilon^{\frac{1}{6}}A\left(r+60|\log\varepsilon|^2;x\right)\right]
\int_{\Gamma_\alpha}\eta_\alpha^2dA_{\alpha,0}.\]

\item
Because $H_\alpha=O(\varepsilon)$, by \eqref{Schauder estimates},
\begin{eqnarray*}
  |II| &\lesssim& \varepsilon\|\phi\|_{C^{2,\theta}(B_{r+8|\log\varepsilon|}(x))}\int_{\Gamma_\alpha}\eta_\alpha^2dA_{\alpha,0} \\
   &\lesssim&\left[ \varepsilon^2+A\left(r+60|\log\varepsilon|^2;x\right)^2\right]
\int_{\Gamma_\alpha}\eta_\alpha^2dA_{\alpha,0}.
\end{eqnarray*}

\item Integrating by parts in $z$ gives
\begin{eqnarray*}
  III &=& -\int_{\Gamma_\alpha}\eta_\alpha^2\left[\int_{-\infty}^{+\infty}\phi_{z} g_\alpha^{\prime\prime\prime}\lambda_\alpha dz\right]dy-\int_{\Gamma_\alpha}\eta_\alpha^2\left[\int_{-\infty}^{+\infty}\phi_{z} g_\alpha^{\prime\prime}\partial_z\lambda_\alpha dz\right]dy\\
    &=&-\int_{\Gamma_\alpha}\eta_\alpha^2\left[\int_{-\infty}^{+\infty}W^{\prime\prime}(g_\alpha)g_\alpha^{\prime}\phi_{z} \lambda_\alpha dz\right]dy\\
    &-&\int_{\Gamma_\alpha}\eta_\alpha^2\left[\int_{-\infty}^{+\infty}\phi_{z}\xi_\alpha^\prime \lambda_\alpha dz\right]dy-\int_{\Gamma_\alpha}\eta_\alpha^2\left[\int_{-\infty}^{+\infty}\phi_{z} g_\alpha^{\prime\prime}\partial_z\lambda_\alpha dz\right]dy.
\end{eqnarray*}
Because $\xi_\alpha=O\left(\varepsilon^3\right)$, the length $|\{z:\xi_\alpha^\prime(\cdot,z)|\neq 0\}|\lesssim|\log\varepsilon|$ and $\partial_z\lambda_\alpha=O(\varepsilon)$ (see
 \eqref{metirc tensor} and the definition of $\lambda_\alpha$), using \eqref{Schauder estimates} and reasoning as in the previous case we obtain
\begin{eqnarray}\label{9.2}
  III     &=&-\int_{\Gamma_\alpha}\eta_\alpha^2\left[\int_{-\infty}^{+\infty}W^{\prime\prime}(g_\alpha)g_\alpha^{\prime}\phi_{z} \lambda_\alpha dz\right]dy\\
    &+&O\left( \varepsilon^2+A\left(r+60|\log\varepsilon|^2;x\right)^2\right)
\int_{\Gamma_\alpha}\eta_\alpha^2dA_{\alpha,0}.\nonumber
\end{eqnarray}

\item Integrating by parts in $z$ leads to
\begin{eqnarray*}
  IV&=&-\int_{\Gamma_\alpha}\eta_\alpha^2\int_{-\infty}^{+\infty}W^{\prime\prime}(u)  g_\alpha^{\prime}\phi_z\lambda_\alpha dzdy \\
  &-&\int_{\Gamma_\alpha}\eta_\alpha^2\int_{-\infty}^{+\infty}\left[W^{\prime\prime}(u)-W^{\prime\prime}(g_\alpha)\right]\big|g_\alpha^{\prime}\big|^2\lambda_\alpha dzdy \\
  &-&\sum_{\beta\neq\alpha}\underbrace{\int_{\Gamma_\alpha}\eta_\alpha^2\int_{-\infty}^{+\infty}\left[W^{\prime\prime}(u)-W^{\prime\prime}(g_\beta)\right]g_\alpha^{\prime}g_\beta^\prime\left(\frac{\partial d_\beta}{\partial z}-\frac{\partial}{\partial z}\left(h_\beta\circ\Pi_\beta\right)\right) \lambda_\alpha dzdy }_{IX_\beta}\\
  &-&\underbrace{\int_{\Gamma_\alpha}\eta_\alpha^2\int_{-\infty}^{+\infty}\left(W^{ \prime}(u)-\sum_\beta W^\prime(g_\beta)\right) g_\alpha^{\prime}\partial_z\lambda_\alpha dzdy}_{X}.
\end{eqnarray*}
The first integral  cancel with III (see \eqref{9.2}) up to a higher order term. The second integral is the one we want to rewrite in Subsection \ref{subsec 9.2}.

First let us estimate the term $X$. By Taylor expansion we have
\begin{equation}\label{9.3}
  W^{ \prime}(u)-\sum_\beta W^\prime(g_\beta)=\mathcal{I}+O\left(\phi\right)
\end{equation}
Then by \eqref{Schauder estimates}, Lemma \ref{upper bound on interaction} and the fact that $\partial_z\lambda_\alpha=O(\varepsilon)$, we get
\begin{equation}\label{9.4}
|X|\lesssim \left[ \varepsilon^2+A\left(r+60|\log\varepsilon|^2;x\right)^2\right]
\int_{\Gamma_\alpha}\eta_\alpha^2dA_{\alpha,0}.
\end{equation}

It remains to rewrite the integral $IX_\beta$. First replace $W^{\prime\prime}(u)$ by $W^{\prime\prime}(g_\ast)$. This introduces an error bounded by
\begin{eqnarray}\label{error 1}
   && O\left(\|\phi\|_{C^{2,\theta}(B_{r+8|\log\varepsilon|}(x))}\right)  \int_{\Gamma_\alpha}\eta_\alpha^2\int_{-\infty}^{+\infty}
  g_\alpha^{\prime}g_\beta^\prime \lambda_\alpha dzdy \\
  &\lesssim& \left[ \varepsilon^2+A\left(r+60|\log\varepsilon|^2;x\right)^{3/2}\right]
\int_{\Gamma_\alpha}\eta_\alpha^2dA_{\alpha,0}. \nonumber
\end{eqnarray}

Next, if $g_\alpha^\prime\neq0$ and $g_\beta^\prime\neq 0$ at the same time, by Lemma \ref{comparison of distances},
\[\frac{\partial d_\beta}{\partial z}=1+O\left(\varepsilon^{1/3}\right),\]
\begin{equation}\label{9.5}
  d_\beta(y,z)=d_\beta(y,0)\pm z+O\left(\varepsilon^{1/3}\right).
\end{equation}
Replace $\frac{\partial d_\beta}{\partial z}$ by $1$ and throw  away the term involving $h_\beta$. This  introduces another error controlled by
\begin{eqnarray}\label{error 2}
  && \left[\varepsilon^{\frac{1}{3}}+A(r+60|\log\varepsilon|^2;x)\right] \int_{\Gamma_\alpha}\eta_\alpha^2\int_{-\infty}^{+\infty}\Big|W^{\prime\prime}(g_\ast)-W^{\prime\prime}(g_\beta)\Big|
  g_\alpha^{\prime}g_\beta^\prime \lambda_\alpha dzdy  \nonumber\\
  &\lesssim&\left[\varepsilon^{\frac{1}{3}}+A(r+60|\log\varepsilon|^2;x)\right]\int_{\Gamma_\alpha}\eta_\alpha^2
  \int_{-\infty}^{+\infty}
  g_\alpha^{\prime}g_\beta^\prime \sum_{\gamma}g_\gamma^\prime dzdy\\
  &\lesssim&\left[\varepsilon^{\frac{1}{3}}A(r+60|\log\varepsilon|^2;x)+A(r+60|\log\varepsilon|^2;x)^2\right]\int_{\Gamma_\alpha}\eta_\alpha^2
 dA_{\alpha,0}. \nonumber
\end{eqnarray}

Finally, in order to determine
\[XI_\beta:=\int_{-\infty}^{+\infty}\left[W^{\prime\prime}(g_\ast)-W^{\prime\prime}(g_\beta)\right]  g_\alpha^{\prime}g_\beta^\prime \lambda_\alpha dz,\]
by \eqref{9.5} we can assume for each $\beta\neq\alpha$,
\[g_\beta(y,z)=\bar{g}\left((-1)^\beta\left(d_\beta(y,0)\pm z\right)\right).\]
We can also replace $\lambda_\alpha(y,z)$ by $\lambda_\alpha(y,0)$. These two procedures lead to a third error, which can be estimated as in \eqref{error 2}.

Then proceeding as in Case (8) in Appendix \ref{sec proof of Lemma 5.1} and applying Lemma \ref{lem form of interaction} we get
\begin{eqnarray}\label{9.7}
XI_\beta
&=&-2A_{(-1)^\beta}^2e^{-|d_\beta(y,0)|}+O\left(e^{-\frac{4}{3}|d_\beta(y,0)|}\right) \\
&+&  O\left(\varepsilon^{\frac{1}{3}}|\log\varepsilon|e^{- |d_\beta(y,0)|}\right)+O\left(A(r+60|\log\varepsilon|^2;x)|d_\beta(y,0)|e^{-|d_\beta(y,0)|}\right)\nonumber.
\end{eqnarray}

If $|\beta-\alpha|\geq 2$ and $|d_\beta(y,0)|\leq 2|\log\varepsilon|$,  by Lemma \ref{comparison of distances} we have
\[|d_\beta(y,0)|\geq 2\min_\alpha\min_{\Gamma_\alpha\cap B_r(x)}D_\alpha+C|\beta-\alpha|-C.\]
In particular, for any $ |\beta-\alpha|\geq 2$,
\begin{equation}\label{9.8}
|XI_\beta|\lesssim e^{-C|\beta-\alpha|}A(r;x)^2.
\end{equation}
Summing \eqref{9.7} in $\beta$ and applying Lemma \ref{lem distance ladder}, \eqref{error 1}, \eqref{error 2} and \eqref{9.8} we get
\begin{eqnarray}\label{9.9}
\sum_{\beta\neq\alpha}IX_\beta&=&
-2\int_{\Gamma_\alpha}\eta_\alpha(y)^2\left[A_{(-1)^{\alpha-1}}^2e^{-d_{\alpha-1}(y,0)}+A_{(-1)^{\alpha}}^2e^{d_{\alpha+1}(y,0)}\right]dA_{\alpha,0} \nonumber\\
&+&O\left(\varepsilon^2+\varepsilon^{\frac{1}{4}}A\left(r+60|\log\varepsilon|^2;x\right)+A\left(r+60|\log\varepsilon|^2;x\right)^{\frac{4}{3}}\right)
\int_{\Gamma_\alpha}\eta_\alpha^2dA_{\alpha,0}.
 \end{eqnarray}

Combining \eqref{9.4} and \eqref{9.8} we get
\begin{eqnarray*}%\label{interaction part}
  IV&=&-\int_{\Gamma_\alpha}\eta_\alpha^2\int_{-\infty}^{+\infty}W^{\prime\prime}(u)  g_\alpha^{\prime}\phi_z\lambda_\alpha dzdy \nonumber\\
  &-&\int_{\Gamma_\alpha}\eta_\alpha^2\int_{-\infty}^{+\infty}\left[W^{\prime\prime}(u)-W^{\prime\prime}(g_\alpha)\right]\big|g_\alpha^{\prime}\big|^2\lambda_\alpha dzdy \\
  &-&2\int_{\Gamma_\alpha}\eta_\alpha^2\left[A_{(-1)^{\alpha-1}}^2e^{-d_{\alpha-1}}+A_{(-1)^{\alpha}}^2e^{d_{\alpha+1}}\right]dA_{\alpha,0} \nonumber\\
 &+&O\left(\varepsilon^2+\varepsilon^{\frac{1}{4}}A\left(r+60|\log\varepsilon|^2;x\right)+A\left(r+60|\log\varepsilon|^2;x\right)^{\frac{4}{3}}\right)
 \int_{\Gamma_\alpha}\eta_\alpha^2dA_{\alpha,0}.  \nonumber
\end{eqnarray*}

\item Integrating by parts in $z$ leads to
\begin{eqnarray*}
  V&=&\frac{(-1)^{\alpha+1}}{2}\int_{\Gamma_\alpha}\eta_\alpha^2\left[\int_{-\infty}^{+\infty}\frac{\partial}{\partial z}\left(H_\alpha(y,z)+\Delta_{\alpha,z}h_\alpha(y)\right) \big|g_\alpha^\prime\big|^2\lambda_\alpha dz\right]dy \\
   &+&\frac{(-1)^{\alpha+1}}{2}\int_{\Gamma_\alpha}\eta_\alpha^2\left[\int_{-\infty}^{+\infty}\left(H_\alpha(y,z)+\Delta_{\alpha,z}h_\alpha(y)\right) \big|g_\alpha^\prime\big|^2\partial_z\lambda_\alpha dz\right]dy.
\end{eqnarray*}

By \eqref{A(z)},
\[\frac{\partial}{\partial z}H_\alpha(y,z)=O\left(\varepsilon^2\right).\]
By \eqref{metirc tensor}, \eqref{Laplacian}, \eqref{Schauder estimates} and Lemma \ref{control on h_0},
\begin{eqnarray*}
\Big|\frac{\partial}{\partial z}\Delta_{\alpha,z}h_\alpha(y)\Big|&\lesssim&\varepsilon \left(\big|\nabla_{\alpha,0}^2h_\alpha(y)\big|+\big|\nabla_{\alpha,0}h_\alpha(y)\big|\right)\\
&\lesssim& \varepsilon^2+A\left(r+60|\log\varepsilon|^2;x\right)^2.
\end{eqnarray*}
Finally, because $\partial_z\lambda_\alpha=O(\varepsilon)$, we have
\begin{eqnarray*}
  &&\int_{-\infty}^{+\infty}\left(H_\alpha(y,z)+\Delta_{\alpha,z}h_\alpha(y)\right) \big|g_\alpha^\prime\big|^2\partial_z\lambda_\alpha dz \\
   &=&\int_{-\infty}^{+\infty}\left(H_\alpha(y,0)+\Delta_{\alpha,0}h_\alpha(y)+O\left(\varepsilon^2|z|\right)+O\left(\varepsilon|z|\right)\right) \big|g_\alpha^\prime\big|^2\partial_z\lambda_\alpha dz \\
   &=& O\left(\varepsilon^2+A\left(r+60|\log\varepsilon|^2;x\right)^2\right). \quad (\mbox{By \eqref{H eqn}})
\end{eqnarray*}

Combining these three estimates we see
\[|V|\lesssim \left[\varepsilon^2+A\left(r+60|\log\varepsilon|^2;x\right)^2\right]\int_{\Gamma_\alpha}\eta_\alpha^2dA_{\alpha,0}.\]

\item By Lemma \ref{control on h_0} and \eqref{Schauder estimates}, we get
\[|VI|\lesssim \left[\varepsilon^2+A\left(r+60|\log\varepsilon|^2;x\right)^{\frac{3}{2}}\right]\int_{\Gamma_\alpha}\eta_\alpha^2dA_{\alpha,0}.\]

\item Following the proof of \eqref{B.2}, we get a similar estimate on
\[\int_{-\infty}^{+\infty}\left[(-1)^{\beta}g_\beta^\prime\mathcal{R}_{\beta,1}-
g_\beta^{\prime\prime}\mathcal{R}_{\beta,2}\right] g_\alpha^{\prime\prime}\lambda_\alpha dz,\]
which then gives
\[
|VII|\lesssim \left[\varepsilon^2+A\left(r+60|\log\varepsilon|^2;x\right)^{\frac{3}{2}}\right]\int_{\Gamma_\alpha}\eta_\alpha^2dA_{\alpha,0}.
\]

\item Finally, by the definition of $\xi_\beta$ and \eqref{lem distance ladder}, and because $\{g_\alpha^\prime(y,\cdot)\neq 0\}$ has length at most $16|\log\varepsilon|$, we obtain
\[
 |VIII|\lesssim \varepsilon^3|\log\varepsilon| \int_{\Gamma_\alpha}\eta_\alpha^2dA_{\alpha,0}
   \lesssim\varepsilon^2 \int_{\Gamma_\alpha}\eta_\alpha^2dA_{\alpha,0}.
\]
\end{enumerate}

\medskip

Combining all of these estimates together, we obtain
\begin{eqnarray*}
&&\int_{\Gamma_\alpha}\eta_\alpha^2\int_{-\infty}^{+\infty}
\left[W^{\prime\prime}\left(u\right)-W^{\prime\prime}\left(g_\alpha\right)\right]\big|g_\alpha^\prime\big|^2\lambda_\alpha
dzdy\\
&=&\int_{\Gamma_\alpha}\eta_\alpha^2\int_{-\infty}^{+\infty}\left[W^{\prime\prime}(u)-W^{\prime\prime}(g_\alpha)\right]  g_\alpha^{\prime}\phi_z\lambda_\alpha dzdy \\
    &-&2\int_{\Gamma_\alpha}\eta_\alpha^2\left[A_{(-1)^{\alpha-1}}^2e^{-d_{\alpha-1}}+A_{(-1)^{\alpha}}^2e^{d_{\alpha+1}}\right]dA_{\alpha,0} \nonumber\\
 &+&O\left(\varepsilon^2+\varepsilon^{\frac{1}{4}}A\left(r+60|\log\varepsilon|^2;x\right)+A\left(r+60|\log\varepsilon|^2;x\right)^{\frac{4}{3}}\right)
 \int_{\Gamma_\alpha}\eta_\alpha^2dA_{\alpha,0}.
\end{eqnarray*}

The first integral in the right hand side of this equation is estimated in the following way. As in \eqref{9.3} and Lemma \ref{upper bound on interaction},
\begin{equation}\label{error estimate}
  \big|W^{\prime\prime}(u)-W^{\prime\prime}(g_\alpha)\big|\lesssim |\phi|+\sum_{\beta\neq\alpha}\left(1-g_\beta^2\right).
\end{equation}
Then arguing as in the proof of \eqref{B.1} and using \eqref{Schauder estimates} to estimate $\phi$, we see this integral is also bounded by
\[O\left(\varepsilon^2+A\left(r+60|\log\varepsilon|^2;x\right)^{\frac{4}{3}}\right)
 \int_{\Gamma_\alpha}\eta_\alpha^2dA_{\alpha,0}.\]
Therefore we arrive at the following form
\begin{eqnarray}\label{interaction part 2}
&&\int_{\Gamma_\alpha}\eta_\alpha^2\int_{-\infty}^{+\infty}
\left[W^{\prime\prime}\left(u\right)-W^{\prime\prime}\left(g_\alpha\right)\right]\big|g_\alpha^\prime\big|^2\lambda_\alpha\nonumber
dzdy\\
&=&-2\int_{\Gamma_\alpha}\eta_\alpha(y)^2\left[A_{(-1)^{\alpha-1}}^2e^{-d_{\alpha-1}}+A_{(-1)^{\alpha}}^2e^{d_{\alpha+1}}\right]dA_{\alpha,0}\\
 &+&O\left(\varepsilon^2+\varepsilon^{\frac{1}{4}}A\left(r+60|\log\varepsilon|^2;x\right)+A\left(r+60|\log\varepsilon|^2;x\right)^{\frac{4}{3}}\right)
 \int_{\Gamma_\alpha}\eta_\alpha^2dA_{\alpha,0}.\nonumber
\end{eqnarray}
This completes the reduction of the vertical part.

\subsection{Cross terms}\label{subsec 9.4}
In this section we estimate the integral of cross terms,
\[\sum_{\alpha\neq\beta}\int_{B_{r+8|\log\varepsilon|}(x)}\left[\nabla\varphi_\alpha\cdot\nabla\varphi_\beta+W^{\prime\prime}(u)\varphi_\alpha\varphi_\beta\right].\]
%Since these integrands exhibit similar patterns on each $\mathcal{M}_\alpha^0$, we will be only concerned with one of these domains.

\begin{lem}\label{lem 9.3}
For any $\alpha\neq\beta$, we have
\begin{eqnarray*}
  &&\Big|\int_{B_{r+8|\log\varepsilon|}(x)}\nabla_{\alpha,z}\varphi_\alpha\cdot\nabla_{\alpha,z}\varphi_\beta\Big|\\
   &\lesssim&  A(r;x)^{\frac{1}{2}}
   \left[\int_{\Gamma_{\alpha}}|\nabla_{\alpha,0}\eta_\alpha|^2dA_{\alpha,0}+
     \int_{\Gamma_{\beta}}|\nabla_{\beta,0}\eta_\beta|^2dA_{\beta,0}\right]\\
     &+&\left(\varepsilon^2+\varepsilon^{\frac{1}{7}}A(r;x)+A(r;x)^{\frac{3}{2}}\right)\left[
     \int_{\Gamma_{\alpha}}\eta_\alpha^2dA_{\alpha,0}
     +\int_{\Gamma_{\beta}}\eta_\beta^2dA_{\beta,0}\right].
\end{eqnarray*}
\end{lem}
\begin{proof}
 In Fermi coordinates with respect to $\Gamma_\alpha$,   write
\begin{eqnarray*}
\nabla_{\alpha,z}\varphi_\alpha\cdot\nabla_{\alpha,z}\varphi_\beta&=&\left[g_\alpha^\prime\nabla_{\alpha,z}\eta_\alpha-(-1)^\alpha \eta_\alpha g_\alpha^{\prime\prime}\nabla_{\alpha,z}h_\alpha\right]\\
  && \times\left[g_\beta^\prime\nabla_{\alpha,z}\eta_\beta+(-1)^\beta\eta_\beta
  g_\beta^{\prime\prime}\left(\nabla_{\alpha,z}d_\beta-\nabla_{\beta,0}h_\beta\circ D_{\alpha,z}\Pi_\beta\right)\right]
\end{eqnarray*}

In the following we assume $\beta<\alpha$. $\Gamma_\alpha$ and $\Gamma_\beta$ divide $B_{r+8|\log\varepsilon|}(x)$ into three domains, $\Omega_{\alpha\beta}^0$ the one between them, $\Omega_{\alpha\beta}^+$ the one above $\Gamma_\alpha$ and $\Omega_{\alpha\beta}^-$ the one below $\Gamma_\beta$.

{\bf Case 1.} In $\Omega_{\alpha\beta}^0$, we have
 \[g_\alpha^\prime(y,z)g_\beta^\prime(y,z)+g_\alpha^\prime(y,z) \big|g_\beta^{\prime\prime}(y,z)\big|+g_\beta^\prime(y,z)\big|g_\alpha^{\prime\prime}(y,z)\big|+
 \big|g_\alpha^{\prime\prime}(y,z)g_\beta^{\prime\prime}(y,z)\big|\lesssim  e^{-d_\beta(y,0)}.\]
Using Lemma \ref{control on h_0} and \eqref{Schauder estimates} to estimate terms involving $h$, using Lemma \ref{comparison of distances} to estimate $\nabla_{\alpha,z}d_\beta$ (note that if $g_\alpha^\prime(y,z)g_\beta^\prime(y,z)\neq 0$, then $d_\beta(y,0)\leq 16|\log\varepsilon|$), we get
\begin{eqnarray}\label{9.4.2}
  &&\big|\nabla_{\alpha,z}\varphi_\alpha(y,z)\cdot\nabla_{\alpha,z}\varphi_\beta(y,z)\big| \nonumber\\
   &\lesssim& e^{-d_\beta(y,0)} |\nabla_{\alpha,z}\eta_\alpha||\nabla_{\alpha,z}\eta_\beta | \nonumber \\
  &+& \left[\varepsilon^{\frac{1}{6}}A\left(r+60|\log\varepsilon|^2;x\right)+A\left(r+60|\log\varepsilon|^2;x\right)^2\right]e^{-d_\beta(y,0)} \eta_\alpha \eta_\beta  \\
  &+& \left[\varepsilon^{\frac{1}{6}}+A\left(r+60|\log\varepsilon|^2;x\right)\right]e^{-d_\beta(y,0)}  \eta_\alpha|\nabla_{\alpha,z}\eta_\beta | \nonumber \\
  &+& \left[\varepsilon^{\frac{1}{6}}+A\left(r+60|\log\varepsilon|^2;x\right)\right]e^{-d_\beta(y,0)}   \eta_\beta  |\nabla_{\alpha,z}\eta_\alpha|. \nonumber
\end{eqnarray}

{\bf Subcase 1.1.} Here we show how to estimate the integral of the first term in the right hand side of \eqref{9.4.2}. First by  Lemma \ref{biLip of projection operator} we can replace $\nabla_{\alpha,z}\eta_\alpha$ by $\nabla_{\alpha,0}\eta_\alpha$.  Then by Lemma \ref{comparison of distances} and Cauchy inequality we obtain
\[
 \int_{\Omega_{\alpha\beta}^0 }e^{-d_\beta(y,0)} |\nabla_{\alpha,0}\eta_\alpha||\nabla_{\alpha,z}\eta_\beta |
   \lesssim  \int_{\Omega_{\alpha\beta}^0 }e^{-d_\beta(y,0)} |\nabla_{\alpha,0}\eta_\alpha|^2+\int_{\Omega_{\alpha\beta}^0 }e^{-d_\beta(y,0)} |\nabla_{\alpha,z}\eta_\beta |^2.
\]
Since $\Omega_{\alpha\beta}^0\subset \{(y,z):|z|<2 d_\beta(y,0) \}$, the first integral is controlled by
\begin{eqnarray*}
   \int_{\Gamma_{\alpha}}\left(\int_{- 2 d_\beta(y,0) }^{0}e^{-d_\beta(y,0)}dz\right)|\nabla_{\alpha,0}\eta_\alpha|^2dA_{\alpha,0}
   \lesssim  \left(\max_{ \Gamma_\alpha\cap B_r(x)} d_\beta  e^{- d_\beta }\right) \int_{\Gamma_{\alpha}} |\nabla_{\alpha,0}\eta_\alpha |^2 dA_{\alpha,0} .
\end{eqnarray*}
The second one can be estimated in the same way.

{\bf Subcase 1.2.}
To estimate the integral of $\varepsilon^{1/6}e^{-d_\beta(y,0)}\eta_\alpha\eta_\beta$, the above method needs a revision. Here we note that the domain of integration can be restricted to $\{|z|<8|\log\varepsilon|\}\cap \{|d_\beta|<8|\log\varepsilon|\}$, because otherwise $g_\alpha^\prime$ or $ g_\beta^\prime=0$.
Hence we have
\begin{eqnarray}\label{9.4.2.1}
  &&\varepsilon^{\frac{1}{6}}\int_{\Omega_{\alpha\beta}^0\cap\{|z|<8|\log\varepsilon|\}\cap\{|d_\beta(y,z)|<8|\log\varepsilon|\}}e^{-d_\beta(y,0)} \eta_\alpha \eta_\beta  \\
  &\lesssim& \varepsilon^{\frac{1}{6}}\int_{\Omega_{\alpha\beta}^0\cap\{|z|<8|\log\varepsilon|\}}e^{-d_\beta(y,0)} \eta_\alpha^2+\varepsilon^{\frac{1}{6}}\int_{\Omega_{\alpha\beta}^0\cap\{d_\beta(y,z)< 8|\log\varepsilon|\}}e^{-d_\beta(y,0)}  \eta_\beta^2. \nonumber
\end{eqnarray}
The first integral is rewritten as
\begin{eqnarray*}
  \varepsilon^{\frac{1}{6}} \int_{\Gamma_{\alpha}}\left(\int_{-8|\log\varepsilon|}^{8|\log\varepsilon|}e^{-d_\beta(y,0)}dz\right) \eta_\alpha^2dA_{\alpha,0}
  &\lesssim& \varepsilon^{\frac{1}{6}}|\log\varepsilon|\left(\max_{y\in\Gamma_\alpha\cap B_r(x)}e^{- d_\beta }\right) \int_{\Gamma_{\alpha}} \eta_\alpha^2 dA_{\alpha,0}\\
 &\lesssim& \varepsilon^{\frac{1}{7}}A(r;x)\int_{\Gamma_{\alpha}} \eta_\alpha^2 dA_{\alpha,0}.
\end{eqnarray*}
The second integral in \eqref{9.4.2.1} and integrals involving $\eta_\alpha|\nabla_{\alpha,z}\eta_\beta|$ as well as
$\eta_\beta  |\nabla_{\alpha,z}\eta_\alpha|$  can be estimated in a similar way.

{\bf Case 2.} In $\Omega_{\alpha\beta}^+$, we have
 \[g_\alpha^\prime(y,z)g_\beta^\prime(y,z)+g_\alpha^\prime(y,z) \big|g_\beta^{\prime\prime}(y,z)\big|+g_\beta^\prime(y,z)\big|g_\alpha^{\prime\prime}(y,z)\big|+
 \big|g_\alpha^{\prime\prime}(y,z)g_\beta^{\prime\prime}(y,z)\big|\lesssim  e^{-d_\beta(y,0)-z}.\]
Similar to the above case, we have a bound on
$
  \big|\nabla_{\alpha,z}\varphi_\alpha(y,z)\cdot\nabla_{\alpha,z}\varphi_\beta(y,z)\big| $. By noting that they are nonzero only in the $8|\log\varepsilon|$ neighborhood of $\Gamma_\alpha\cup\Gamma_\beta$, we obtain
\begin{eqnarray*}
  &&\int_{\Omega_{\alpha\beta}^+\cap\{z<8|\log\varepsilon|\}}e^{-d_\beta(y,0)-z} |\nabla_{\alpha,0}\eta_\alpha||\nabla_{\beta,z}\eta_\beta | \\
  &\lesssim& \left(\int_{\Omega_{\alpha\beta}^+\cap\{z<8|\log\varepsilon|\}}e^{-2d_\beta(y,0)-z} |\nabla_{\alpha,0}\eta_\alpha|^2\right)^{\frac{1}{2}}\left(\int_{\Omega_{\alpha\beta}^+\cap\{z<8|\log\varepsilon|\}}e^{-z}|\nabla_{\beta,z}\eta_\beta |^2\right)^{\frac{1}{2}}\\
  &\lesssim& \left(\int_{\Gamma_{\alpha}}\int_{0}^{8|\log\varepsilon|}e^{-2d_\beta(y,0)-z} |\nabla_{\alpha,0}\eta_\alpha(y)|^2dzdA_{\alpha,0}\right)^{\frac{1}{2}}\\
&&\times  \left(\int_{0}^{8|\log\varepsilon|}\int_{\Gamma_{\alpha}}e^{-z}|\nabla_{\beta,z}\eta_\beta |^2dA_{\alpha,z}dz\right)^{\frac{1}{2}}\\
   &\lesssim& \left(\max_{y\in\Gamma_\alpha\cap B_r(x)}
   e^{- d_\beta(y,0) }\right)\left(\int_{\Gamma_{\alpha}} |\nabla_{\alpha,0}\eta_\alpha(y)|^2 dA_{\alpha,0}\right)^{\frac{1}{2}}
  \left( \int_{\Gamma_{\beta}}|\nabla_{\beta,0} \eta_\beta |^2dA_{\beta,0} \right)^{\frac{1}{2}}.
\end{eqnarray*}
Other terms and integrals in $\Omega_{\alpha\beta}^-$ can be estimated in the same way and we conclude the proof.
\end{proof}

\begin{lem}\label{lem 9.4}
For any $\alpha\neq\beta$, we have
\[\int_{ B_{r+8|\log\varepsilon|}(x)}\partial_z\varphi_\alpha\partial_z\varphi_\beta
   =- \int_{\Gamma_\alpha}
  \int_{-\infty}^{+\infty}\eta_\alpha \eta_\beta  W^{\prime\prime}(g_\beta)g_\alpha^{\prime} g_\beta^{\prime} \lambda_\alpha dzdy+\mathcal{Q}_{\alpha,\beta}(\eta),\]
  where
\begin{eqnarray*}
 |\mathcal{Q}_{\alpha,\beta}(\eta)|  &\lesssim& \left[\varepsilon^{\frac{1}{7}}A(r;x)+A(r;x)^{\frac{3}{2}}\right] \left[\int_{\Gamma_\alpha}
\eta_\alpha ^2dA_{\alpha,0}+\int_{\Gamma_{\beta,0}}
\eta_\beta ^2dA_{\beta,0}\right]\\
&+& \varepsilon^{\frac{1}{7}}A(r;x) \int_{\Gamma_{\beta,0}}
|\nabla_{\beta,0}\eta_\beta |^2dA_{\beta,0}.
\end{eqnarray*}
\end{lem}
\begin{proof}
We have
\begin{equation}\label{9.10}
\partial_z\varphi_\alpha\partial_z\varphi_\beta
= \eta_\alpha\eta_\beta  g_\alpha^{\prime\prime}g_\beta^{\prime\prime}\left[\frac{\partial d_\beta}{\partial z} -(-1)^\beta\nabla h_\beta\cdot\frac{\partial\Pi_\beta}{\partial z}\right] + \eta_\alpha g_\alpha^{\prime\prime} g_\beta^{\prime}\left(\nabla_{\beta,0}\eta_\beta  \cdot\frac{\partial\Pi_\beta}{\partial z}\right) .
\end{equation}
By Lemma \ref{comparison of distances}, if $g_\alpha^{\prime\prime}g_\beta^\prime\neq 0$ or $g_\alpha^{\prime\prime}g_\beta^{\prime\prime}\neq 0$,  then
\begin{equation}\label{9.4.1}
\Big|\frac{\partial d_\beta}{\partial z}-1\Big|+\Big|\frac{\partial\Pi_\beta}{\partial z}\Big|\lesssim\varepsilon^{1/6},
\end{equation}
We can proceed as in the proof of Lemma \ref{lem 9.3} to estimate the integral of
\[\eta_\alpha g_\alpha^{\prime\prime}\eta_\beta  g_\beta^{\prime\prime}\left(\frac{\partial d_\beta}{\partial z}-1 -(-1)^\beta\nabla h_\beta\cdot\frac{\partial\Pi_\beta}{\partial z}\right)+\eta_\alpha g_\alpha^{\prime\prime} g_\beta^{\prime}\left(\nabla_{\beta,0}\eta_\beta  \cdot\frac{\partial\Pi_\beta}{\partial z}\right).\]

It remains to determine the integral
\[\int_{B_{r+8|\log\varepsilon|}(x)}\eta_\alpha\eta_\beta  g_\alpha^{\prime\prime}g_\beta^{\prime\prime}.\]
Write this in Fermi coordinates with respect to $\Gamma_\alpha$. Integrating by parts in $z$ leads to
\begin{eqnarray*}
 &&\int_{\Gamma_\alpha} \int_{-\infty}^{+\infty}\eta_\alpha\eta_\beta  g_\alpha^{\prime\prime}g_\beta^{\prime\prime} \lambda_\alpha dz dy\\
  &=&-\int_{\Gamma_\alpha} \int_{-\infty}^{+\infty}\eta_\alpha\eta_\beta  W^{\prime\prime}\left(g_\beta\right) g_\alpha^{\prime}g_\beta^{\prime} \lambda_\alpha -\underbrace{\int_{\Gamma_\alpha} \int_{-\infty}^{+\infty}\eta_\alpha\left(\nabla_{\beta,0}\eta_\beta  \cdot\frac{\partial\Pi_\beta}{\partial z}\right) g_\alpha^{\prime} g_\beta^{\prime\prime}\lambda_\alpha}_{I}\\
  &-&\underbrace{\int_{\Gamma_\alpha} \int_{-\infty}^{+\infty}\eta_\alpha \eta_\beta  g_\alpha^{\prime}\xi_\beta^{\prime}\lambda_\alpha}_{II}
-\underbrace{\int_{\Gamma_\alpha} \int_{-\infty}^{+\infty}\eta_\alpha\eta_\beta g_\alpha^{\prime} g_\beta^{\prime\prime}\partial_z\lambda_\alpha dz}_{III}.
\end{eqnarray*}

When $g_\alpha^\prime g_\beta^{\prime\prime}\neq 0$, by Lemma \ref{comparison of distances},
\[\big|\frac{\partial}{\partial z}\eta_\beta \big|\leq \varepsilon^{\frac{1}{6}}|\nabla_{\beta,0}\eta_\beta|.\]
Hence as in the proof of Lemma \ref{lem 9.3} (here it is useful to observe that $g_\alpha^\prime g_\beta^{\prime\prime}=0$ outside the $8|\log\varepsilon|$ neighborhood of $\Gamma_\alpha\cup\Gamma_\beta$), we get
\begin{equation}\label{9.4.6}
  |I|\lesssim \varepsilon^{\frac{1}{7}}A(r;x)\left[ \int_{\Gamma_{\alpha}}\eta_\alpha^2dA_{\alpha,0}+
\int_{\Gamma_{\beta}}\big|\nabla_{\beta,0}\eta_\beta \big|^2dA_{\beta,0}\right].
\end{equation}

By the definition of $\xi_\beta$, we also have
\begin{equation}\label{9.4.7}
  |II|\lesssim\varepsilon^2\left[ \int_{\Gamma_{\alpha}}\eta_\alpha^2dA_{\alpha,0}+
\int_{\Gamma_{\beta}} \eta_\beta^2dA_{\beta,0}\right].
\end{equation}
Because $\partial_z\lambda_\alpha=O(\varepsilon)$, we get
\begin{equation}\label{9.4.8}
  |III|\lesssim \varepsilon^{\frac{1}{4}}A(r;x)\left[ \int_{\Gamma_{\alpha}}\eta_\alpha^2dA_{\alpha,0}+
\int_{\Gamma_{\beta}} \eta_\beta^2dA_{\beta,0}\right].
\end{equation}
The conclusion follows by combining \eqref{9.4.6}-\eqref{9.4.8}.
\end{proof}

\begin{lem}\label{lem 9.5}
  We have
\begin{eqnarray*}
 &&\sum_{\beta\neq\alpha}\int_{B_{r+8|\log\varepsilon|}(x)}\eta_\alpha \eta_\beta \left[W^{\prime\prime}(u)- W^{\prime\prime}(g_\beta)\right]g_\alpha^{\prime} g_\beta^{\prime}\\
   &=&- 4\sum_{\beta=\alpha,\alpha+1}A_{(-1)^{\beta-1}}^2\int_{\Gamma_\alpha}
  \eta_\alpha(y)\eta_\beta\left(\Pi_\beta(y,0)\right) e^{-|d_\beta(y,0)|}dA_{\alpha,0}\\
   &+&O\left(\varepsilon^{\frac{1}{7}}A(r;x)+A\left(r+60|\log\varepsilon|^2;x\right)^{\frac{3}{2}}\right)  \sum_{\beta\neq\alpha}\int_{\Gamma_{\beta}}|\nabla_{\beta,0}\eta_\beta|^2dA_{\beta,0}\\
     &+&O\left(\varepsilon^2+ \varepsilon^{\frac{1}{7}}A(r;x)+A\left(r+60|\log\varepsilon|^2;x\right)^{\frac{3}{2}}\right)
     \sum_{\beta}\int_{\Gamma_{\beta}}\eta_\beta^2dA_{\beta,0}.
\end{eqnarray*}

\end{lem}
\begin{proof}
The proof is divided into three steps.

{\bf Step 1.} First by Taylor expansion and \eqref{Schauder estimates}, proceeding as in the proof of Lemma \ref{lem 9.3} we get
\begin{eqnarray*}
   &&\Big| \int_{B_{r+8|\log\varepsilon|}(x)}\left[W^{\prime\prime}(u)-W^{\prime\prime}(g_\ast)\right]\varphi_\alpha\varphi_\beta\Big|\\
  &\lesssim& \|\phi\|_{L^\infty(B_{r+8|\log\varepsilon|}(x))}\int_{B_{r+8|\log\varepsilon|}(x)}\eta_\alpha\eta_\beta g_\alpha^\prime g_\beta^\prime\\
  &\lesssim& \left[\varepsilon^2+A\left(r+60|\log\varepsilon|^2;x\right)^{\frac{3}{2}}\right]\left[
     \int_{\Gamma_{\alpha}}\eta_\alpha^2dA_{\alpha,0}
     +\int_{\Gamma_{\beta}}\eta_\beta^2dA_{\beta,0}\right].
\end{eqnarray*}

{\bf Step 2.} Next by Lemma \ref{lem distance ladder} we have
\begin{eqnarray*}
   && \int_{\Gamma_\alpha}\eta_\alpha(y)\left(\int_{-\infty}^{+\infty}\big|\eta_\beta\left(\Pi_\beta(y,0)\right)-\eta_\beta\left(\Pi_\beta(y,z)\right)
   \big|\left[W^{\prime\prime}(g_\ast)- W^{\prime\prime}(g_\beta)\right]g_\beta^{\prime} g_\alpha^{\prime} \lambda_\alpha dz\right)dy \\
   &\lesssim&\int_{0}^{1}\int_{\Gamma_\alpha}\eta_\alpha(y)\left(\int_{-\infty}^{+\infty}\Big|\frac{d}{dt}\eta_\beta\left((1-t)\Pi_\beta(y,0)+
   t\Pi_\beta(y,z)\right)
   \big|g_\beta^{\prime} g_\alpha^{\prime} \lambda_\alpha dz\right)dy\\
    &\lesssim&\varepsilon^{\frac{1}{6}}\int_{0}^{1}\int_{\Gamma_\alpha}\eta_\alpha(y)\left(\int_{-\infty}^{+\infty}\Big|\nabla_{\beta,0}\eta_\beta\left((1-t)\Pi_\beta(y,0)+
   t\Pi_\beta(y,z)\right)
   \big|g_\beta^{\prime} g_\alpha^{\prime} \lambda_\alpha dz\right)dy\\
       &\lesssim& \varepsilon^{\frac{1}{7}}A(r;x)\left[ \int_{\Gamma_{\alpha}}\eta_\alpha^2dA_{\alpha,0}+
\int_{\Gamma_{\beta}}\big|\nabla_{\beta,0}\eta_\beta \big|^2dA_{\beta,0}\right].
\end{eqnarray*}

{\bf Step 3.} Finally, by Lemma \ref{lem form of interaction} (and Lemma \ref{comparison of distances} to estimate errors coming from comparing distances) we get
for $\beta=\alpha+1$ or $\alpha-1$,
\begin{eqnarray*}
   &&\int_{-\infty}^{+\infty}\left[W^{\prime\prime}(g_\ast(y,z))- W^{\prime\prime}(g_\beta(y,z))\right]g_\beta^{\prime}(y,z)g_\alpha^{\prime}(y,z)dz \\
   &=&-2A_{(-1)^{\alpha-\beta}}^2e^{-|d_\beta(y,0)|}+O\left(e^{-\frac{4}{3}|d_\beta(y,0)|}\right).
\end{eqnarray*}
and for $|\beta-\alpha|\geq 2$, we use the estimate
\begin{eqnarray*}
   &&\Big|\int_{-\infty}^{+\infty}\left[W^{\prime\prime}(g_\ast(y,z))- W^{\prime\prime}(g_\beta(y,z))\right]g_\beta^{\prime}(y,z)g_\alpha^{\prime}(y,z)dz\Big| \\
   &\lesssim&D_\alpha(y)e^{-|d_\beta(y,0)|}\\
   &\lesssim& \varepsilon^2+\left(\max_\alpha\max_{y\in\Gamma_\alpha\cap B_r(x)}D_\alpha\right) A(r;x)^2\\
   &\lesssim&\varepsilon^2+A(r;x)^{3/2},
\end{eqnarray*}
where we have used the fact that if $|d_\beta(y,0)|\leq 3|\log\varepsilon|$, then by Lemma \ref{comparison of distances} we have
\[|d_\beta(y,0)|=|d_{\beta-1}(y,0)|+|\Pi_\beta(y,0)-\Pi_{\beta-1}(y,0)|+o(1).\]

Combining these three steps we finish the proof.
\end{proof}

\section{A decay estimate}\label{sec decay estimate}
\setcounter{equation}{0}

Recall the definition of  $A(r;x)$ in Section \ref{sec Fermi coordinates}.
In this section we establish the following decay estimate.
\begin{prop}\label{decay estimate}
There exist two universal constants  $M\gg K\gg 1$ such that for any $r\in [2R/3,5R/6]$, if
\begin{equation}\label{absurd assumption}
\kappa:=A(r;0)\geq M\varepsilon^2|\log\varepsilon|,
\end{equation}
then we have
\[
A\left(r-KR_\ast;0\right) \leq\frac{1}{2}A(r;0),
\]
where
\[R_\ast:=\max\left\{\kappa^{-\frac{1}{2}}, 200|\log\varepsilon|^2\right\}.\]
\end{prop}

\subsection{Reduction to a decay estimate for Toda system}

In this subsection we reduce the proof of Proposition \ref{decay estimate} to a decay estimate for Toda system.

    By \eqref{Toda system} and Corollary \ref{coro imporved estimates on h}, for any $\beta$ and $y\in\Gamma_\beta\cap B_{r-R_\ast}(0)$,
  \begin{equation}\label{10.5}
H_\beta(y,0)=\frac{2 A_{(-1)^\beta}^2}{\sigma_0} e^{-d_{\beta-1}(y,0)}-\frac{2 A_{(-1)^{\beta-1}}^2}{\sigma_0}e^{d_{\beta+1}(y,0)}+O\left(\kappa^{7/6}\right)+O\left(\varepsilon^2\right).
\end{equation}

Take an arbitrary index $\alpha$ and $x\in \Gamma_\alpha\cap B_{r-KR_\ast}(0)$. To prove Proposition \ref{decay estimate}, it suffices to show that
\begin{equation}\label{decay estimate 3}
e^{-D_\alpha(x )}\leq\frac{\kappa}{2}.
\end{equation}

After a rotation and  a translation, assume $x=0$. In the finite cylinder $\mathcal{C}_{ R/7}(0):=B_{R/7}^{n-1}(0)\times(-R/7, R/7)$, $\Gamma_\alpha$ is represented by the graph $\{x_n=f_\alpha(y)\}$, where $y\in B^{n-1}_{R/7}(0)$. Without loss of generality assume it holds that
\begin{equation}\label{10.1}
f_\alpha(0)=0, \quad \nabla f_\alpha(0)=0.
\end{equation}

In the following, we also assume
\begin{equation}\label{assumption}
|d_{\alpha-1}(0)|\geq d_{\alpha+1}(0), \quad \mbox{ and } \quad d_{\alpha+1}(0)\leq 2|\log\varepsilon|.
\end{equation}
Then by Lemma \ref{lem graph construction},   we get a function $f_{\alpha+1}$ such that
\[\Gamma_{\alpha+1}\cap\mathcal{C}_{R/7}(0)=\left\{x_n=f_{\alpha+1}(x^\prime)\right\}.\]
Moreover, we have the Lipschitz bound
\begin{equation}\label{Lip bound}
  \|\nabla f_{\alpha}\|_{L^\infty(B_{R/7}^{n-1}(0))}+\|\nabla f_{\alpha+1}\|_{L^\infty(B_{R/7}^{n-1}(0))}\leq C.
\end{equation}
Curvature bounds on $\Gamma_\alpha$ and $\Gamma_{\alpha+1}$ can be transformed into
\begin{equation}\label{Hessian bound}
\|\nabla^2 f_{\alpha}\|_{L^\infty(B_{R/7}^{n-1}(0))}+  \|\nabla^2 f_{\alpha+1}\|_{L^\infty(B_{R/7}^{n-1}(0))}\lesssim\varepsilon.
\end{equation}

By \eqref{10.1} and \eqref{Hessian bound},  for any $y\in B^{n-1}_{K\kappa^{-1/2}}(0)$,
\begin{equation}\label{first derivative 1}
|\nabla f_\alpha(y)|\lesssim\varepsilon|y|\lesssim K\varepsilon\kappa^{-1/2}\lesssim KM^{-1/2}|\log\varepsilon|^{-1/2}.
\end{equation}

Concerning $f_{\alpha+1}$ we have the following estimates. In the following a positive constant $\delta<\min\{1/48, (1-\theta)/32\}$ will be fixed.
\begin{lem}
For $y\in B^{n-1}_{K\kappa^{-1/2}}(0)$, we have
\begin{equation}\label{first derivative 2}
 |\nabla f_{\alpha+1}(y)\big|=O_K\left(M^{-\delta}|\log\varepsilon|^{-\delta}\right).
\end{equation}
\end{lem}
\begin{proof}
By \eqref{assumption} and \eqref{Lip bound} we have
\[
\max_{B_{K\kappa^{-1/2}}^{n-1}(0)}\left(f_{\alpha+1}-f_\alpha\right)\lesssim K\kappa^{-1/2} \lesssim KM^{-1/2}|\log\varepsilon|^{-1/2}\varepsilon^{-1}.
\]
 As in Lemma \ref{comparison of distances}, combining this bound with an interpolation argument we get
\[\max_{B_{K\kappa^{-1/2}}^{n-1}(0)}|\nabla f_{\alpha+1}-\nabla f_\alpha\big|\lesssim C(K)M^{-\delta}|\log\varepsilon|^{-\delta}.\]
Combining this estimate with \eqref{first derivative 1}  we get \eqref{first derivative 2}.
\end{proof}

The following lemma shows that $d_{\alpha+1}$ is well approximated by vertical distances. The proof uses the fact that under assumptions \eqref{10.1} and \eqref{assumption}, $\Gamma_\alpha$ and $\Gamma_{\alpha+1}$ are almost parallel and horizontal.
\begin{lem}\label{lem comparison of distance 2}
For $y\in B^{n-1}_{K\kappa^{-1/2}}(0)$, if $e^{-|d_{\alpha+1}(y)|}\geq\varepsilon^2$, then
\begin{equation}\label{comparison of distance 2}
e^{-|d_{\alpha+1}(y)|}=e^{-\left(f_{\alpha+1}(y)-f_\alpha(y)\right)}+O_K\left(M^{-1}\kappa\right).
\end{equation}
\end{lem}
\begin{proof}
Assume the nearest point on $\Gamma_{\alpha+1}$ to $(y,f_\alpha(y))$ is $(y_\ast,f_{\alpha+1}(y_\ast))$. Because \[d_\alpha(y_\ast,f_{\alpha+1}(y_\ast))\leq |d_{\alpha+1}(y,f_\alpha(y))|\leq 2|\log\varepsilon|,\]
using Lemma \ref{comparison of distances} we deduce that
\[ |\nabla f_{\alpha+1}(y_\ast)-\nabla f_\alpha(y_\ast)|\lesssim \varepsilon^{1/3}.\]
Combining this estimate with \eqref{first derivative 1} and noting the fact that $M^{-1/2}|\log\varepsilon|^{-1/2}\gg \varepsilon^{1/3}$, we get
\begin{equation}\label{10.2.1}
|\nabla f_{\alpha+1}(y_\ast)|\lesssim KM^{-1/2}|\log\varepsilon|^{-1/2}.
\end{equation}

By \eqref{Lip bound}, we have
\[|f_{\alpha+1}(y)-f_\alpha(y)|\lesssim |d_{\alpha+1}(y)|\lesssim |\log\varepsilon|.\]
Hence we also have
\[|y-y_\ast|\lesssim |\log\varepsilon|.\]
Then by \eqref{Hessian bound}, we get
\begin{equation}\label{10.2.2}
\mbox{dist}\left((y,f_{\alpha+1}(y)), T_{(y_\ast,f_{\alpha+1}(y_\ast))}\Gamma_{\alpha+1}\right)\lesssim \varepsilon|\log\varepsilon|^2,
\end{equation}
where $T_{(y_\ast,f_{\alpha+1}(y_\ast))}\Gamma_{\alpha+1}$ denotes the tangent hyperplane of $\Gamma_{\alpha+1}$ at $(y_\ast,f_{\alpha+1}(y_\ast))$.

Let $\vartheta$ be the angel between $N_{\alpha+1}(y_\ast)$ (normal vector of $\Gamma_{\alpha+1}$ at $(y_\ast,f_{\alpha+1}(y_\ast)$) and the direction $e_{n+1}=(0,\cdots,0,1)$. By \eqref{10.2.1}, we get
\begin{equation}\label{10.2.3}
  |\vartheta|\lesssim KM^{-1/2}|\log\varepsilon|^{-1/2}.
\end{equation}
From this we deduce that
\begin{eqnarray*}
f_{\alpha+1}(y)-f_\alpha(y)&\leq& |d_{\alpha+1}(y)|\left[1+C\left(\sin\vartheta\right)^2\right]+\varepsilon|\log\varepsilon|^2\\
&\leq& |d_{\alpha+1}(y)|+C(K)M^{-1}.
\end{eqnarray*}
Then by Taylor expansion and the fact that $e^{d_{\alpha+1}(y)}\leq\kappa$, we obtain \eqref{comparison of distance 2}.
\end{proof}

By this lemma, now the Toda system  \eqref{10.5} is rewritten as, for any $y\in B_{K\kappa^{-1/2}}^{n-1}(0)$,
\begin{equation}\label{Toda system 3}
  \begin{cases}
           &\mbox{div}\left(\frac{\nabla f_\alpha(y)}{\sqrt{1+|\nabla f_\alpha(y)|^2}}\right)\geq -\frac{2 A_1^2}{\sigma_0}e^{-\left[f_{\alpha+1}(y)-f_\alpha(y)
\right]}+O_K\left(M^{-1}\kappa\right), \\
     & \mbox{div}\left(\frac{\nabla f_{\alpha+1}(y)}{\sqrt{1+|\nabla f_{\alpha+1}(y)|^2}}\right)\leq\frac{2 A_1^2}{\sigma_0}e^{-\left[f_{\alpha+1}(y)-f_\alpha(y)
\right]}+O_K\left(M^{-1}\kappa\right),
 \end{cases}
\end{equation}
Taking the difference we obtain the equation for $f_{\alpha+1}-f_\alpha$,
\begin{equation}\label{difference eqn}
\mbox{div}\left[A_\alpha \nabla\left(f_{\alpha+1}-f_\alpha\right)\right]\leq\frac{4 A_1^2}{\sigma_0}e^{-\left[f_{\alpha+1}(y)-f_\alpha(y)
\right]}+O_K\left(M^{-1}\kappa\right).
\end{equation}
Here $A_\alpha$ is the symmetric matrix with entries defined by
\begin{equation}\label{definition of A alpha}
 \int_{0}^{1} \left[ \frac{\delta_{ij}}{\sqrt{1+|\nabla f^t_\alpha|^2}}-\frac{\partial_if^t_\alpha\partial_jf^t_\alpha}
{\left(1+|\nabla f^t_\alpha|^2\right)^{3/2}}\right]dt, \quad 1\leq i,j\leq n-1,
\end{equation}
where $f^t_\alpha:=(1-t)f_\alpha+tf_{\alpha+1}$ and $\delta_{ij}$ denotes Kronecker delta.

In view of \eqref{first derivative 1} and \eqref{first derivative 2}, we have
\begin{equation}\label{close to idenity}
|A_\alpha(y)-Id|\lesssim M^{-\delta}, \quad \forall y\in B_{K\kappa^{-1/2}}(0).
\end{equation}

Define the function in $B^{n-1}_{K}(0)$,
\[v_\alpha(y):=f_{\alpha+1}\left(\kappa^{-1/2} y\right)-f_{\alpha}\left(\kappa^{-1/2} y\right)-|\log\kappa|.\]
It satisfies in $B^{n-1}_K(0)$,
\begin{equation}\label{Liouville eqn}
 \mbox{div}\left(\widetilde{A}_\alpha \nabla v_\alpha\right) \leq\frac{4 A_1^2}{\sigma_0}e^{-v_\alpha}+O_K\left(M^{-1}\right).
\end{equation}
Here $\widetilde{A}_\alpha(y):=A_\alpha(\kappa^{-1/2}y)$, which still satisfies \eqref{close to idenity}.

\subsection{Completion of the proof of Proposition \ref{decay estimate}}\label{subsec completion of decay estimate}

First we  show that $v_\alpha$ is almost stable.
\begin{lem}\label{almost stable}
For any $\widetilde{\eta}_\alpha$ and $\widetilde{\eta}_{\alpha+1}\in C_0^\infty(B_K^{n-1}(0))$,
\begin{eqnarray*}
&&\sum_{\beta=\alpha,\alpha+1}  \int_{B_{K}^{n-1}(0)}|\nabla\widetilde{\eta}_\beta |^2dy
+O_K\left(M^{-\delta}\right)\sum_{\beta=\alpha,\alpha+1} \int_{B_{K}^{n-1}(0)}\widetilde{\eta}_\beta ^2dy\\
&\geq&\frac{2A_{(-1)^{\beta-1}}^2}{\sigma_0} \int_{B_{K}^{n-1}(0)}e^{-v_\alpha}
\left[\widetilde{\eta}_{\alpha+1}-\widetilde{\eta}_\alpha\right]^2dy.
\end{eqnarray*}
\end{lem}
\begin{proof}
For $y\in B_{K\kappa^{-1/2}}^{n-1}(0)$, let $\eta_\alpha(y):=\widetilde{\eta}_\alpha\left(\kappa^{1/2}y\right)$ and  $\eta_{\alpha+1}$ be defined similarly. We will view them as functions on $\Gamma_\alpha$ (respectively $\Gamma_{\alpha+1}$), by identifying $y$ with $(y,f_\alpha(y))$ etc.

Then Proposition \ref{prop reduction of stability} says
\begin{eqnarray*}%\label{10.6}
&&\sum_\beta \int_{B_{K\kappa^{-1/2}}^{n-1}(0)}|\nabla\eta_\beta |^2\left[1+O\left(|\nabla f_\beta |^2\right)\right]dy+\mathcal{Q}(\eta_\alpha,
\eta_{\alpha+1})\nonumber\\
& \geq& \sum_\beta\frac{2A_{(-1)^{\beta-1}}^2}{\sigma_0} \int_{B_{K\kappa^{-1/2}}^{n-1}(0)}e^{-d_{\alpha+1}(y,0) }\left[\eta_{\alpha}(y)- \eta_{\alpha+1}\left(\Pi_{\alpha+1}(y,0)\right)^2
\right]\left[1+O\left(|\nabla f_\beta|^2\right)\right]dy,
\end{eqnarray*}
where
\begin{eqnarray*}
|\mathcal{Q}(\eta)|&\lesssim&\left[\varepsilon^{\frac{1}{4}}+\kappa^{\frac{1}{2}}\right]\left(\sum_\beta \int_{B_{K\kappa^{-1/2}}^{n-1}(0)}|\nabla\eta_\beta(y)|^2\left[1+O\left(|\nabla f_\beta(y)|^2\right)\right]dy \right)\\
  &+&  \left[\varepsilon^2+\kappa^{\frac{3}{2}}+\varepsilon^{\frac{1}{7}}\kappa\right]
\left(\sum_\beta \int_{B_{K\kappa^{-1/2}}^{n-1}(0)}\eta_\beta(y)^2\left[1+O\left(|\nabla f_\beta(y)|^2\right)\right]dy\right)\\
&\lesssim&\kappa^{1/8}\left(\sum_\beta \int_{B_{K\kappa^{-1/2}}^{n-1}(0)}|\nabla\eta_\beta(y)|^2dy \right)+\frac{\kappa}{M}\left(\sum_\beta \int_{B_{K\kappa^{-1/2}}^{n-1}(0)}\eta_\beta(y)^2dy\right).
\end{eqnarray*}
Substituting Lemma \ref{lem comparison of distance 2} into this inequality and taking a rescaling, we obtain
\begin{eqnarray*}
&& \left[1+O\left(M^{-\delta}\right)\right]\sum_{\beta=\alpha,\alpha+1} \int_{B_{K}^{n-1}(0)}|\nabla\widetilde{\eta}_\beta |^2dy
+O\left(M^{-\delta}\right)\sum_{\beta=\alpha,\alpha+1} \int_{B_{K}^{n-1}(0)}\widetilde{\eta}_\beta ^2dy\\
&\geq&\frac{2A_{(-1)^{\beta-1}}^2}{\sigma_0} \int_{B_{K}^{n-1}(0)}e^{-v_\alpha}
\left[\widetilde{\eta}_{\alpha+1}-\widetilde{\eta}_\alpha\right]^2dy.
\end{eqnarray*}
Moving the coefficient $\left[1+O\left(M^{-\delta}\right)\right]$ before the first group of integrals to  other groups, we conclude the proof.
\end{proof}

This almost stability condition implies an $L^1$ estimate.
\begin{lem}\label{lem integral small}
For any $\sigma>0$, if $n\leq 10$ and we have chosen $M\gg K\gg 1$ (depending only on $n$ and $\sigma$), then
\begin{equation}\label{integral small}
\int_{B^{n-1}_2(0)} e^{-v_\alpha} \leq\sigma.
\end{equation}
\end{lem}
\begin{proof}
Let $V_\alpha:=e^{-v_\alpha}$. Direct calculation using \eqref{Liouville eqn} gives
\begin{equation}\label{universal Simons type eqn}
-\mbox{div}\left(\widetilde{A}_\alpha \nabla V_\alpha\right) \leq \frac{4A_1^2}{\sigma_0}V_\alpha^2-\left[1+O\left(M^{-\delta}\right)\right]V_\alpha^{-1}|\nabla V_\alpha|^2+CM^{-\delta} V_\alpha.
 \end{equation}
Following Farina \cite{Farina}, for any $\eta\in C_0^\infty(B_K^{n-1}(0))$ and $q>0$, multiplying \eqref{universal Simons type eqn} by $V_\alpha^{2q-1}\eta^2$ and integrating by parts, we get
\begin{eqnarray}\label{integral estimate 1}
&&\left[2q+O_K\left(M^{-\delta}\right)\right]\int V_\alpha^{2q-2}|\nabla V_\alpha|^2\eta^2-\frac{1}{2q}\int V_\alpha^{2q}\Delta\eta^2\\
& \leq&\frac{4A_1^2}{\sigma_0}\int V_\alpha^{2q+1}\eta^2+O_K\left(M^{-\delta}\right) \int V_\alpha^{2q}\eta^2. \nonumber
\end{eqnarray}

On the other hand, for any $\eta\in C_0^\infty(B_K^{n-1}(0))$, substituting $\eta_\alpha=\eta$ and $\eta_{\alpha+1}=-\eta$ into Lemma \ref{almost stable} gives
\begin{equation}\label{universal stability inequality}
\int_{B_{K}^{n-1}(0)}|\nabla \eta|^2dy
+C(K)M^{-\delta}\int_{B_{K}^{n-1}(0)}\eta^2dy\geq\frac{4A_1^2}{\sigma_0} \int_{B_{K}^{n-1}(0)}V_{\alpha}\eta^2dy.
\end{equation}
Taking $V_\alpha^q\eta$ as test functions in \eqref{universal stability inequality} leads to
\begin{eqnarray}\label{integral estimate 2}
\frac{4A_1^2}{\sigma_0}\int V_\alpha^{2q+1}\eta^{2 }&\leq & q^2\int V_\alpha^{2q-2}|\nabla V_\alpha|^2\eta^2 \\
&&+C\int V_\alpha^{2q}\left(\big|\Delta\eta^2\big|+|\nabla\eta|^2\right)+C(K)M^{-\delta} \int V_\alpha^{2q}\eta^2.\nonumber
\end{eqnarray}

Combining \eqref{integral estimate 1} and \eqref{integral estimate 2} we get, if
\[2q>q^2\left(1+C(K)M^{-\delta}\right),\]
which is true if $M$ is large enough and $q<15/8$, then
\begin{equation}\label{integral estimate 3}
\int V_\alpha^{2q-2}|\nabla V_\alpha|^2\eta^2+\int V_\alpha^{2q+1}\eta^2\leq C(q)\int V_\alpha^{2q}\left(\big|\Delta\eta^2\big|+|\nabla\eta|^2+C(K)M^{-\delta} \eta^2\right).
\end{equation}

Still as in Farina \cite{Farina}, replace $\eta$ by $\eta^m$ for some $m\gg 1$ (depending only on $q$), where $\eta$ is a standard cut-off function. Applying the H\"{o}lder inequality to \eqref{integral estimate 3} gives, for any $q<15/8$,
\begin{equation}\label{integral estimate 4}
\int_{B_2^{n-1}(0)} \left[V_\alpha^{2q-2}|\nabla V_\alpha|^2+V_\alpha^{2q+1}\right]\leq C(q) K^{n-1-2(2q+1)}+C(q,K)M^{-\delta} K^{n-1}.
\end{equation}
If $n\leq 10$ and $q>7/4$, then
\[n-1-2(2q+1)<0.\]
First choose a $K$ so large that $C(q) K^{n-1-2(2q+1)}<\sigma/2$, then take an $M$ so large that $C(q,K)M^{-\delta} K^{n-1}<\sigma/2$, we get \eqref{integral small} by the H\"{o}lder inequality.
\end{proof}

Now we improve this $L^1$ estimate to an $L^\infty$ estimate. To this end, we need the following decay estimate.
\begin{lem}\label{lem decay estimate 2}
There exist two universal constants $\sigma_\ast$ and $\tau_\ast\in(0,1)$ so that the following holds.
For any $y\in B_1^{n-1}(0)$ and $r\in(0,1)$, suppose
\begin{equation}\label{integral small 2}
  r^{3-n}\int_{B_r^{n-1}(y)}V_\alpha\leq \sigma_\ast,
\end{equation}
then
\begin{equation}\label{decay estimate 2}
\left(\tau_\ast r\right)^{3-n}\int_{B_{\tau_\ast r}^{n-1}(y)}V_\alpha\leq  \frac{1}{2}r^{3-n}\int_{B_r^{n-1}(y)}V_\alpha.
\end{equation}
\end{lem}
The proof follows the method introduced by the first author in \cite{Wangstable1,Wangstable2} with minor modifications. In that proof what we need are:
\begin{itemize}
  \item [(i)] the elliptic inequality \eqref{Liouville eqn};
  \item [(ii)] an integral estimate
\begin{equation}\label{10.7}
\int_{B_{r/2}^{n-1}(y)}V_\alpha^2\lesssim \left[r^{-2}+C(K)M^{-\delta}\right]\int_{B_r(y)}V_\alpha.
\end{equation}
\end{itemize}
The estimate \eqref{10.7} follows  by taking $q=1/2$ in \eqref{integral estimate 3}, which gives, for any $\eta\in C_0^\infty(B_2^{n-1}(0))$,
\begin{equation}\label{integral estimate 5}
\int_{B_2^{n-1}(0)}V_\alpha^2\eta^2\lesssim \int_{B_2^{n-1}(0)} V_\alpha\left(|\eta||\Delta\eta|+|\nabla\eta|^2+M^{-\delta} \eta^2\right),
\end{equation}
and then choosing  $\eta$ to be a standard cut-off function with $\eta\equiv 1$ in $B_{r/2}^{n-1}(y)$, $\eta\equiv 0$ outside $B_{r}^{n-1}(y)$, $|\nabla^2\eta|+|\nabla\eta|^2\lesssim r^{-2}$.

Using this lemma we get
\begin{lem}\label{lem small regularity}
If
  \begin{equation}\label{integral small 2}
    \int_{B_2^{n-1}(0)}V_\alpha\leq \sigma_{\ast },
  \end{equation}
  then
  \begin{equation}\label{decay estimate final}
    \max_{B_1^{n-1}(0)}V_\alpha\leq \frac{1}{2}.
  \end{equation}
\end{lem}
\begin{proof}
For   any $y\in B_1^{n-1}(0)$,
\[  \int_{B_1^{n-1}(0)}V_\alpha\leq \int_{B_2^{n-1}(0)}V_\alpha\leq \sigma_{\ast }.\]
 Then by the previous lemma, we get
\begin{equation}\label{iteration of decay estimate}
  \tau_\ast^{(3-n)(k+1)}\int_{B_{\kappa_\ast^{k+1}}(y)}V_\alpha\leq \frac{1}{2}\tau_\ast^{(3-n)k}\int_{B_{\tau_\ast^k}(y)}V_\alpha, \quad \forall k\geq 1.
\end{equation}
In other words, for  any $y\in B_1^{n-1}(0)$ and $r\in(0,1)$,
\begin{equation}\label{Morrey bound}
  \int_{B_r^{n-1}(y)}V_\alpha\leq C\sigma_\ast r^{n-3+\frac{\log 2}{|\log\tau_\ast|}}.
\end{equation}
Combining standard elliptic estimates in Morrey spaces   and \eqref{integral small 2} we get \eqref{decay estimate final}.
\end{proof}

In view of Lemma \ref{lem integral small}, Lemma \ref{lem small regularity} is applicable to $V_\alpha$ for all $\kappa$ small. Rescaling \eqref{decay estimate final} back we get \eqref{decay estimate 3}. The proof of Proposition \ref{decay estimate} is thus complete.

\section{Distance bound}\label{sec distance bound}
\setcounter{equation}{0}

In this section we give a lower bound on $D_\alpha$.
\begin{prop}\label{improved distance}
There exists a universal constant $C$ such that
\[ A(2R/3;0)\leq C\varepsilon^2\left(\log|\log\varepsilon|\right)^2.\]
\end{prop}
\begin{rmk}
This is a small improvement of  \cite[Proposition 3.2]{chodosh2018minimal}, where they established a bound in the form
\[ A(2R/3;0)=o\left(\varepsilon^2 |\log\varepsilon|\right).\]
 We do not know if there exists a universal constant $C$ such that
  \[A(2R/3;0)\leq C\varepsilon^2.\]
\end{rmk}

\subsection{Non-optimal lower bounds}\label{subsec nonoptimal lower bound}
Before proving this proposition, we first provide three non-optimal lower bounds.
\begin{lem}\label{improved distance I}
For any $\theta\in(0,1)$, for any $\alpha$ and $y\in \Gamma_\alpha\cap B_{19R/24}(0)$,
\[D_{\alpha}(y)\geq \left(1+\theta\right)|\log\varepsilon|.\]
\end{lem}
\begin{proof}
Assume by the contrary $A(19R/24;0)\geq \varepsilon^{1+\theta}$. Then by the monotone dependence of $A(r;0)$ on $r$, we have
\[A(r;0)\geq A(19R/24;0)\geq \varepsilon^{1+\theta}, \quad \forall r\in[19R/24,5R/6].\]

Now Proposition \ref{decay estimate} is applicable, which says
\[A\left(r-KR^{\frac{1+\theta}{2}};0\right)\leq \frac{1}{2}A(r;0).\]
Here we have used the estimate on the constant in  Proposition \ref{decay estimate}, $R_\ast\leq R^{\frac{1+\theta}{2}}$.

  An iteration of this decay estimate from $r=5R/6$ to $19R/24$ leads to a contradiction
\[A\left(19R/24;0\right)\leq 2^{-cK^{-1}R^{\frac{1-\theta}{2}}}A\left(5R/6;0\right)\leq \varepsilon^2.\]
In the last inequality we have used $A(5R/6;0)\leq 1$, which is a consequence of Lemma \ref{O(1) scale}.
\end{proof}

\begin{lem}\label{improved distance II}
There exists a universal constant $C$ such that for any $\alpha$ and $y\in \Gamma_\alpha\cap B_{3R/4}(0)$,
\[D_{\alpha}(y)\geq 2|\log\varepsilon|- 2\log|\log\varepsilon|.\]
\end{lem}
\begin{proof}
By Lemma \ref{improved distance I}, we can assume
\begin{equation}\label{distance bound 10.1}
  A(19R/24;0)\lesssim \varepsilon^{1+\theta},
\end{equation}
where $\theta$ is very close to $1$ (to be determined below).

Now assume by the contrary $A(3R/4;0)\geq \varepsilon^2|\log\varepsilon|^2$. Then by the monotone dependence of $A(r;0)$ on $r$, we have
\[A(r;0)\geq A(3R/4;0)\geq \varepsilon^2|\log\varepsilon|^2, \quad \forall r\in[3R/4,19R/24].\]

Now  Proposition \ref{decay estimate} is applicable, which says
\[A\left(r-K\frac{R}{\log R};0\right)\leq \frac{1}{2}A(r;0).\]
Here we have used the estimate on the constant in  Proposition \ref{decay estimate}, $R_\ast\leq R/\log R$.

  An iteration of this decay estimate from $r=19R/24$ to $3R/4$ leads to a contradiction, i.e.
\[A\left(3R/4;0\right)\leq 2^{-cK^{-1}\log R}A\left(5R/6;0\right)\leq C\varepsilon^{1+\theta+\frac{c\log 2}{K}}\leq \varepsilon^2,\]
provided $1+\theta+\frac{c\log 2}{K}>2$, i.e. $\theta$ has been chosen to be very close to $1$.
\end{proof}

\begin{lem}\label{improved distance III}
There exists a universal constant $C$ such that for any $\alpha$ and $y\in \Gamma_\alpha\cap B_{17R/24}(0)$,
\[D_{\alpha}(y)\geq 2|\log\varepsilon|- \log|\log\varepsilon|-C.\]
\end{lem}
\begin{proof}
By Lemma \ref{improved distance I}, we can assume
\begin{equation}\label{distance bound 10.1}
  A(3R/4;0)\lesssim \varepsilon^2|\log\varepsilon|^2.
\end{equation}

Now assume by the contrary $A(17R/24;0)\geq M\varepsilon^2|\log\varepsilon|$ for some $M>0$ large. Then by the monotone dependence of $A(r;0)$ on $r$,  we have
\[A(r;0)\geq A(17R/24;0)\geq  M\varepsilon^2|\log\varepsilon|^2,\quad \forall r\in[17R/24,3R/4].\]

Now  Proposition \ref{decay estimate} is applicable, which says
\[A\left(r-K\frac{R}{\sqrt{M\log R}};0\right)\leq \frac{1}{2}A(r;0).\]
Here we have used the estimate on the constant in  Proposition \ref{decay estimate}, $R_\ast\leq R/\sqrt{M\log R}$.

  An iteration of this decay estimate from $r=3R/4$ to $17R/24$ leads to a contradiction, i.e.
\[A\left(17R/24;0\right)\leq 2^{-cK^{-1}\sqrt{\log R}}A\left(3R/4;0\right)\leq 2^{-cK^{-1}\sqrt{M\log R}}\varepsilon^2|\log\varepsilon|^2\leq \varepsilon^2|\log\varepsilon|.\]
The last inequality follows from the estimate
\[2^{-cK^{-1}\sqrt{M|\log \varepsilon|}}|\log\varepsilon|=2^{-cK^{-1}\sqrt{M|\log \varepsilon|}+ (\log 2) (\log|\log\varepsilon|)}\leq 1,\]
which is true if $\varepsilon$ is small enough.
\end{proof}

\subsection{Proof of Proposition \ref{improved distance}}

 By Lemma \ref{improved distance III}, now we can assume
  \begin{equation}\label{distance bound 10.2}
A(17R/24;0)\lesssim \varepsilon^2|\log\varepsilon|.
  \end{equation}
Hence by \eqref{5.2}, for any $\alpha$
\begin{equation}\label{10.8}
  \|H_\alpha\|_{L^\infty(B_{17R/24}(0))}\lesssim \varepsilon^2|\log\varepsilon|.
\end{equation}

Denote $\rho:=|d_{\alpha+1}|$. By \cite[Eqn. (2.41), Lemma 2.9 and Appendix A]{chodosh2018minimal} and \eqref{10.5}, $\rho$ satisfies
\begin{equation}\label{graph eqn}
 \mathcal{L}\rho(y)|+|A_\alpha(y)|^2\rho(y)+\mathcal{N}(\rho)\leq \frac{8A_1^2}{\sigma_0}e^{-\rho(y)}+O\left(\varepsilon^2\right).
\end{equation}
Here $\mathcal{L}$ is the linear uniformly elliptic operator
\[\mathcal{L}\varphi:=a(y)^{-1}\mbox{div}_{\Gamma_\alpha}\left[a(y)\nu_\alpha(y)\cdot \nu_{\alpha+1}(y,\rho(y))\nabla_{\alpha+1}\varphi\right],\]
where
\[a(y):=\frac{\lambda_\alpha(y,0)}{\lambda_{\alpha+1}(y,\rho(y))}.\]
The nonlinear error term $\mathcal{N}(\rho)$ satisfies
\begin{equation}\label{h.o.t}
  |\mathcal{N}(\rho)(y)|\lesssim \varepsilon^3|\rho(y)|^2+\varepsilon|\nabla_\alpha \rho(y)|^2
\end{equation}

\begin{lem}\label{lem gradient estimate}
For any  $y\in \Gamma_\alpha\cap B_{11R/16}(0)$, if $\rho(y)\leq 2|\log\varepsilon|$, then
\[\frac{|\nabla_\alpha\rho(y)|}{\rho(y)}\lesssim\varepsilon.\]
\end{lem}
\begin{proof}
Fix an $\alpha$ and a point $y_\ast\in\Gamma_\alpha\cap B_{11R/16}(0)$.  By our assumption and Lemma \ref{improved distance III},
\begin{equation}\label{10.2.4}
2|\log\varepsilon|-\log|\log\varepsilon|-C\leq \rho(y_\ast)\leq 2|\log\varepsilon|.
\end{equation}

Choose a coordinate system such that $\Gamma_\alpha\cap B_{R/48}(0)$ and $\Gamma_{\alpha+1}\cap B_{R/48}(0)$ are represented by graphs of functions $f_\alpha$ and $f_{\alpha+1}$, and $\rho(y_\ast)$ is attained at $(y_{\ast},f_{\alpha+1}(y_\ast))$. Therefore we have
\[\rho(y_\ast)=f_{\alpha+1}(y_\ast)-f_\alpha(y_\ast), \quad \rho(y)\leq f_{\alpha+1}(y)-f_\alpha(y),\]
and consequently
\begin{equation}\label{10.2.5}
  \nabla\rho(y_\ast)=\nabla f_{\alpha+1}(y_\ast)-\nabla f_\alpha(y_\ast).
\end{equation}

 Define
\[\widetilde{\rho}(\widetilde{y}):=\rho(y_\ast)^{-1}\left[
 f_{\alpha+1}(y_\ast+\varepsilon^{-1}\widetilde{y})-f_{\alpha}(y_\ast+\varepsilon^{-1}\widetilde{y})\right], \quad \widetilde{y}\in B_{1/48}^{n-1}(0).\]
Taking difference of \eqref{10.8}, by \eqref{10.2.4} we obtain
\begin{equation}\label{10.9}
  \Big\|\mbox{div}\left(\bar{A}_\alpha\nabla \widetilde{\rho}\right)\Big\|_{L^\infty(B_{1/48}^{n-1}(0))}\leq C.
\end{equation}
Here $\bar{A}_{\alpha,\varepsilon}(\widetilde{y})=A_\alpha(y_\ast+\varepsilon^{-1}\widetilde{y})$ and $A_\alpha$ is defined as in  \eqref{definition of A alpha}.

Note that $\widetilde{\rho}>0$ in $B_{1/48}^{n-1}(0)$. On the other hand, by our assumption and Lemma \ref{improved distance III}, $\widetilde{\rho}(0)=1$.
Then by Moser's Harnack inequality for inhomogeneous equations (see \cite[Theorem 8.17 and 8.18]{GT}), there exists a $\sigma>0$ such that
\begin{equation}\label{local Harnack inequality}
  1/2\leq \widetilde{\rho} \leq 2, \quad \mbox{in } B_\sigma^{n-1}(0).
\end{equation}
Using standard elliptic estimates we get a universal constant $C$ such that $\|\widetilde{\rho}\|_{C^{1,1/2}(B_\sigma^{n-1}(0))}\leq C$. In particular, $|\nabla\widetilde{\rho}(0)|\leq C$. Rescaling back and using \eqref{10.2.5} we conclude the proof.
\end{proof}
The proof, in particular, \eqref{local Harnack inequality} implies that
\begin{coro}
If $\rho(y)\leq 2|\log\varepsilon|$, then
\begin{equation}\label{local Harnack inequlaity II}
  \sup_{B_{\sigma R}(y)\cap\Gamma_\alpha}\rho\leq 4|\log\varepsilon|.
\end{equation}
\end{coro}
Substituting  \eqref{local Harnack inequlaity II} and Lemma \ref{lem gradient estimate} into \eqref{h.o.t} we obtain
\begin{coro}
If $\rho(y)\leq 2|\log\varepsilon|$, then
  \begin{equation}\label{h.o.t upper bound}
  \sup_{B_{\sigma R}(y)\cap\Gamma_\alpha}|\mathcal{N}(\rho)|\lesssim   \varepsilon^3|\log\varepsilon|^2.
\end{equation}
\end{coro}

Next we give a decay estimate with a weaker assumption than Proposition \ref{decay estimate}.
\begin{prop}\label{decay estimate improved}
There exist two universal constants  $M\gg K\gg 1$ such that for any $r\in [2R/3,5R/6]$, if
\begin{equation}\label{absurd assumption II}
\kappa:=A(r;0)\geq M\varepsilon^2,
\end{equation}
then we have
\[
A\left(r-KR_\ast;0\right) \leq\frac{1}{2}A(r;0),
\]
where
\[R_\ast:=\max\left\{\kappa^{-\frac{1}{2}}, 200|\log\varepsilon|^2\right\}.\]
\end{prop}
\begin{proof}
Fix an $\alpha$ and a point $x_\ast\in \Gamma_\alpha\cap B_{r-KR_\ast}$. Assume
\begin{equation}\label{absurd assumption 3}
e^{-\rho(x_\ast)}\geq\frac{\kappa}{2}\geq\frac{M}{2}\varepsilon^2.
\end{equation}
Let
\[\widetilde{\Gamma}:=\kappa^{\frac{1}{2}}\left(\Gamma_\alpha-x_\ast\right), \quad \widetilde{\rho}\left(\tilde{y}\right):=\rho\left(\kappa^{-\frac{1}{2}}\widetilde{y}\right)+\log\kappa, \quad \forall \tilde{y}\in \widetilde{\Gamma}\cap B_K(0).\]

By \eqref{absurd assumption}, we have
\begin{equation}\label{flat after scaling}
\big|A_{\widetilde{\Gamma}}\big|\leq\frac{C\varepsilon}{\kappa^{1/2}}\leq \frac{C}{M}.
\end{equation}
Hence $\widetilde{\Gamma}$ is very close to a hyperplane in $B_K(0)$.

A direct rescaling of \eqref{graph eqn} leads to
\begin{equation}\label{rescaling of graph eqn}
\widetilde{\mathcal{L}}\widetilde{\rho}(\widetilde{y})+O\left(\kappa^{-1}\varepsilon^3|\log\varepsilon|^2\right)\leq \frac{8A_1^2}{\sigma_0}e^{-\widetilde{\rho}(\widetilde{y})}+O\left( \frac{\varepsilon^2}{\kappa}\right).
\end{equation}
Here $\widetilde{\mathcal{L}}$ is the rescaling of $\mathcal{L}$, but we rewrite it as
\[\mathcal{L}=a(y)^{-1}\mbox{div}\left(\widehat{A}_\alpha(y)\nabla_\alpha\right),\]
where $|a(y)-1|\ll 1$ and
\[\|\widehat{A}_\alpha-Id\|_{L^\infty(\Gamma_\alpha\cap B_{K\kappa^{-1/2}}(0))}\ll 1,\]
by a derivation similar to the one of \eqref{close to idenity}. (Note that under assumption \eqref{absurd assumption 3}, we have \eqref{absurd assumption II}, which then enables us to apply Lemma \ref{comparison of distances} to estimate $\nu_\alpha\cdot\nu_{\alpha+1}$ etc.)

As in \eqref{universal stability inequality}, for any $\eta\in C_0^\infty(\widetilde{\Gamma}\cap B_K(0))$, we still have
\begin{equation}\label{universal stability inequality II}
\int_{\widetilde{\Gamma}\cap B_{K}(0)}|\nabla \eta|^2
+C(K)M^{-\delta}\int_{\widetilde{\Gamma}\cap B_{K}(0)}\eta^2\geq\frac{4A_1^2}{\sigma_0} \int_{\widetilde{\Gamma}\cap B_{K}(0)}e^{-\widetilde{\rho}}\eta^2dy.
\end{equation}
Then proceeding as in Subsection \ref{subsec completion of decay estimate} we conclude the proof.
\end{proof}

As in Subsection \ref{subsec nonoptimal lower bound}, we use this decay estimate to prove Proposition \ref{improved distance}.
\begin{proof}[Proof of Proposition \ref{improved distance}]
By Lemma \ref{improved distance III}, we can assume
\begin{equation}\label{distance bound 10.4}
  A(17R/24;0)\lesssim \varepsilon^2|\log\varepsilon|.
\end{equation}

Now assume by the contrary there exists a large constant $M$ such that
\begin{equation}\label{absurd assumption 4}
  A(2R/3;0)\geq M\varepsilon^2\left(\log|\log\varepsilon|\right)^2.
\end{equation}
Then by the monotone dependence of $A(r;0)$ on $r$, we have
\[\kappa:=A(r;0)\geq A(2R/3;0)\geq  M\varepsilon^2\left(\log|\log\varepsilon|\right)^2, \quad \forall r\in[2R/3,17R/24].\]

Now  Proposition \ref{decay estimate improved} is applicable, which says
\[A\left(r-K\frac{R}{\sqrt{M}\log|\log\varepsilon|};0\right)\leq \frac{1}{2}A(r;0).\]
Here we have used the estimate on the constant in  Proposition \ref{decay estimate improved}, $R_\ast\leq \frac{R}{\sqrt{C_\varepsilon}}$.

  An iteration of this decay estimate from $r=17R/24$ to $2R/3$ leads to a contradiction, i.e.
\begin{eqnarray*}
 A\left(2R/3;0\right)&\leq& 2^{-cK^{-1}\sqrt{M}\log|\log\varepsilon|}A(17R/24;0)\\
 &\leq& C2^{-cK^{-1}\sqrt{M}\log|\log\varepsilon|}\varepsilon^2|\log\varepsilon|\\
 &\leq& \frac{\left(\log|\log\varepsilon|\right)^2}{2}\varepsilon^2.
\end{eqnarray*}
The last inequality is true provided $M$ is large enough. This is a contradiction with \eqref{absurd assumption 4} and the proof is complete.
\end{proof}

\section{Proof of main results}\label{sec completion of proof}
\setcounter{equation}{0}

In this section we prove Theorem \ref{main result} and its two corollaries, Corollary \ref{main coro 1} and \ref{main coro 2}.

\begin{proof}[Proof of Theorem \ref{main result}]
Substituting Proposition \ref{improved distance}
 into \eqref{Schauder estimates}, we get
\begin{equation}\label{estimate 3}
\|\phi\|_{C^{2,\theta}(B_{2R/3}(0))}+\max_\alpha\|H_\alpha+\Delta_{\alpha,0} h_\alpha\|_{C^{\theta}(\Gamma_\alpha\cap B_{2R/3}(0))}
\lesssim\varepsilon^{1+\theta}.
\end{equation}
By Lemma \ref{control on h_0}, for any $\alpha$,
\begin{eqnarray*}
\|H_\alpha\|_{C^{\theta}(\Gamma_\alpha\cap B_{R/2}(0))}&\lesssim &
\|\phi\|_{C^{2,\theta}(B_{2R/3}(0))}+\|H_\alpha+\Delta_{\alpha,0} h_\alpha\|_{C^{\theta}(\Gamma_\alpha\cap B_{2R/3}(0))}+A(2R/3;0)\\
&\lesssim&\varepsilon^{1+\theta}.
\end{eqnarray*}
After rescaling back to $u_\varepsilon$, this says for any connected component of $\{u_\varepsilon=0\}$, say $\Gamma_{\alpha,\varepsilon}$, its mean curvature satisfies
\begin{equation}\label{rescaled back}
  \|H_{\alpha,\varepsilon}\|_{L^\infty(\Gamma_{\alpha,\varepsilon}\cap B_{1/2}(0))}
\lesssim\varepsilon\left(\log|\log\varepsilon|\right)^2, \quad \mbox{and} \quad  \|H_{\alpha,\varepsilon}\|_{C^{\theta}(\Gamma_{\alpha,\varepsilon}\cap B_{1/2}(0))}
\leq C.
\end{equation}
Because $\Gamma_{\alpha,\varepsilon}\cap B_R(0)$ is a Lipschitz graph in some direction (see Lemma \ref{lem graph construction}), by standard estimates on the minimal surface equations (see for example \cite[Chapter 16]{GT} or \cite[Appendix C]{Giusti}) we obtain a uniform bound on the $C^\theta$ norm of its second fundamental form $A_{\alpha,\varepsilon}$.

As mentioned at the beginning of Section \ref{sec preliminary analysis}, all of these estimates hold uniformly for  $t\in[-1+b_1,1-b_1]$. This completes the proof of Theorem \ref{main result}.
\end{proof}

Next we show how   Corollary  \ref{main coro 1} and \ref{main coro 2} follow from Theorem \ref{main result}.
\begin{proof}[Proof of Corollary \ref{main coro 1}]
This follows the same reasoning in \cite[Section 7]{Wang-Wei2}. Here we include the proof for completeness.

First as in Lemma \ref{O(1) scale} or \cite[Lemma 7.1]{Wang-Wei2}, we deduce \eqref{close to 1D} from {\bf (H1)} and {\bf (H2)}. Thus for all $\varepsilon$ small, $\nabla u_\varepsilon\neq 0$ in $\{|u|<1-b\}$ and hence $|B(u_\varepsilon)|$ is well defined. In order to apply Theorem \ref{main result}, it suffices to establish
 a uniform bound on $|B(u_\varepsilon)|$ as in \eqref{C 1 1 bound}.

Now assume by the contrary, as $\varepsilon\to0$,
\begin{equation}\label{11.9}
\lim_{\varepsilon\to0}\max_{x\in \{|u_\varepsilon|<1-b\}\cap \mathcal{C}_{2/3}}|B(u_\varepsilon)(x)|=+\infty.
\end{equation}
Let $x_\varepsilon\in\mathcal{C}_1\cap\{|u_\varepsilon|\leq 1-b\}$ attain the following maxima (we denote $x=(x^\prime,x_n)$)
\begin{equation}
\max_{\mathcal{C}_1\cap\{|u_\varepsilon|\leq 1-b\}}\left(1-|x^\prime|\right)|B(u_\varepsilon)(x)|.
\end{equation}
By {\bf (H1)}, $x_\varepsilon\in \{|x_n|\leq 1/2\}$.

Denote
\begin{equation}\label{11.1}
L_\varepsilon:=|B(u_\varepsilon)(x_\varepsilon)|,\ \ \ \ \ \ \ r_\varepsilon:=\left(\frac{3}{2}-|x_\varepsilon^\prime|\right)/2.
\end{equation}
 Then by definition
\begin{equation}\label{11.2}
L_\varepsilon r_\varepsilon\geq\frac{1}{3}\sup_{\mathcal{C}_{2/3}\cap\{|u_\varepsilon|\leq 1-b\}}|B(u_\varepsilon)(x)|\to+\infty.
\end{equation}
In particular, $L_\varepsilon\to+\infty$. On the other hand, by \eqref{close to 1D}, we get
\begin{equation}\label{coarse bound}
L_\varepsilon=o\left(\frac{1}{\varepsilon}\right).
\end{equation}

By the choice of $r_\varepsilon$ at \eqref{11.1}, we have (here $\mathcal{C}_{r_\varepsilon}(x_\varepsilon^\prime):=B_{r_\varepsilon}^{n-1}(x_\varepsilon^\prime)\times(-1,1)$)
\begin{equation}\label{11.3}
\max_{x\in\mathcal{C}_{r_\varepsilon}(x_\varepsilon^\prime)\cap\{|u_\varepsilon|\leq 1-b\}}|B(u_\varepsilon)(x)|\leq 2L_\varepsilon.
\end{equation}

Let $\kappa:=L_\varepsilon\varepsilon$ and define $u_\kappa(x):=u_\varepsilon(x_\varepsilon+L_\varepsilon^{-1}x)$.
Then $u_\kappa$ satisfies \eqref{equation} with parameter $\kappa$ in $B_{L_\varepsilon r_\varepsilon}(0)$. By \eqref{coarse bound}, $\kappa\to0$ as $\varepsilon\to 0$. For any $t\in[-1+b,1-b]$, the level set $\{u_\kappa=t\}$ consists of $Q$ Lipschitz graphs
\begin{equation}\label{11.5}
\left\{x_n=f_{\beta,\kappa}^t(x^\prime):=L_\varepsilon\left[f_{\beta,\varepsilon}^t(x_\varepsilon^\prime+L_\varepsilon^{-1}x^\prime)-f_{\alpha,\varepsilon}^t(x_\varepsilon^\prime)\right]\right\},
\end{equation}
where $\alpha$ is chosen so that $x_\varepsilon$ lies in the connected component of $\{|u_\varepsilon|\leq 1-b\}$ containing $\Gamma_{\alpha,\varepsilon}$.

By \eqref{11.3}, we also have
\[
|B(u_\kappa)|\leq 2, \quad\mbox{for } x\in \mathcal{C}_1\cap\{|u_\kappa|\leq 1-b\}.
\]
Now Theorem \ref{main result} is applicable to $u_\kappa$. Hence $f_{\alpha,\kappa}$ are uniformly bounded in $C^{2,\theta}_{loc}(\R^{n-1})$. After passing to a subsequence, it converges to a limit $f_\infty$, which by \eqref{vansihing of mean curvature} is an entire solution of the minimal surface equation. Since the rescaling \eqref{11.5} preserves the Lipschitz constants, $f_\infty$ is global Lipschitz. Then by Moser's Liouville theorem on minimal surface equations (see \cite[Theorem 17.5]{Giusti}), $f_\infty$ is an affine function. In particular,
\begin{equation}\label{11.8}
  \nabla^2f_\infty\equiv 0.
\end{equation}

On the other hand, by the construction we have $|B(u_\kappa)(0)|=1$. If $n=2$,  as in the proof of \cite[Theorem 3.6]{Wang-Wei2}, we get
\[|B(u_\kappa)(0)|^2\lesssim \kappa^{\theta},\]
a contradiction  with \eqref{11.8}. If $n\geq 3$, we have
\[1=|B(u_\kappa)(0)|^2=|\nabla^2f_{\alpha,\kappa}(0)|^2+O\left(\kappa^{\theta}\right).\]
(The only difference here with the $n=2$ case is that now the Hessian of the distance function to $\Gamma_{\alpha,\kappa}$ does not converge to $0$, but its leading order term is exactly $\nabla^2f_{\alpha,\kappa}(0)$, see \eqref{error in z 1}.) This gives
\[\lim_{\kappa\to0}|\nabla^2f_{\alpha,\kappa}(0)|^2=1,\]
a contradiction with \eqref{11.8}. This contradiction implies that the assumption \eqref{11.9} cannot hold and the proof is thus complete.
\end{proof}

\begin{proof}[Proof of Corollary \ref{main coro 2}]
If $\delta_2$ is sufficiently small in \eqref{close to 1D}, by unique continuation principle $\nabla u_\varepsilon\neq 0$ in $\{|u_\varepsilon|\leq 1-b\}$ and hence $|B(u_\varepsilon)|$ is well defined. As in the proof of Corollary \ref{main coro 1}, the proof is  reduced to a uniform bound on $|B(u_\varepsilon)|$.

Assume by the contrary, we perform a similar blow up analysis as in the proof of Corollary \ref{main coro 1}. This gives another sequence of solutions $u_\kappa$ defined in an expanding domain. Moreover, $u_\kappa$ satisfies all of the assumptions in Theorem \ref{main result}. Hence the connected component of $\{u_\kappa=0\}$ passing through $0$ is a minimal hypersurface in $\R^n$, denoted by $\Sigma$. Its second fundamental form satisfies $|A_\Sigma|\leq 3$ and $|A_\Sigma(0)|=1$ (as in the proof of Corollary \ref{main coro 2}).

We claim that $\Sigma$ is stable. This then leads to a contradiction if \emph{Stable Bernstein conjecture} is true, which states that $\Sigma$ must be a hyperplane and hence $A_\Sigma\equiv 0$. The stability of $\Sigma$ follows from the general analysis in \cite{chodosh2018minimal}: first if there are at least two interfaces of $u_{\kappa}$ both converging to $\Sigma$, we can construct a positive Jacobi field on $\Sigma$ as in \cite[Theorem 4.1]{chodosh2018minimal}, which implies the stability of $\Sigma$; secondly, if there is only one such an interface, then there exist $\sigma>0$ and $C>0$ such that
\[\int_{B_\sigma(0)}\left[\frac{\kappa}{2}|\nabla u_\kappa|^2+\frac{1}{\kappa}W(u_\kappa)\right]\leq C.\]
Because $u_\kappa$ is stable, the stability of $\Sigma$ then follows by applying the main result in \cite{Tonegawa}.
\end{proof}

\bigskip

\appendix{}

\section{Some facts about the one dimensional solution}\label{sec 1d solution}
\setcounter{equation}{0}

In this appendix we recall some facts about one dimensional solution of \eqref{equation}, see \cite{Wang-Wei2} for more details.

It is known that the following identity holds for $g$,
\begin{equation}\label{first integral}
g^\prime(t)=\sqrt{2W(g(t))}>0, \quad \forall t\in\R.
\end{equation}
Moreover, as $t\to\pm\infty$, $g(t)$ converges exponentially to $\pm1$ and  the following quantity is well defined
\[\sigma_0:=\int_{-\infty}^{+\infty}\left[\frac{1}{2}g^\prime(t)^2+W(g(t))\right]dt\in (0,+\infty).\]

In fact,  as $t\to\pm\infty$, the following expansions hold. There exists a positive constant $A_{1}$ such that
for all $t>0$ large,
\[g(t)=1-A_1 e^{- t}+O(e^{-2  t}),\]
\[g^\prime(t)= A_1 e^{-  t}+O(e^{-2  t}),\]
\[g^{\prime\prime}(t)=- A_1 e^{-  t}+O(e^{-2  t}),\]
and a similar expansion holds as $t\to-\infty$ with $A_1$ replaced by another positive constant $A_{-1}$.

\medskip

The following result describes the interaction between two one dimensional profiles.
\begin{lem}\label{lem form of interaction}
For all $T>0$ large, we have the following expansion:
\[
\int_{-\infty}^{+\infty}\left[W^{\prime\prime}(g(t))-1\right]\left[g(-t-T)+1\right]g^\prime(t)dt=-2 A_{-1}^2e^{-T}+O\left(e^{-\frac{4}{3}T}\right).
\]
\[
\int_{-\infty}^{+\infty}\left[W^{\prime\prime}(g(t))-1\right]\left[g(T-t)-1\right]g^\prime(t)dt= 2 A_1^2e^{-T}+O\left(e^{-\frac{4 T}{3}}\right).
\]
\end{lem}

Next we discuss the spectrum  of the linearized operator at $g$,
\[\mathcal{L}=-\frac{d^2}{dt^2}+W^{\prime\prime}(g(t)).\]
 By a direct differentiation we see $g^\prime(t)$ is an eigenfunction of $\mathcal{L}$ corresponding to eigenvalue $0$.
By \eqref{first integral}, $0$ is the lowest eigenvalue. In other words, $g$ is stable.

Concerning the second eigenvalue, we have
\begin{thm}\label{second eigenvalue for 1d}
There exists a constant $\mu>0$ such that for any $\varphi\in H^1(\R)$ satisfying
\begin{equation}\label{orthogonal condition 1d}
\int_{-\infty}^{+\infty}\varphi(t)g^\prime(t)dt=0,
\end{equation}
we have
\[\int_{-\infty}^{+\infty}\left[\varphi^\prime(t)^2+W^{\prime\prime}(g(t))\varphi(t)^2\right]dt\geq\mu\int_{-\infty}^{+\infty}\varphi(t)^2dt.\]
\end{thm}
This can be proved via a contradiction argument.

\section{Proof of Lemma \ref{lem 5.1}}\label{sec proof of Lemma 5.1}
\setcounter{equation}{0}

The proof of Lemma \ref{lem 5.1} is similar to the one given in \cite[Appendix B]{Wang-Wei2}. However, since the setting is a little different (see Step 1 in Subsection \ref{subsec outline of proof}), for reader's convenience, we will include a complete proof.

Before proving Lemma \ref{lem 5.1}, we first derive the exponential nonlinearity in Toda system \eqref{Toda system}.
\begin{lem}\label{lem derivation of exponential nonlinearity}
  For any $y\in\Gamma_\alpha$,
  \[
    \int_{-\infty}^{+\infty}\mathcal{I}(y,z)g_\alpha^\prime(y,z)dz =  (-1)^\alpha\left[ 2A_{(-1)^{\alpha-1}}^2e^{-d_{\alpha-1}(y,0)}-2A_{(-1)^{\alpha}}^2e^{d_{\alpha+1}(y,0)}\right]+\mathcal{A}_\alpha(y),\]
  where
 \begin{eqnarray*}
 \|\mathcal{A}_\alpha \|_{C^{\theta}(B_1^\alpha(y))}
&\lesssim&\varepsilon^2+\varepsilon^{1/3}\max_{B_1^\alpha(y)}e^{-  D_\alpha}+
\max_{B_1^\alpha(y)}e^{-\frac{3}{2}D_\alpha}\\
&+&\max_{\beta: |d_\beta(y,0)|\leq 8|\log\varepsilon|}\max_{B_1^\beta(\Pi_\beta(y,0))}e^{-2D_\beta}+\max_{|z|<8|\log\varepsilon|} \|\phi\|_{C^{2,\theta}(B_1(y,z))}^2.
\end{eqnarray*}
\end{lem}
\begin{proof}

To determine the integral
$\int_{-\infty}^{+\infty}\mathcal{I}g_\alpha^\prime$,
 consider for each $\beta$, the integral on $(-\infty,+\infty)\cap\mathcal{M}_\beta^0$, which we assume to be an interval $(\rho_\beta^-(y),\rho_\beta^+(y))$.

{\bf Step 1.}
If $\beta\neq\alpha$, by Lemma \ref{upper bound on interaction}, in $(\rho_\beta^-(y),\rho_\beta^+(y))$,
\[
\big|\mathcal{I}\big|
\lesssim e^{-(|d_\beta|+|d_{\beta-1}|)}+e^{-(|d_\beta|+|d_{\beta+1}|)}+\varepsilon^2.\]
We only consider the case $\beta>\alpha$ and estimate
\[\int_{\rho_\beta^-(y)}^{\rho_\beta^+(y)}e^{-(|d_\beta|+|d_{\beta-1}|)}g_\alpha^\prime.\]

If $|z|$, $|d_\beta|$ and $|d_{\beta-1}|$ are all smaller than $8|\log\varepsilon|$ at the same time, by Lemma \ref{comparison of distances},
\begin{equation}\label{B1}
d_\beta(y,z)=z+d_\beta(y,0)+O\left(\varepsilon^{1/3}\right),
\end{equation}
\begin{equation}\label{B2}
d_{\beta-1}(y,z)=z+d_{\beta-1}(y,0)+O\left(\varepsilon^{1/3}\right).
\end{equation}
Note that since $\beta>\alpha$, by our convention on the sign of $d_\beta$, we have $z>0$ and $d_\beta(y,0)<d_{\beta-1}(y,0)\leq 0$.

By \eqref{B1} and \eqref{B2} we get
\begin{eqnarray*}
\int_{\rho_\beta^-(y)}^{\rho_\beta^+(y)}e^{-(|d_\beta|+|d_{\beta-1}|)}g_\alpha^\prime&\lesssim&\int_{\rho_\beta^-(y)}^{\rho_\beta^+(y)}e^{-\left(|z|+|z+d_{\beta-1}(y,0)|+|z+d_\beta(y,0)|\right)}\\
&\lesssim&\int_{\rho_\beta^-(y)}^{-d_\beta(y,0)}e^{-\left(z+d_{\beta-1}(y,0)-d_\beta(y,0)\right)}+\int_{-d_\beta(y,0)}^{\rho_\beta^+(y)}e^{-\left(3z+d_{\beta-1}(y,0)+d_\beta(y,0)\right)}\\
&\lesssim&e^{-\left(d_{\beta-1}(y,0)-d_\beta(y,0)\right)-\rho_\beta^-(y)}+e^{-\left(d_{\beta-1}(y,0)-2d_\beta(y,0)\right)}.
\end{eqnarray*}
By definition,
\[-d_\beta(y,\rho_\beta^-(y))=d_{\beta-1}(y,\rho_\beta^-(y)).\]
Thus by \eqref{B1} and \eqref{B2},
\[\rho_\beta^-(y)=-\frac{d_{\beta-1}(y,0)+d_\beta(y,0)}{2}+O\left(\varepsilon^{1/3}\right).\]
Substituting this into the above estimate gives
\[
\int_{\rho_\beta^-(y)}^{\rho_\beta^+(y)}e^{-(|d_\beta|+|d_{\beta-1}|)}g_\alpha^\prime \lesssim e^{-\frac{1}{2}\left(d_{\beta-1}(y,0)-3d_\beta(y,0)\right)}+e^{-\left(d_{\beta-1}(y,0)-2d_\beta(y,0)\right)}.
\]

If $\beta=\alpha+1$, because $d_{\beta-1}(y,0)=0$, the right hand side is bounded by $O\left(e^{\frac{3}{2}d_{\alpha+1}(y,0)}\right)$.

If $\beta\geq\alpha+2$, the right hand side is bounded by $O\left(e^{d_{\alpha+2}(y,0)}\right)$.

{\bf Step 2.} It remains to consider the integration in $(\rho_\alpha^-(y),\rho_\alpha^+(y))$. In this case we use Lemma \ref{interaction term}, which gives
\begin{eqnarray}\label{7.1}
&&\int_{\rho_\alpha^-(y)}^{\rho_\alpha^+(y)}\mathcal{I}g_\alpha^\prime \\
&=&\int_{\rho_\alpha^-(y)}^{\rho_\alpha^+(y)}\left[W^{\prime\prime}(g_\alpha)-1\right]\left[g_{\alpha-1}-(-1)^{\alpha-1}\right]g_\alpha^\prime+
\int_{\rho_\alpha^-(y)}^{\rho_\alpha^+(y)}\left[W^{\prime\prime}(g_\alpha)-1\right]\left[g_{\alpha+1}+(-1)^\alpha\right]g_\alpha^\prime\nonumber\\
&+&\int_{\rho_\alpha^-(y)}^{\rho_\alpha^+(y)}\left[O\left(e^{-2d_{\alpha-1}}+e^{2d_{\alpha+1}}\right)+O\left(e^{-d_{\alpha-2}-|z|}+e^{d_{\alpha+2}-|z|}\right)\right]g_\alpha^\prime.\nonumber
\end{eqnarray}

Because $g_\alpha^\prime\lesssim e^{-|z|}$ and
\[e^{-2d_{\alpha-1}}\lesssim e^{-2d_{\alpha-1}(y,0)-2z}+\varepsilon^2,\]
we get
\begin{eqnarray*}
\int_{\rho_\alpha^-(y)}^{\rho_\alpha^+(y)}e^{-2d_{\alpha-1}}g_\alpha^\prime
&\lesssim&\varepsilon^2+e^{-2d_{\alpha-1}(y,0)}\left[\int_{\rho_\alpha^-(y)}^0e^{-z}dz+ \int_0^{\rho_\alpha^+(y)}e^{-3z}dz\right]\\
&\lesssim&\varepsilon^2+e^{-2d_{\alpha-1}(y,0)-\rho_\alpha^-(y)}\\
&\lesssim&\varepsilon^2+e^{-\frac{3}{2}d_{\alpha-1}(y,0)}.
\end{eqnarray*}

Similarly, we have
\[\int_{\rho_\alpha^-(y)}^{\rho_\alpha^+(y)}e^{2d_{\alpha+1}}g_\alpha^\prime\lesssim\varepsilon^2+e^{\frac{3}{2}d_{\alpha+1}(y,0)},\]
\[\int_{\rho_\alpha^-(y)}^{\rho_\alpha^+(y)}O\left(e^{-d_{\alpha-2}-|z|}+e^{d_{\alpha+2}-|z|}\right)g_\alpha^\prime\lesssim e^{-d_{\alpha-2}}+e^{d_{\alpha+2}}.\]

To determine the first integral in the right hand side of \eqref{7.1}, arguing as in Step 1, if both $g_\alpha^\prime$ and $g_{\alpha-1}-(-1)^{\alpha-1}$ are nonzero, then
\[g_{\alpha-1}(y,z)=\bar{g}\left((-1)^{\alpha-1}\left(z+d_{\alpha-1}(y,0)+h_{\alpha-1}(\Pi_{\alpha-1}(y,z))+O\left(\varepsilon^{1/3}\right)\right)\right).\]
Therefore
\begin{eqnarray*}
&&\int_{\rho_\alpha^-(y)}^{\rho_\alpha^+(y)}\left[W^{\prime\prime}(g_\alpha)-1\right]\left(g_{\alpha-1}-(-1)^{\alpha-1}\right)g_\alpha^\prime\\
&=&\int_{\rho_\alpha^-(y)}^{\rho_\alpha^+(y)}\left[W^{\prime\prime}\left(\bar{g}\left((-1)^{\alpha}(z-h_\alpha(y))\right)\right)-1\right]
\bar{g}^\prime\left((-1)^{\alpha}(z-h_\alpha(y))\right)\\
&&\times\left[\bar{g}\left((-1)^{\alpha-1}\left(z+d_{\alpha-1}(y,0)+h_{\alpha-1}(\Pi_{\alpha-1}(y,z))+O\left(\varepsilon^{1/3}\right)\right)\right)-(-1)^{\alpha-1}\right]
dz\\
&=&\int_{-\infty}^{+\infty}\left[W^{\prime\prime}\left(\bar{g}\left((-1)^{\alpha}(z-h_\alpha(y))\right)\right)-1\right]\bar{g}^\prime\left((-1)^{\alpha }(z-h_\alpha(y))\right)\\
&&\times\left[\bar{g}\left((-1)^{\alpha-1}\left(z+d_{\alpha-1}(y,0)+h_{\alpha-1}(\Pi_{\alpha-1}(y,z))+O\left(\varepsilon^{1/3}\right)\right)\right)-(-1)^{\alpha-1}\right]
dz\\
&+&O\left(e^{-\frac{3}{2}d_{\alpha-1}(y,0)}\right) \\
&=&(-1)^\alpha 2A_{(-1)^{\alpha-1}}^2e^{-d_{\alpha-1}(y,0)}+O\left(|h_\alpha(y)|+|h_{\alpha-1}(\Pi_{\alpha-1}(y,z))|+\varepsilon^{1/3}\right)e^{-d_{\alpha-1}(y,0)}\\
&+&O\left(e^{-\frac{3}{2}d_{\alpha-1}(y,0)}\right) .
\end{eqnarray*}

{\bf Step 3.} What we have proven says
\begin{eqnarray*}
\Big|\mathcal{A}_\alpha(y)\Big|
&\lesssim&\varepsilon^2+\left(|h_\alpha(y)|+|h_{\alpha-1}(\Pi_{\alpha-1}(y,z))|+\varepsilon^{1/3}\right)e^{-d_{\alpha-1}(y,0)}\\
&+&\left(|h_\alpha(y)|+|h_{\alpha+1}(\Pi_{\alpha+1}(y,z))|+\varepsilon^{1/3}\right)e^{d_{\alpha+1}(y,0)}\\
&+&e^{-\frac{3}{2}d_{\alpha-1}(y,0)}+e^{\frac{3}{2}d_{\alpha+1}(y,0)}+e^{-d_{\alpha-2}(y,0)}
+e^{d_{\alpha+2}(y,0)}.
\end{eqnarray*}

By taking derivatives of $\int_{-\infty}^{+\infty}\mathcal{I}g_\alpha^\prime$ in $y$, and then using Lemma \ref{upper bound on interaction} and Lemma \ref{control on h_0},
we see the  $C^{\theta}(B_1^\alpha(y))$ norm (in fact, the Lipschitz norm) of $\mathcal{A}_\alpha$ is bounded by
\begin{eqnarray*}
&&\left(\varepsilon^2+\max_{B_1^\alpha(y)}e^{-  D_\alpha}\right)\left(\sum_{\beta : |d_\beta|<8|\log\varepsilon|}\|h_\beta\|_{C^{2,\theta}(B_1^\beta(\Pi_\beta(y,0)))}\right)\\
&\lesssim&\varepsilon^2+\max_{\beta: |d_\beta(y,0)|\leq 8|\log\varepsilon|}\max_{B_1^\beta(\Pi_\beta(y,0))}e^{-2D_\beta}+\max_{|z|<8|\log\varepsilon|} \|\phi\|_{C^{2,\theta}(B_1(y,z))}^2.
\end{eqnarray*}
Similarly, by Lemma \ref{comparison of distances}, the
 $C^{\theta}(B_1^\alpha(y))$ norm (in fact, the Lipschitz norm) of $2A_{(-1)^\alpha}^2e^{-d_{\alpha-1}(y,0)}-2A_{(-1)^{\alpha-1}}^2e^{d_{\alpha+1}(y,0)}$ is controlled by
$\varepsilon^{1/3}\max_{B_1^\alpha(y)}e^{-  D_\alpha}$. This gives the estimate on $\mathcal{A}_\alpha$.
\end{proof}

Now let us prove Lemma \ref{lem 5.1}.
Differentiating \eqref{orthogonal condition 2} twice leads to
\begin{equation}\label{orthogonal condition 2.1}
\int_{-\infty}^{+\infty}\left[\frac{\partial\phi}{\partial y^i}g_\alpha^\prime+(-1)^{\alpha-1}\phi
g_\alpha^{\prime\prime}\frac{\partial h_\alpha}{\partial y^i}\right]=0
\end{equation}
and
\begin{eqnarray}\label{orthogonal condition 2.2}
&&\int_{-\infty}^{+\infty}\left[\frac{\partial^2\phi}{\partial y^i\partial
y^j}g_\alpha^\prime+(-1)^{\alpha-1}\frac{\partial\phi}{\partial y^i}
g^{\prime\prime}_\alpha\frac{\partial h_\alpha}{\partial
y^j}+(-1)^{\alpha-1}\frac{\partial\phi}{\partial y^j}
g^{\prime\prime}_\alpha\frac{\partial h_\alpha}{\partial y^i}\right]\\
&+&\int_{-\infty}^{+\infty}\left[(-1)^{\alpha-1}\phi
g^{\prime\prime}_\alpha\frac{\partial^2 h_\alpha}{\partial y^i\partial
y^j}+\phi g^{\prime\prime\prime}_\alpha\frac{\partial h_\alpha}{\partial
y^i}\frac{\partial h_\alpha}{\partial y^j}\right]=0.\nonumber
\end{eqnarray}
Therefore
\begin{eqnarray}\label{orthogonal condition 2.3}
\int_{-\infty}^{+\infty}\Delta_{\alpha,0}\phi(y,z)g_\alpha^\prime&=&(-1)^{\alpha}\Delta_{\alpha,0}h_\alpha\int_{-\infty}^{+\infty}\phi
g_\alpha^{\prime\prime}-|\nabla_{\alpha,0}h_\alpha|^2\int_{-\infty}^{+\infty}\phi
g_\alpha^{\prime\prime\prime}\\
&+&2(-1)^{\alpha-1}\int_{-\infty}^{+\infty}
g_{\alpha}^{ij}(y,0)\frac{\partial\phi}{\partial y^i}\frac{\partial
h_\alpha}{\partial y^j}g_\alpha^{\prime\prime}. \nonumber
\end{eqnarray}

Substituting \eqref{orthogonal condition 2.3} into \eqref{H eqn}, we obtain
\begin{eqnarray*}
&&\int_{-\infty}^{+\infty}\left(\Delta_{\alpha,z}\phi-\Delta_{\alpha,0}\phi\right)g_\alpha^\prime +(-1)^{\alpha}\left(\int_{-\infty}^{+\infty}\phi g_\alpha^{\prime\prime}\right)\Delta_{\alpha,0}h_\alpha-\left(\int_{-\infty}^{+\infty}\phi g_\alpha^{\prime\prime\prime}\right)|\nabla_{\alpha,0} h_\alpha(y)|^2\\
&+&2(-1)^{\alpha-1}\int_{-\infty}^{+\infty}g_\alpha^{\prime\prime}g_{\alpha}^{ij}(y,0)\frac{\partial\phi}{\partial y_i}\frac{\partial h_\alpha}{\partial y_j}-\int_{-\infty}^{+\infty}H_\alpha(y,z)g_\alpha^\prime\phi_z+\int_{-\infty}^{+\infty}\xi_\alpha^\prime\phi\\
&=&\int_{-\infty}^{+\infty}\left[W^{\prime\prime}(g_\ast)-W^{\prime\prime}(g_\alpha)\right]g_\alpha^\prime\phi+\int_{-\infty}^{+\infty}\mathcal{R}(\phi)g_\alpha^\prime\\
&+&\int_{-\infty}^{+\infty}\mathcal{I}g_\alpha^\prime+(-1)^{\alpha}\left(\int_{-\infty}^{+\infty}|g_\alpha^\prime|^2\right)\left[H_\alpha(y,0)+\Delta_{\alpha,0}h_\alpha(y)\right]\\
&+&(-1)^{\alpha }\int_{-\infty}^{+\infty}|g_\alpha^\prime|^2\left[H_\alpha(y,z)-H_\alpha(y,0)\right]+(-1)^{\alpha }\int_{-\infty}^{+\infty}|g_\alpha^\prime|^2\left[\Delta_{\alpha,z}h_\alpha(y)-\Delta_{\alpha,0}h_\alpha(y)\right]\\
&+&\frac{1}{2}\left(\int_{-\infty}^{+\infty}|g_\alpha^\prime|^2\frac{\partial}{\partial z}g_{\alpha}^{ij}(y,z)\right)\frac{\partial h_\alpha}{\partial y_i}\frac{\partial h_\alpha}{\partial y_j}\\
&+&\sum_{\beta\neq\alpha}(-1)^{\beta }\int_{-\infty}^{+\infty}g_\alpha^\prime g_\beta^\prime\mathcal{R}_{\beta,1}-\sum_{\beta\neq\alpha}\int_{-\infty}^{+\infty}g_\alpha^\prime g_\beta^{\prime\prime}\mathcal{R}_{\beta,2}-\sum_{\beta }\int_{-\infty}^{+\infty}g_\alpha^\prime \xi_\beta.
\end{eqnarray*}

We estimate the H\"{o}lder norm of these terms one by one.

\begin{enumerate}
\item By \eqref{error in z 5}, we have
 \begin{eqnarray*}
\Big\|\int_{-\infty}^{+\infty}\left(\Delta_{\alpha,z}\phi-\Delta_{\alpha,0}\phi\right)g_\alpha^\prime\Big\|_{C^{\theta}(B_1^\alpha(y))}
&\lesssim&\varepsilon\max_{|z|<8|\log\varepsilon|} \|\phi\|_{C^{2,\theta}(B_1(y,z))}\\
&\lesssim&\varepsilon^2+\max_{|z|<8|\log\varepsilon|} \|\phi\|_{C^{2,\theta}(B_1(y,z))}^2.
\end{eqnarray*}

\item
By the exponential decay of $\bar{g}^\prime$ and Lemma \ref{control on h_0}, we have
\begin{eqnarray*}
\Big\|\left(\int_{-\infty}^{+\infty}\phi g_\alpha^{\prime\prime}\right)\Delta_{\alpha,0}h_\alpha\Big\|_{C^{\theta}(B_1^\alpha(y))}&\lesssim&\|h_\alpha\|_{C^{2,\theta}(B_1^\alpha(y))}
\max_{|z|<8|\log\varepsilon|} \|\phi\|_{C^{\theta}(B_1(y,z))}\\
&\lesssim&\|h_\alpha\|_{C^{2,\theta}(B_1^\alpha(y))}^2+\max_{|z|<8|\log\varepsilon|} \|\phi\|_{C^{\theta}(B_1(y,z))}^2\\
&\lesssim&\max_{B_1^\alpha(y)}e^{-2D_\alpha}+\max_{|z|<8|\log\varepsilon|} \|\phi\|_{C^{2,\theta}(B_1(y,z))}^2.
\end{eqnarray*}

\item
By the exponential decay of $\bar{g}^\prime$ and Lemma \ref{control on h_0}, we have
 \begin{eqnarray*}
\Big\|\int_{-\infty}^{+\infty}g_\alpha^{\prime\prime}g_{\alpha}^{ij}(y,0)\frac{\partial\phi}{\partial y_i}\frac{\partial h_\alpha}{\partial y_j}\Big\|_{C^{\theta}(B_1^\alpha(y))}
&\lesssim&\|h_\alpha\|_{C^{1,\theta}(B_1^\alpha(y))}\max_{|z|<8|\log\varepsilon|} \|\phi\|_{C^{1,\theta}(B_1(y,z))}\\
&\lesssim&\|h_\alpha\|_{C^{1,\theta}(B_1^\alpha(y))}^2+\max_{|z|<8|\log\varepsilon|} \|\phi\|_{C^{1,\theta}(B_1(y,z))}^2\\
&\lesssim&\max_{B_1^\alpha(y)}e^{-2D_\alpha}+\max_{|z|<8|\log\varepsilon|} \|\phi\|_{C^{2,\theta}(B_1(y,z))}^2.
\end{eqnarray*}

\item By the exponential decay of $\bar{g}^\prime$ and Lemma \ref{control on h_0}, we have
\begin{eqnarray*}
\Big\|\left(\int_{-\infty}^{+\infty}\phi g_\alpha^{\prime\prime\prime}\right)|\nabla_{\alpha,0} h_\alpha(y)|^2\Big\|_{C^{\theta}(B_1^\alpha(y))}&\lesssim&\|h_\alpha\|_{C^{1,\theta}(B_1^\alpha(y))}^2\max_{|z|<8|\log\varepsilon|} \|\phi\|_{C^{\theta}(B_1(y,z))}\\
&\lesssim&\|h_\alpha\|_{C^{1,\theta}(B_1^\alpha(y))}^2\\
&\lesssim&\max_{B_1^\alpha(y)}e^{-2D_\alpha}+\|\phi\|_{C^{2,\theta}(B_1(y,0))}^2.
\end{eqnarray*}

\item By \eqref{bound on 3rd derivatives} and the exponential decay of $\bar{g}^\prime$, we have
\begin{eqnarray*}
\Big\|\int_{-\infty}^{+\infty}H_\alpha(y,z)g_\alpha^\prime\phi_z\Big\|_{C^{\theta}(B_1^\alpha(y))}&\lesssim&\varepsilon\max_{|z|<8|\log\varepsilon|} \|\phi\|_{C^{2,\theta}(B_1(y,z))}\\
&\lesssim&\varepsilon^2+\max_{|z|<8|\log\varepsilon|} \|\phi\|_{C^{2,\theta}(B_1(y,z))}^2.
\end{eqnarray*}

\item
The $C^{\theta}(B_1^\alpha(y))$ norm of $\int_{-\infty}^{+\infty}\left[W^{\prime\prime}(g_\ast)-W^{\prime\prime}(g_\alpha)\right]g_\alpha^\prime\phi$ is bounded by
\begin{eqnarray*}
&&\left[\max_{|z|<8|\log\varepsilon|} \|\phi\|_{C^{\theta}(B_1(y,z))}\right]\left[\max_{B_1^\alpha(y)}\left(\int_{-\infty}^{+\infty}\left(|g_{\alpha-1}^2-1|+|g_{\alpha+1}^2-1|\right)g_\alpha^\prime\right)\right]\\
&\lesssim&\left(\max_{|z|<8|\log\varepsilon|} \|\phi\|_{C^{\theta}(B_1(y,z))}\right)\left(\max_{B_1(y)}D_\alpha e^{-D_\alpha}\right)\\
&\lesssim&\max_{|z|<8|\log\varepsilon|} \|\phi\|_{C^{2,\theta}(B_1(y,z))}^2+\max_{B_1^\alpha(y)}e^{-\frac{3}{2}D_\alpha}.
\end{eqnarray*}

\item The $C^{\theta}(B_1^\alpha(y))$ norm of $\int_{-\infty}^{+\infty}\mathcal{R}(\phi)g_\alpha^\prime$ is bounded by
$\max_{|z|<8|\log\varepsilon|}\|\phi\|_{C^{\theta}(B_1(y,z))}^2$.

\item By the definition of $\bar{g}$ (see Subsection \ref{sec optimal approximation}), the $C^{\theta}(B_1^\alpha(y))$ norm of $\int_{-\infty}^{+\infty}|g_\alpha^\prime|^2-\sigma_0$ is bounded by $O\left(\varepsilon^2\right)$.

\item By \eqref{bound on 3rd derivatives} and \eqref{error in z 4}, the $C^{\theta}(B_1^\alpha(y))$ norm of $\int_{-\infty}^{+\infty}|g_\alpha^\prime|^2\left[H_\alpha(y,z)-H_\alpha(y,0)\right]$ is bounded by $O\left(\varepsilon^2\right)$.

\item
By \eqref{error in z 5} and Lemma \ref{control on h_0}, the $C^{\theta}(B_1^\alpha(y))$ norm of $\int_{-\infty}^{+\infty}|g_\alpha^\prime|^2\left[\Delta_{\alpha,0}h_\alpha(y)-\Delta_{\alpha,z}h_\alpha(y)\right]$ is bounded by
\[
\varepsilon\|h_\alpha\|_{C^{2,\theta}(B_1^\alpha(y))}
\lesssim \varepsilon^2+\|h_\alpha\|_{C^{2,\theta}(B_1^\alpha(y))}^2\\
\lesssim \varepsilon^2+\max_{B_1^\alpha(y)}e^{-2D_\alpha}+\|\phi\|_{C^{2,\theta}(B_1(y,0))}^2.\]

\item By \eqref{metirc tensor}, the $C^{\theta}(B_1^\alpha(y))$ norm of $\left(\int_{-\infty}^{+\infty}|g_\alpha^\prime|^2\frac{\partial}{\partial z}g_{\alpha}^{ij}(y,z)\right)\frac{\partial h_\alpha}{\partial y_i}\frac{\partial h_\alpha}{\partial y_j}$ is bounded by
\[
\varepsilon\|h_\alpha\|_{C^{1,\theta}(B_1^\alpha(y))}^2\lesssim\max_{B_1^\alpha(y)}e^{-2D_\alpha}+\|\phi\|_{C^{2,\theta}(B_1(y,0))}^2.\]

\item
For $\beta\neq\alpha$, if $|d_\beta(y,0)|>8|\log\varepsilon|$, the $C^{\theta}(B_1^\alpha(y))$ norm of $\int_{-\infty}^{+\infty}g_\alpha^\prime g_\beta^\prime\mathcal{R}_{\beta,1}$ is bounded by $O\left(\varepsilon^2\right)$.

If $|d_\beta(y,0)|\leq 8|\log\varepsilon|$,   first note that in Fermi coordinates with respect to $\Gamma_\beta$, we have the decomposition
\begin{eqnarray*}
 g_\alpha^\prime g_\beta^\prime\mathcal{R}_{\beta,1}&=& \underbrace{g_\alpha^\prime g_\beta^\prime\left[H_\beta(y,0)+\Delta_{\beta,0}h_\beta(y)\right]}_I+\underbrace{ g_\alpha^\prime g_\beta^\prime \left[H_\beta(y,z)-H_\beta(y,0)\right]}_{II}\\
 &+&\underbrace{ g_\alpha^\prime g_\beta^\prime \left[\Delta_{\beta,z}h_\beta(y)-\Delta_{\beta,0}h_\beta(y)\right]}_{III}.
\end{eqnarray*}
These three terms are estimated in the following way. First we have
\[\|I\|_{C^\theta(B_1(y,z))}\lesssim e^{-|d_\alpha(y,z)|-|z|}\|H_\beta+\Delta_{\beta,0}h_\beta\|_{C^\theta(B_1^\beta(y))}.\]
By  \eqref{error in z 4}, we get
\[\|II\|_{C^\theta(B_1 (y,z))}\lesssim \varepsilon^2 |z| e^{-|d_\alpha(y,z)|-|z|},\]
and by \eqref{error in z 5}, we get
\[\|III\|_{C^\theta(B_1^\beta(y))}\lesssim \varepsilon |z| e^{-|d_\alpha(y,z)|-|z|}\|h_\beta\|_{C^{2,\theta}(B_1^\beta(\Pi_\beta(y,0)))}.\]

Putting these estimates together and coming back to Fermi coordinates with respect to $\Gamma_\alpha$, applying Lemma \ref{comparison of distances} to change distances to be measured with respect to $\Gamma_\alpha$, we see the $C^{\theta}(B_1^\alpha(y))$ norm of $\int_{-\infty}^{+\infty}g_\alpha^\prime g_\beta^\prime\mathcal{R}_{\beta,1}$ is controlled by
\begin{eqnarray}\label{B.1}
&&\varepsilon^2+|d_\beta(y,0)|e^{-|d_\beta(y,0)|}\|H_\beta+\Delta_{\beta,0}h_\beta\|_{C^\theta(B_1^\beta(\Pi_\beta(y,0)))}+\varepsilon |d_\beta(y,0)|^2e^{-|d_\beta(y,0)|}\\
&\lesssim&\varepsilon^2+e^{-\frac{3}{2}|d_\beta(y,0)|}+e^{-\frac{3}{4}|d_\beta(y,0)|}\|H_\beta+\Delta_{\beta,0}h_\beta\|_{C^\theta(B_1^\beta(\Pi_\beta(y,0)))}.
\nonumber
\end{eqnarray}
Summing in $\beta\neq\alpha$ and applying Lemma \ref{lem distance ladder} we obtain
\begin{eqnarray}\label{B.2}
&&\Big\|\sum_{\beta\neq\alpha}\int_{-\infty}^{+\infty}g_\alpha^\prime g_\beta^\prime\mathcal{R}_{\beta,1}\Big\|_{C^{\theta}(B_1^\alpha(y))}\\
&\lesssim&\varepsilon^2+\max_{B_1^\alpha(y)}e^{-\frac{3}{2}D_\alpha}+\max_{\beta\neq\alpha: |d_\beta(y,0)|\leq 8|\log\varepsilon|}\|H_\beta+\Delta_{\beta,0}h_\beta\|_{C^\theta(B_1^\beta(\Pi_\beta(y,0)))}^2. \nonumber
\end{eqnarray}

\item By the same reasoning as in the previous case, for $\beta\neq\alpha$, the $C^{\theta}(B_1^\alpha(y))$ norm of $\sum_{\beta\neq\alpha}\int_{-\infty}^{+\infty}g_\alpha^\prime g_\beta^{\prime\prime}\mathcal{R}_{\beta,2}$ is controlled by
\begin{eqnarray*}
&&\varepsilon^2+\max_{|z|<8|\log\varepsilon|} \|\phi\|_{C^{2,\theta}(B_1(y,z))}^2+\sum_{\beta\neq\alpha: |d_\beta(y,0)|\leq 8|\log\varepsilon|}|d_\beta(y,0)|e^{-\left(|d_\beta(y,0)|+2D_\beta(\Pi_\beta(y,0))\right)}\\
&\lesssim&\varepsilon^2+\max_{\beta: |d_\beta(y,0)|\leq 8|\log\varepsilon|}\max_{B_1^\beta(\Pi_\beta(y,0))}e^{-\frac
{3}{2}D_\beta}+\max_{|z|<8|\log\varepsilon|}\|\phi\|_{C^{2,\theta}(B_1(y,z))}^2.
\end{eqnarray*}

\item By the definition of $\xi$, the $C^{\theta}(B_1^\alpha(y))$ norm of $\sum_{\beta }(-1)^{\beta-1}\int_{-\infty}^{+\infty}g_\alpha^\prime \xi_\beta$ is bounded by
 $O\left(\varepsilon^2\right)$.

\end{enumerate}

\medskip
Combining all of these estimates we get \eqref{5.1}.

\section{Proof of Lemma \ref{Holder for RHS}}\label{sec proof of Lem 9.1}
\setcounter{equation}{0}

We estimate the H\"{o}lder norm of the right hand side of \eqref{error equation} term by term. Since they reproduce similar patterns on each $\mathcal{M}_\alpha^0$, it is sufficient to consider  one of such domains.

\begin{enumerate}

\item Because
\[\mathcal{R}(\phi)=W^\prime(g_\ast+\phi)-W^\prime(g_\ast)-W^{\prime\prime}(g_\ast)\phi,\]
we get
\[\|\mathcal{R}(\phi)\|_{C^{\theta}(B_r(x))}\lesssim\|\phi\|_{C^{\theta}(B_r(x))}^2.\]

\item By Lemma \ref{Holder bound on interaction}, we have
\[\Big\|W^\prime(g_\ast)-\sum_\beta W^\prime(g_\beta)\Big\|_{C^{\theta}(\mathcal{M}_\alpha^0\cap B_r(x))}\lesssim \varepsilon^2+A\left(r;x\right) .\]

\item Take the decomposition
\begin{eqnarray*}
    g_\alpha^\prime\left[H_\alpha(y,z)+\Delta_zh_\alpha(y)\right]&=&g_\alpha^\prime\left[H_\alpha(y,0)+\Delta_0h_\alpha(y)\right]\\
    &+&g_\alpha^\prime\left[H_\alpha(y,z)-H_\alpha(y,0)\right]+g_\alpha^\prime\left[\Delta_{\alpha,z}h_\alpha(y)-\Delta_{\alpha,0}h_\alpha(y)\right].
    \end{eqnarray*}
First we have
\begin{eqnarray*}
&&\|g_\alpha^\prime(y,z)\left[H_\alpha(y,z)-H_\alpha(y,0)\right]\|_{C^{\theta}(\mathcal{M}_\alpha^0\cap B_r(x))}\\
&\lesssim&\|g_\alpha^\prime(y,z)\left[H_\alpha(y,z)-H_\alpha(y,0)\right]\|_{Lip(\mathcal{M}_\alpha^0\cap B_r(x))}\\
&\lesssim& \max_{(y,z)\in \mathcal{M}_\alpha^0\cap B_r(x)} e^{-|z|}\big|H_\alpha(y,z)-H_\alpha(y,0)\big|\\
&&+ \max_{(y,z)\in \mathcal{M}_\alpha^0\cap B_r(x)} |z||e^{-|z|}|\big|A_\alpha(y,0)||\nabla_{\alpha,0} A_\alpha(y,0)\big|\\
&\lesssim&\varepsilon^2.
\end{eqnarray*}

Next because
\[\Delta_{\alpha,z}h_\alpha(y)-\Delta_{\alpha,0}h_\alpha(y)=\sum_{i,j=1}^{n-1}\left[g_{\alpha}^{ij}(y,z)-g_{\alpha}^{ij}(y,0)\right]\frac{\partial^2h_\alpha}{\partial y_i\partial y_j}+\sum_{i=1}^{n-1}\left[b_\alpha^i(y,z)-b_\alpha^i(y,0)\right]\frac{\partial h_\alpha}{\partial y_i},\]
we get
\begin{eqnarray*}
&&\big\|g_\alpha^\prime\left[\Delta_{\alpha,0}h_\alpha(y)-\Delta_{\alpha,z}h_\alpha(y)\right]\big\|_{C^{\theta}(\mathcal{M}_\alpha^0\cap B_r(x))}\\
&\lesssim& \max_{(y,z)\in\mathcal{M}_\alpha^0\cap B_r(x)} e^{-|z|}\left[|g_{\alpha}^{ij}(y,z)-g_{\alpha}^{ij}(y,0)||\nabla_{\alpha,0}^2h_\alpha(y)|+|b^i(y,z)-b^i(y,0)||\nabla_{\alpha,0} h_\alpha(y)|\right]\\
&+& \max_{(y,z)\in\mathcal{M}_\alpha^0\cap B_r(x)}  e^{-|z|}\left(|\nabla_{\alpha,0}^2h_\alpha(y)|+|\nabla_{\alpha,0} h_\alpha(y)|\right)\\
&&\times\left(\big\|g_{\alpha}^{ij}(\tilde{y},\tilde{z})-g_{\alpha}^{ij}(\tilde{y},0)\big\|_{C^{\theta}(B_1(y,z))}+\big\|b_\alpha^i(\tilde{y},\tilde{z})
-b_\alpha^i(\tilde{y},0)\big\|_{C^{\theta}(B_1(y,z))}\right)\\
&+&\|h_\alpha\|_{C^{2,\theta}(\Gamma_\alpha\cap B_r(x))}\left(\max_{(y,z)\in \mathcal{M}_\alpha^0\cap B_r(x)} e^{-|z|}\left(|g_{\alpha}^{ij}(y,z)-g_{\alpha}^{ij}(y,0)|+|b_\alpha^i(y,z)-b_\alpha^i(y,0)|\right)\right)\\
&\lesssim&\varepsilon\|h_\alpha\|_{C^{2,\theta}(\Gamma_\alpha\cap B_{r }(x))}\\
&\lesssim&\varepsilon^2+\|h_\alpha\|_{C^{2,\theta}(\Gamma_\alpha\cap B_{r }(x))}^2  \quad \quad \mbox{(by Cauchy inequality)}\\
&\lesssim&\varepsilon^2+\|\phi\|_{C^{2,\theta}(B_{r+8|\log\varepsilon|}(x))}^2+\max_\alpha\max_{\Gamma_\alpha\cap B_{r+8|\log\varepsilon|}(x)}e^{-2D_\alpha}. \quad \quad \mbox{(by Lemma \ref{control on h_0})}
\end{eqnarray*}

\item By Lemma \ref{control on h_0}, we have
\begin{eqnarray*}
&&\|g_\alpha^{\prime\prime}|\nabla_{\alpha,z}h_\alpha|^2\|_{C^{\theta}(\mathcal{M}_\alpha^0\cap B_r(x))}\\
&\lesssim &\|\nabla_{\alpha,0} h_\alpha\|_{L^\infty(\Gamma_\alpha\cap B_r(x))}^2+\|\nabla_{\alpha,0} h_\alpha\|_{L^\infty(\Gamma_\alpha\cap B_r(x))}\|\nabla_{\alpha,0}^2 h_\alpha\|_{L^\infty(\Gamma_\alpha\cap B_r(x))}\\
&\lesssim& \|\phi\|_{C^{2,\theta}(B_{r+8|\log\varepsilon|}(x))}^2+\max_\alpha\max_{\Gamma_\alpha\cap B_{r+8|\log\varepsilon|}(x)}e^{-2D_\alpha}.
\end{eqnarray*}

\item As in the previous case, we first estimate the H\"{o}lder norm of $\mathcal{R}_{\beta,1}$ in Fermi coordinates with respect to $\Gamma_\beta$ for each $\beta\neq\alpha$. Coming back to Fermi coordinates with respect to $\Gamma_\alpha$ and noting that if $g_\beta^\prime\neq 0$, then $|d_{\beta}(y,z)|<8|\log\varepsilon|$, we obtain
\begin{eqnarray*}
&&\|g_\beta^\prime\mathcal{R}_{\beta,1}\|_{C^{\theta}(\mathcal{M}_\alpha^0\cap B_r(x))}\\
&\lesssim& \left(\max_{(y,z)\in \mathcal{M}_\alpha^0\cap B_r(x)}e^{-|d_\beta(y,z)|}\right)\\
&&\times\left(\varepsilon^2+\|\phi\|_{C^{2,\theta}(B_{r+9|\log\varepsilon|}(x))}^2+\max_\alpha\max_{\Gamma_\alpha\cap B_{r+9|\log\varepsilon|}(x)}e^{-2D_\alpha}\right)\\
&+&\max_{(y,z)\in \mathcal{M}_\alpha^0\cap B_r(x), |d_\beta(y,z)|\leq 8|\log\varepsilon|}e^{-|d_\beta(y,z)|}\|H_\beta+\Delta_{\beta,0}h_\beta\|_{C^\theta(B_2^\beta(\Pi_\beta(y,0)))}.
\end{eqnarray*}
Then using Lemma \ref{lem distance ladder} and summing in $\beta$, we get
\begin{eqnarray*}
&&\big\|\sum_{\beta\neq\alpha}g_\beta^\prime\mathcal{R}_{\beta,1}\big\|_{C^{\theta}(\mathcal{M}_\alpha^0\cap B_r(x))}\\
&\lesssim& \varepsilon^2+\|\phi\|_{C^{2,\theta}(B_{r+9|\log\varepsilon|}(x))}^2+\max_\alpha\max_{\Gamma_\alpha\cap B_{r+9|\log\varepsilon|}(x)}e^{-2D_\alpha}\\
&+&\left(\max_\alpha\max_{\Gamma_\alpha\cap B_{r+9|\log\varepsilon|}(x)}e^{-\frac{D_\alpha }{2}}\right)
\left(\max_\alpha\|H_\alpha+\Delta_{\alpha,0}h_\alpha\|_{C^{\theta}(B_{r+9|\log\varepsilon|}(x))}\right).
\end{eqnarray*}

\item Similar to the previous case, we have
\begin{eqnarray*}
&&\sum_{\beta\neq\alpha}\|g_\beta^\prime\mathcal{R}_{\beta,2}\|_{C^{\theta}(\mathcal{M}_\alpha^0\cap B_r(x))}\\
&\lesssim &\sum_{\beta\neq\alpha}\max_{(y,z)\in \mathcal{M}_\alpha^0\cap B_r(x)}e^{-|d_\beta(y,z)|}\\
&&\times\left(\|\nabla_{\beta,0} h_\beta\|_{L^\infty(\Gamma_\beta\cap B_{r+8|\log\varepsilon|}(x)}^2+\|\nabla_{\beta,0} h_\beta\|_{L^\infty(\Gamma_\beta\cap B_{r+8|\log\varepsilon|}(x))}\|\nabla_{\beta,0}^2 h_\alpha\|_{L^\infty(\Gamma_\beta\cap B_{r+8|\log\varepsilon|}(x))}\right)\\
&\lesssim&\sum_{\beta\neq\alpha} \max_{(y,z)\in \mathcal{M}_\alpha^0\cap B_r(x)}e^{-|d_\beta(y,z)|}\left(\varepsilon^2+\|\phi\|_{C^{2,\theta}(B_{r+9|\log\varepsilon|}(x))}^2+\max_\alpha\max_{\Gamma_\alpha\cap B_{r+9|\log\varepsilon|}(x)}e^{-2D_\alpha}\right)\\
&\lesssim&  \varepsilon^2+\|\phi\|_{C^{2,\theta}(B_{r+9|\log\varepsilon|}(x))}^2+\max_\alpha\max_{\Gamma_\alpha\cap B_{r+9|\log\varepsilon|}(x)}e^{-2D_\alpha}.
\end{eqnarray*}

\item For any $\beta$, $\xi_\beta(y,z)\neq0$ only if $|d_\beta(y,z)|\leq 8|\log\varepsilon|$. Hence by Lemma \ref{O(1) scale},
\[\Big\|\sum_\beta\xi_\beta\Big\|_{C^{\theta}(\mathcal{M}_\alpha^0\cap B_r(x))}\lesssim \varepsilon^3|\log\varepsilon|\lesssim \varepsilon^2.\]

\end{enumerate}

\medskip

Putting these estimates together we finish the proof of Lemma \ref{Holder for RHS}.

\end{document}